\documentclass[graybox]{svmult}

\usepackage{mathptmx}       
\usepackage{helvet}         
\usepackage{courier}        
\usepackage{type1cm}        
\usepackage{makeidx}         
\usepackage{graphicx}        
\usepackage{multicol}        
\usepackage[bottom]{footmisc}

\usepackage{amssymb}
\usepackage{amsmath}
\usepackage{bm}
\usepackage[utf8]{inputenc}
\usepackage{graphicx}
\usepackage{subfigure}
\usepackage{amsfonts}
\usepackage{amssymb}
\usepackage{color}
\usepackage{cancel}
\usepackage{sidecap}
\usepackage{pdfcomment}
\usepackage{tikz}
\usetikzlibrary{arrows}
\usepackage{wrapfig}
\usepackage[exponent-product = \cdot]{siunitx}

\usepackage{enumitem}
\usepackage{wrapfig}
\usepackage{relsize}
\usepackage{xspace}
\usepackage{amsmath}
\usepackage{bbm}
\usepackage{mathtools}
\usepackage{todonotes}

\newtheorem{lem}{Lemma}
\newtheorem{ass}{Assumption}

\DeclareMathOperator{\rank}{rank}
\DeclareMathOperator{\ima}{im}

\makeindex             

\newtheorem{benchmark}{Benchmark}

\RequirePackage[normalem]{ulem} 
\RequirePackage{color}\definecolor{RED}{rgb}{1,0,0}\definecolor{BLUE}{rgb}{0,0,1}

\catcode`@=11
\def\frownfill{$\scriptscriptstyle\m@th\mathord\frown$}
\def\bow#1{\vbox{\m@th\ialign{##\crcr
			\hfil\frownfill\hfil\crcr\noalign{\kern-0.2\p@\nointerlineskip}
			$\hfil\displaystyle{#1}\hfil$\crcr}}}
\def\bbow#1{\vbox{\m@th\ialign{##\crcr
			\hfil\frownfill\hfil\crcr\noalign{\kern-0.7\p@\nointerlineskip}
			\hfil\frownfill\hfil\crcr\noalign{\kern-0.3\p@\nointerlineskip}
			$\hfil\displaystyle{#1}\hfil$\crcr}}}
\def\widefrownfill{$\m@th\mathord\frown$}
\def\widebow#1{\vbox{\m@th\ialign{##\crcr
			\hfil\widefrownfill\hfil\crcr\noalign{\kern-0.9\p@\nointerlineskip}
			$\hfil\displaystyle{#1}\hfil$\crcr}}}
\def\widebbow#1{\vbox{\m@th\ialign{##\crcr
			\hfil\widefrownfill\hfil\crcr\noalign{\kern-1.8\p@\nointerlineskip}
			\hfil\widefrownfill\hfil\crcr\noalign{\kern-0.9\p@\nointerlineskip}
			$\hfil\displaystyle{#1}\hfil$\crcr}}}

\graphicspath{{img/}}

\begin{document}

\title{Systems of Differential Algebraic Equations in Computational Electromagnetics}
\author{Idoia  Cortes Garcia, Sebastian Schöps, Herbert De Gersem and Sascha Baumanns}
\institute{Idoia Cortes Garcia, Sebastian Schöps and Herbert De Gersem \at Technische Universität Darmstadt, Graduate School of Computational Engineering
Dolivostraße 15, 64293 Darmstadt, Germany, \email{cortes@gsc.tu-darmstadt.de}, \email{schoeps@gsc.tu-darmstadt.de}, \email{degersem@temf.tu-darmstadt.de}
\and Sascha Baumanns \at Universität zu Köln, Mathematisches Institut, Weyertal 86-90, 50931 Köln, Germany, \email{sbaumanns@math.uni-koeln.de}}
\maketitle

\abstract{Starting from space-discretisation of Maxwell's equations, various classical formulations are proposed for the simulation of 
electromagnetic fields. They differ in the phenomena considered as well as in the variables chosen for discretisation. This contribution presents a 
literature survey of the most common approximations and formulations with a focus on their structural properties. The differential-algebraic 
character is discussed and quantified by the differential index concept.}

\section{Introduction}

\label{sect:introduction}
Electromagnetic theory has been established by Maxwell in 1864 and was reformulated into the language of vector calculus by Heavyside in 1891 
\cite{Maxwell_1864aa,Heaviside_1891aa}. A historical overview can be found in the  review article \cite{Rautio_2014aa}. The theory is well understood 
and rigorously presented in many text books, e.g. \cite{Haus_1989aa}, \cite{Jackson_1998aa}, \cite{Griffiths_1999aa}. More recently researchers have 
begun to formulate the equations in terms of exterior calculus and differential forms which expresses the relations more elegantly and metric-free, 
e.g.~\cite{Hehl_2003aa}.

The simulation of three-dimensional spatially distributed electromagnetic phenomena based on Maxwell's equations is roughly 50 years old. An early 
key contribution was the proposition of the finite difference time domain method (FDTD) by Yee to solve the high-frequency hyperbolic problem on 
equidistant grids in 1966 \cite{Yee_1966aa}, and its subsequent generalisations and improvements, e.g. \cite{Taflove_1995aa}. Among the most 
interesting generalisations are the Finite Integration Technique \cite{Weiland_1977aa} and the Cell Method \cite{Alotto_2006aa} because they can be 
considered as discrete differential forms \cite{Bossavit_2000ab}. Most finite-difference codes formulate the problem in terms of the electric and 
magnetic field strength and yield ordinary differential equations after space discretisation which are solved explicitly in time. FDTD is very robust 
and remarkable efficient \cite{Monk_1994aa} and is considered to be among the `top rank of computational tools for engineers and scientists studying 
electrodynamic phenomena and systems' \cite{Taflove_2007aa}.

Around the same time at which FIT was proposed, circuit simulation programs became popular, e.g. \cite{Weeks_1973aa,Nagel_1973aa} and Albert Ruehli 
proposed the Partial Element Equivalent Circuit method (PEEC) \cite{Ruehli_1974aa,Ruehli_2015aa}. PEEC is based on an integral formulation of the 
equations and utilises Green's functions similarly to the Boundary Element Method (BEM) or the Method of Moments (MOM) as BEM is called in the 
electromagnetics community \cite{Harrington_1993aa}.

Historically, the Finite Element Method (FEM) was firstly employed to Maxwell's equations using nodal basis functions. For vectorial fields, this 
produces wrong results known as `spurious modes' in the literature. Their violation of the underlying structure, or more specifically of the function 
spaces, is nowadays well understood. Nédélec proposed his edge elements in 1980 \cite{Nedelec_1980aa} which are also known as Whitney elements 
\cite{Bossavit_1998aa}. A rigorous mathematical discussion can be found in many text books, e.g. \cite{Monk_2003aa}, \cite{Alonso-Rodriguez_2010aa} 
and \cite{Assous_2018aa}. Albeit less wide spread, the application of nodal elements is still popular, for example in the context of discontinuous 
Galerkin FEM \cite{Hesthaven_2008aa,Godel_2010ab}. Also equivalences among the methods have been shown, most prominently FIT can be interpreted on 
hexahedral meshes as lowest order Nédélec FEM with mass lumping \cite{Bossavit_2000ab,Bondeson_2005aa}.

From an application point of view, electromagnetic devices may behave very differently, e.g. a transformer in a power plant and an antenna of a 
mobile phone are both described by the same set of Maxwell's equations but still feature different phenomena. Therefore, engineers often solve 
subsets (\emph{simplifications}) of Maxwell's equation that are relevant for their problem, for example the well-known eddy-current problem 
\cite{Haus_1989aa,Dirks_1996aa,Rapetti_2014aa} or the well-known wave equation \cite{Taflove_2007aa}. For each, one or more \emph{formulations} have 
been proposed. They are either distinguished by the use of different variables or gauging conditions \cite{Biro_1995aa,Biro_1990aa}. 
 
It follows from the variety of simplifications and formulations that discretisation methods have individual strengths and weaknesses for the 
different classes of applications. For example, the formulation used in FDTD relies on an explicit time integration method which is particularly 
efficient if the mass matrices are easily invertible, i.e., if they are diagonal or at least block diagonal \cite{Bossavit_1999af}. This allows FDTD 
to solve problems with several billions of degrees of freedom. Classical FEM is less commonly applied in that case but one may also analyse 
high-frequency electromagnetic phenomena in the frequency domain, either to investigate resonance behaviour 
\cite{Weiland_1985aa} or source problems with right-hand-sides that can be assumed to vary sinusoidally at a given frequency. In these cases, one solves smaller complex-valued linear systems but the reduced sparsity due to FEM is counteracted by the flexibility in the mesh generation \cite{Boffi_2010aa}. Also coupling Maxwell's equations to other physics may require tailored formulations, see for example for applications in the field of semiconductors \cite{Schoenmaker_2017aa,Schilders_2005aa}.

In the low-frequency regime the situation is often more involved since one deals with degenerated versions of Maxwell's equations as certain 
contributions to the equation vanish in the static limit and the original system becomes unstable \cite{Ostrowski_2012aa,Jochum_2015aa,Eller_2017aa}. 
One often turns to approximations of Maxwell's equations as the well-known eddy-current problem. These approximate formulations are often more 
complicated as they may yield parabolic-semi-elliptic equations that become eventually systems of differential-algebraic equations (DAEs) after space 
discretisation. The resulting systems are commonly integrated in time domain by fully or linear-implicit methods, e.g. 
\cite{Nicolet_1996aa,Clemens_2003aa}. Only recently, explicit method gained again interest, \cite{Auserhofer_2009ab,Schops_2012aa,Dutine_2017ae}.

\begin{figure}[t]
	\centering
	\includegraphics{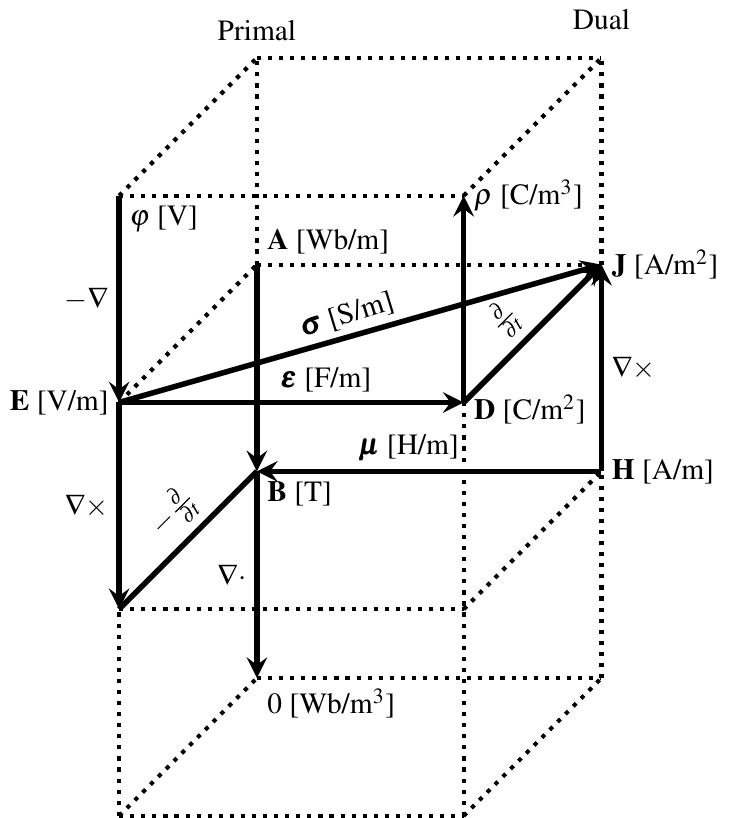}
	\caption{Maxwell's house, based on similar diagrams in \cite{Bossavit_1991aa,Deschamps_1981aa,Tonti_1975aa}. The concept of duality is for example discussed in the framework of differential forms in \cite{Hehl_2003aa} and using traditional vector calculus in \cite[Section 6.11]{Jackson_1998aa}.}
	\label{fig:house}
\end{figure}

Most circuit and electromagnetic field formulations yield DAE systems; the first mathematical treatment of such problems can be traced back to the 60s \cite{Lamour_2013aa} but gained increased interest in the 80s, e.g. \cite{Petzold_1982aa}. An important concept in the analysis of DAEs and their well-posedness are the various index concepts, which try to quantify the difficulty of the numerical time-domain solution, see e.g. \cite{Petzold_1982aa}. This paper discusses the most important low and high-frequency formulations in computational electromagnetics with respect to their differential index. An detailed introduction of the index and its variants is not discussed here and the reader is referred to text books and survey articles \cite{Brenan_1995aa, Hairer_1996aa, Lamour_2013aa, Mehrmann_2015aa}.

\bigskip

This paper summarises relevant discrete formulations stemming from Maxwell's equations. It collects the corresponding known DAE results from the 
literature, i.e., \cite{Nicolet_1996aa,Tsukerman_2002aa,Bartel_2011aa,Baumanns_2013aa}, homogenises their notation and discusses a few missing cases. 
Each problem is concretised by a mathematical description and specification of an example. The corresponding source code is freely available such 
that these example can be used as \emph{benchmarks}, e.g. for the development of time integrators or numerical tools to analyse differential 
equations. 

The paper is organised as follows: Section~\ref{sec:Maxwell} discusses Maxwell's equations, the relevant material relations and boundary conditions. 
The classical low-frequency approximations and electromagnetic potentials are introduced. Section~\ref{sec:space} outlines the spatial discretisation 
in terms of the finite element method and the finite integration technique. After establishing the DAE index concept in Section~\ref{sec:DAE}, the 
various discrete formulations are derived. They are discussed separately for the high-frequency full-wave case in Section~\ref{sec:fullwave} and the 
quasistatic approximations in Section~\ref{sec:qs}. Finally, conclusion are drawn in Section~\ref{sec:conclu}. 
\section{Maxwell's Equations}\label{sec:Maxwell}
Electromagnetic phenomena are described on the macroscopic level by Maxwell's equations \cite{Maxwell_1864aa,Heaviside_1891aa,Haus_1989aa,Jackson_1998aa,Griffiths_1999aa}
Those can be studied in a standstill frame of reference in integral form
\begin{subequations}\label{eq:maxwell_int}
  \begin{align}
    \int\displaylimits_{\partial A}\ensuremath{\vec{E}} \cdot \mathrm{d}\vec{s} &= -\int\displaylimits_{A}\frac{\partial{\ensuremath{\vec{B}}}}{\partial{{t}}}\cdot \mathrm{d}\ensuremath{\vec{A}}\;,\label{eq:afaraday_lenz_int}\\
    \int\displaylimits_{\partial V}\ensuremath{\vec{D}}\cdot\mathrm{d}\ensuremath{\vec{A}} &=\int\displaylimits_{V}\rho \mathrm{d}V\;,\\
    \int\displaylimits_{\partial  A}\ensuremath{\vec{H}}\cdot\mathrm{d}\vec{s}&=\int\displaylimits_{A}\left(\frac{\partial{\ensuremath{\vec{D}}}}{\partial{{t}}}+\ensuremath{\vec{J}}\right)\cdot\mathrm{d}\ensuremath{\vec{A}}\;,\\
    \int\displaylimits_{\partial V}\ensuremath{\vec{B}}\cdot\mathrm{d}\ensuremath{\vec{A}} &= 0,
  \end{align}
\end{subequations}
for all areas $A$ and volumes $V\subset \mathbb{R}^3$. Using Stokes and Gauß' theorems one derives a set of partial differential equations, see e.g. \cite[Chapter 1.1.2]{Assous_2018aa} for a mathematical discussion on their equivalence, 
\begin{subequations}\label{eq:maxwell}
  \begin{align}
    \nabla\times\ensuremath{\vec{E}} &= -\frac{\partial{\ensuremath{\vec{B}}}}{\partial{{t}}}\;,\label{eq:afaraday_lenz_diff}\\
    \nabla\times\ensuremath{\vec{H}} &= \frac{\partial{\ensuremath{\vec{D}}}}{\partial{{t}}} +\ensuremath{\vec{J}}\;,\label{eq:ampere_maxwell_diff}\\
    \nabla\cdot\ensuremath{\vec{D}}  &= \ensuremath{\rho}\;,\label{eq:gauss_diff} \\
    \nabla\cdot\ensuremath{\vec{B}}  &= 0 ,\label{eq:no_mag_mono_diff}
  \end{align}
\end{subequations}
with 
$\ensuremath{\vec{E}}$ the electric field strength,
$\ensuremath{\vec{B}}$ the magnetic flux density,
$\ensuremath{\vec{H}}$ the magnetic field strength,
$\ensuremath{\vec{D}}$ the electric flux density and
$\ensuremath{\vec{J}}$ the electric current density 
composed of conductive and source currents, being vector fields $\mathcal{I}\times\Omega\rightarrow\mathbbm{R}^3$
depending on space $\ensuremath{ \vec{r} }\in\Omega$  and time $t\in\mathcal{I}$. The electric charge density $\ensuremath{\rho}:\mathcal{I}\times\Omega\rightarrow\mathbbm{R}$ is the 
only 
scalar field. Finally $A$ and $V$ are all areas (respectively volumes) in $\Omega$. 
\begin{ass}[Domain]
The domain $\Omega \subset  \mathbb R^3$ is open, bounded, Lipschitz and contractible (simply connected with connected boundary, see e.g., \cite{Bossavit_1998aa}).
\label{ass:dom}
\end{ass}

Maxwell's equations give raise to the so-called de Rham complex, see e.g. \cite{Bossavit_1998aa}. It describes abstractly the relation of the electromagnetic fields in terms of the images and kernels of the differential operators. A simple visualisation is given in Fig.~\ref{fig:house}.
This diagram is sometimes called Tonti diagram \cite{Tonti_1975aa}, Deschamps diagram \cite{Deschamps_1981aa}, or, in the special case of Maxwell’s equations, Maxwell’s house \cite{Bossavit_1991aa,Bossavit_1999ae}.

\bigskip

Exploiting the  fact that the divergence of a curl vanishes, one can derive from Amp\`{e}re-Maxwell's law 
\eqref{eq:ampere_maxwell_diff}-\eqref{eq:gauss_diff} the continuity equation 
\begin{align}
\label{eq:continuity_diff}
0=\frac{\partial{\rho}}{\partial{{t}}}+\nabla\cdot\ensuremath{\vec{J}}\;,
\end{align}
which can be interpreted in the static case as Kirchhoff's current law.

\begin{figure}[t]
	\centering
	\includegraphics{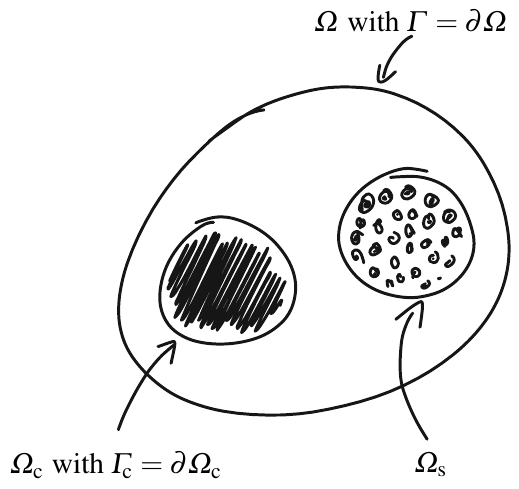}
		\caption{Sketch of domain.}
	\label{fig:patata}
\end{figure}

\subsection{Boundary Conditions and Material Relations}\label{sec:bc}
To mimic the behaviour of the electromagnetic field of an infinite domain on a truncated computational domain  and to model field symmetries, 
boundary 
conditions are imposed on $\Gamma=\partial\Omega$. We restrict ourselves 
to homogeneous electric (`ebc') and magnetic boundary conditions (`mbc')
\begin{equation}\label{eq:BC}
\begin{cases}
\vec{n}\times\ensuremath{\vec{E}} = 0 & \text{ in }\Gamma_\text{ebc}\;, \\ 
\vec{n}\times\ensuremath{\vec{H}}  = 0 & \text{ in }\Gamma_\text{mbc}\;,
\end{cases}
\end{equation}
where $\vec{n}$ is the outward normal to the boundary, $\Gamma_\text{ebc}\cup\Gamma_\text{mbc}=\Gamma$ and 
$\Gamma_\text{ebc}\cap\Gamma_\text{mbc}=\emptyset$.

\begin{remark}
Electrical engineers typically use the physical notation of electric (`ebc') or magnetic (`mbc') boundary conditions rather than the mathematical 
terminology of `Dirichlet' or `Neumann' conditions. The reason is that the mathematical distinction depends on the particular formulation, i.e. the 
variables chosen to describe the problem,  while the physical point of view remains the same. For example in an $E$-based formulation, ebc and mbc 
correspond to Dirichlet and Neumann conditions, respectively, whereas in an $H$-based formulation, ebc and mbc correspond to Neumann and Dirichlet 
conditions, respectively. 
\end{remark}

The fields in Maxwell's equations are further related to each other by the material relations
\begin{align}
  \label{eq:material}
  \ensuremath{\vec{D}}= \boldsymbol{\varepsilon} \ensuremath{\vec{E}} &\;,& \ensuremath{\vec{J}}_\mathrm{c} = \boldsymbol{\sigma}\ensuremath{\vec{E}} &\;,& \ensuremath{\vec{H}}= \boldsymbol{\nu} \ensuremath{\vec{B}}\;, 
\end{align}
where the permittivity $\boldsymbol{\epsilon}$, conductivity $\boldsymbol{\sigma}$ and reluctivity (inverse permeability $\boldsymbol{\mu}$) $\boldsymbol{\nu}$ are 
rank-2 tensor fields,
$\boldsymbol{\xi}:\Omega\rightarrow\mathbbm{R}^{3\times3}$, $\boldsymbol{\xi}\in\{\boldsymbol{\varepsilon}, \boldsymbol{\sigma},\boldsymbol{\mu},  \boldsymbol{\nu}\}$, whose possible 
polarisation or magnetisation and nonlinear or hysteretic dependencies on the 
fields 
are 
disregarded in the following for simplicity of notation and $\ensuremath{\vec{J}}_\mathrm{c}:\mathcal{I}\times \Omega\rightarrow\mathbbm{R}^3$ is the 
conduction current density. With these material relations one defines the total current density as
\begin{equation}\label{eq:Jtot}
\ensuremath{\vec{J}}=\ensuremath{\vec{J}}_\text{c}+\ensuremath{\vec{J}}_\text{s}
\end{equation}
where $\ensuremath{\vec{J}}_\text{s}$ is a given source current density that represents for example the current density impressed by a stranded conductor 
\cite{Schops_2013aa}.
 
\bigskip

We assume the following material and excitation properties as shown in Fig.~\ref{fig:patata}, see also \cite{Alonso-Rodriguez_2003aa,Schops_2013aa} for a more rigorous discussion.
\begin{ass}[Material]\label{ass:material}
The permittivity and permeability tensors, i.e., $\boldsymbol{\varepsilon}$ and $\boldsymbol{\mu}$, are positive definite on the whole domain $\Omega$ and only depend on 
space $\ensuremath{ \vec{r} }$.
The conductivity tensor is positive definite on a subdomain $\Omega_\mathrm{c}\subset\Omega$ and vanishes elsewhere, i.e., $\mathrm{supp}(\boldsymbol{\sigma}) = \Omega_\mathrm{c}$. The source current density is defined on the subdomain $\Omega_\mathrm{s}\subset\Omega$ with 
$\Omega_\mathrm{c}\cap\Omega_\mathrm{s} = 
\emptyset$, such that $\mathrm{supp}(\ensuremath{\vec{J}}_\mathrm{s}) = \Omega_\mathrm{s}$.
\end{ass}

This assumption describes the situation of an excitation given by one or several stranded conductors. The key assumption behind this model is a 
homogeneous current distribution which is justified in many situations, since the individual strands have diameters small than the skin depth and  
are therefore not 
affected by eddy currents, i.e. $\Omega_\mathrm{c}$ and $\Omega_\mathrm{s}$ are disjoint. Other models, e.g. solid and foil conductors, are not 
covered here. However, it can be shown that the various models can be transformed into each other and thus have similar properties 
\cite{Schops_2011ac}.

\subsection{Modelling of excitations}
The excitation has been given in \eqref{eq:Jtot} by the known source current density $\ensuremath{\vec{J}}_\text{s}$ which is typically either determined by given 
voltage drops $u_k:\mathcal{I}\rightarrow\mathbbm{R}$ or lumped currents $i_k:\mathcal{I}\rightarrow\mathbbm{R}$. The 3D-0D coupling is governed by so-called 
\emph{conductor models}. Besides the \emph{solid} and \emph{stranded} models 
\cite{Bedrosian_1993aa}, also more elaborated conductors have been proposed, e.g., foil-conductor models \cite{De-Gersem_2001aa}. 

The source current density $\ensuremath{\vec{J}}_\text{s}$ is not necessarily solenoidal, i.e.
$$ \nabla\cdot \ensuremath{\vec{J}}_\text{s} \neq 0.$$
Divergence-freeness is only required for the total current density $\ensuremath{\vec{J}}$ in the absence of charge variations due to the continuity equation 
\eqref{eq:continuity_diff}.  This has been exploited e.g. in \cite[Figure 3]{Schops_2013aa} to increase the sparsity of the coupling matrices. 
However, most conductor models enforce this property such that the source current can be given alternatively in terms of a source magnetic 
field strength
$$ \ensuremath{\vec{J}}_\text{s} = \nabla\times \ensuremath{\vec{H}}_\text{s}.$$

In \cite{Schops_2013aa} the abstract framework of \emph{winding density functions} was proposed. It unifies the individual stranded, solid and foil conductor models and denotes them abstractly by 
\begin{equation}
  \label{eq:winding}
 \boldsymbol{\chi}_k:\Omega\rightarrow\mathbbm{R}^{3}
\end{equation}
with an superscript if needed to distinguish among models, e.g. $\text{(i)}$ for stranded and $\text{(u)}$ solid conductors. In the simplest 
case they are characteristic functions with a given orientation. 

\begin{example}
If 
$\Omega_\text{s}=\Omega_{\text{s},1}\cup\Omega_{\text{s},2}\cup\Omega_{\text{s},3}$ consists of two parts of a winding oriented in 
$z$-direction, each with cross section $A_k$ and made of $N_k$ strands, and a massive bar with length $\ell_3$ aligned with the $z$-direction, the 
source current is given by
\begin{equation}\label{eq:Jsrc}
\ensuremath{\vec{J}}_\text{s}=\sum_{k=1}^2 \boldsymbol{\chi}_k^\text{(i)} i_k +\boldsymbol{\sigma}\boldsymbol{\chi}_3^\text{(u)} u_3\;.
\end{equation}
The winding density functions for the stranded conductor model are
\begin{equation}
	\boldsymbol{\chi}_k^\text{(i)}(\vec{r})= \begin{cases}
        \frac{N_k}{A_k}\mathbf{n}_z & \vec{r}\in\Omega_{\text{s},k}\\ 
        0 & \text{otherwise}
        \end{cases}
\end{equation}
and the unit vector in $z$-direction is denoted by $\mathbf{n}_z$. The stranded conductor model distributes an applied current in a homogeneous way 
such that the individual strands are neither spatially resolved nor modelled as line currents which would cause a too high computational effort. 
There are many proposals in the literature on how to construct them, most often a Laplace-type problem is solved on the subdomain $\Omega_s$, see 
e.g. \cite{De-Gersem_2004ac,Dyck_2004aa,Schops_2013aa}. 
The winding function for the solid conductor is
\begin{equation}
	\boldsymbol{\chi}_3^\text{(u)}(\vec{r})= \begin{cases}
	\frac{1}{\ell_3}\mathbf{n}_z & \vec{r}\in\Omega_{\text{s},3}\\ 
	0 & \text{otherwise}\;.
	\end{cases}
\end{equation}
The solid conductor model homogeneously distributes an applied voltage drop in the massive-conductor's volume.
\end{example}

The winding density functions allow to retrieve global quantities in a post-processing step, i.e., the current through a solid conductor model is calculated by
\begin{equation}\label{eq:isol}
	i_k =\int_\Omega \boldsymbol{\chi}_k^\text{(u)}\cdot \ensuremath{\vec{J}}\;\mathrm{d}V
\end{equation}
and the voltage induced along a stranded conductor model follows from
\begin{equation}\label{eq:vstr}
	u_k =-\int_\Omega \boldsymbol{\chi}_k^\text{(i)}\cdot \ensuremath{\vec{E}}\;\mathrm{d}V\;.
\end{equation}
The expressions (\ref{eq:Jsrc}), (\ref{eq:isol}) and (\ref{eq:vstr}) can also be used to set up a field-circuit coupled model~\cite{De-Gersem_2004ab}. 

An important property postulated in \cite{Schops_2013aa} is that winding functions should fulfil a partition of unity property. The integration of
$\boldsymbol{\chi}_k(\vec{r})$ along a line $\ell_k$ between both electrodes of a solid
conductor gives always $1$ and analogously, $\boldsymbol{\chi}_k^\text{(i)}(\vec{r})$ integrated
over any cross-sectional plane $A_k$ of a stranded conductor should equal the
number of turns $N_k$ of the winding:
\begin{align}\label{eq:partition_of_unity}
    \int_{\ell_k}\boldsymbol{\chi}_k^\text{(u)}\cdot\mathrm{d}\vec{s}&=1\;, \quad\forall \ell_k
    \quad\text{and}\quad
    \int_{A_k}\boldsymbol{\chi}_k^\text{(i)}\cdot\mathrm{d}\vec{S}=N_k, \quad\forall A_k \;.
\end{align}
Furthermore, conductor models should not intersect, i.e., \cite{Bartel_2011aa}
\begin{align}\label{eq:intersection}
    \boldsymbol{\chi}_i\cdot\boldsymbol{\chi}_j\equiv0 \qquad \text{for} \quad 
 i\neq j
\end{align}
where $\boldsymbol{\chi}_i$ and $\boldsymbol{\chi}_j$ are winding functions of any type. 

For simplicity of notation, we will restrict us in the following to the case of non-intersecting stranded conductors models, i.e.
\begin{ass}[Excitation]
The source current density is given by $n_\text{str}$ winding functions that fulfil \eqref{eq:partition_of_unity} and \eqref{eq:intersection} such 
that the excitation is given by $$\ensuremath{\vec{J}}_\text{s}=\sum_{k=1}^{n_\mathrm{str}} \boldsymbol{\chi}_k i_k
\quad\text{where}\quad\boldsymbol{\chi}_k\equiv\boldsymbol{\chi}_k^{(i)}.$$
\end{ass}

\subsection{Static and Quasistatic Fields}
Following the common classification of slowly varying electromagnetic fields, \cite{Dirks_1996aa}, we introduce the following definition for quasistatic and static fields  
\begin{definition}\label{def:qs}
	(\textbf{Simplifications}) The fields in Equation~\eqref{eq:maxwell} are called 
	\begin{enumerate}[label=(\alph*), ref=\alph*)]
		\item static if the variation of the magnetic and electric flux densities is disregarded:
		\label{stat}
		$$\frac{\partial{}}{\partial{{t}}}\ensuremath{\vec{B}} = 0\quad\text{and}\quad\frac{\partial{}}{\partial{{t}}}\ensuremath{\vec{D}} = 0\;;$$
		\item electroquasistatic if the variation of the magnetic flux density is disregarded:
		\label{eqs}
		$$\frac{\partial{}}{\partial{{t}}}\ensuremath{\vec{B}} = 0\;;$$
		\item magnetoquasistatic if the variation of the electric flux density is disregarded:
		\label{mqs}
		$$\frac{\partial{}}{\partial{{t}}}\ensuremath{\vec{D}} = 0\;;$$
		\item full wave if no simplifications are made. \label{full}
	\end{enumerate}
\end{definition}

In contrast to the full Maxwell's equations, the classical quasistatic approximations above feature only first order derivatives w.r.t. to time. 
However, there is another model for slowly varying fields that does not fit into this categorisation, the so-called Darwin approximation, e.g. 
\cite{Larsson_2007aa}. It considers the decomposition of the electric field strength $\ensuremath{\vec{E}}=\ensuremath{\vec{E}}_\text{irr}+\ensuremath{\vec{E}}_\text{rem}$ into an \textit{irrotational 
part} $\ensuremath{\vec{E}}_\text{irr}$ and a \textit{remainder part} $\ensuremath{\vec{E}}_\text{rem}$. In contrast to (\ref{stat}-(\ref{mqs} the Darwin approximation only neglects 
the displacement currents related to $\ensuremath{\vec{E}}_\text{rem}$ from the law of Amp\`{e}re-Maxwell \eqref{eq:ampere_maxwell_diff}. It still considers second 
order time derivatives.

\bigskip

The various approximations neglect the influence of several transient phenomena with respect to others, which implicitly categorises fields into primary and secondary ones. For example, let us consider a magnetoquasistatic situation, i.e., the displacement current density 
$\frac{\partial{}}{\partial{{t}}}\vec{D}=0$  is disregarded. This still allows the electric field $\frac{\partial{}}{\partial{{t}}}\vec{E}\neq0$ to vary. However, this variation implies that there is a secondary displacement current density $\frac{\partial{}}{\partial{{t}}}\vec{D}=\frac{\partial{}}{\partial{{t}}}\boldsymbol{\varepsilon}\vec{E}\neq0$ which is in the formulation not further considered.

\begin{remark}
Depending on the application, an electrical engineer chooses the formulation that is best suited for the problem at hand. Typically the physical 
dimensions, the materials and the occurring frequency are used to estimate which simplification is acceptable, see e.g. 
\cite{Haus_1989aa,Schmidt_2008aa,Steinmetz_2011aa}.
\end{remark}

\subsection{Electromagnetic Potentials}

Typically, one combines the relevant Maxwell equations into a \emph{formulation} by defining appropriate \emph{potentials}. One possibility is the $\ensuremath{\vec{A}}-\phi$ formulation \cite{Kameari_1990aa,Biro_1989aa,Bossavit_1998aa}, where a \emph{magnetic vector potential}  $\ensuremath{\vec{A}}:\mathcal{I}\times\Omega\rightarrow\mathbbm{R}^{3}$ and an 
\emph{electric scalar potential} $\phi:\mathcal{I}\times\Omega\rightarrow\mathbbm{R}$ 
follow as integration constants from integrating the magnetic Gauss law and
Faraday-Lenz' law in space, i.e.,
\begin{align}\label{eq:a-phi}
\ensuremath{\vec{B}}&=\nabla\times\ensuremath{\vec{A}}
&& \text{and} &
\ensuremath{\vec{E}}&=-\frac{\partial{\ensuremath{\vec{A}}}}{\partial{{t}}} -\nabla\phi\;.
\end{align}
The magnetic flux density $\ensuremath{\vec{B}}$ defines the magnetic vector potential $\ensuremath{\vec{A}}$ only up to a
gradient field. For a unique solution an additional gauging condition is required \cite{Biro_1989aa,Manges_1997aa,Clemens_2002aa}. 

A different approach can be taken with the $\ensuremath{\vec{T}}-\Omega$ 
formulation in case of a magnetoquasistatic approximation (Definition \hyperref[mqs]{\ref*{def:qs}.\ref*{mqs}}) \cite{Carpenter_1980aa,Webb_1993aa,Biro_1995aa}. 
Here, an 
\emph{electric vector potential}
$\ensuremath{\vec{T}}:\mathcal{I}\times\Omega\rightarrow\mathbbm{R}^{3}$ and a \emph{magnetic scalar potential} $\psi:\mathcal{I}\times\Omega\rightarrow\mathbbm{R}$ describe the 
fields as
\begin{align}\label{eq:to_pot}
	\ensuremath{\vec{J}}_\mathrm{c} &= \nabla\times \ensuremath{\vec{T}} &&
	\text{and} &
	\ensuremath{\vec{H}} &= \ensuremath{\vec{H}}_{\mathrm{s}} + \ensuremath{\vec{T}} - \nabla\psi\;,
\end{align}
with $\nabla\times\ensuremath{\vec{H}}_{\mathrm{s}} = \ensuremath{\vec{J}}_\mathrm{s}$. 
Again, to ensure uniqueness of solution, an additional gauge condition is necessary for $\ensuremath{\vec{T}}$. In contrast to the $\vec{A}-\phi$-formulation, the 
electric vector potential $\ensuremath{\vec{T}}$ is only 
non-zero on $\Omega_\mathrm{c}$. 

Existence and uniqueness of the continuous solution will not be discussed in this contributions, see for example \cite{Alonso-Rodriguez_2003aa,Eller_2017aa} for several formulations in the frequency domain case with anisotropic materials and mixed boundary conditions.

The boundary 
conditions introduced in \eqref{eq:BC} can now be translated into expressions involving only the potentials. This yields
for the A-$\phi$-formulation
\begin{equation}
\begin{cases}
\vec{n}\times\ensuremath{\vec{A}}=0, \, \phi=0  \,  &\text{ on }\Gamma_\text{ebc}\;,
\\ 
\vec{n}\times\left(\nu\nabla\times\ensuremath{\vec{A}}\right)= 0, \,  & \text{ on }\Gamma_\text{mbc}\;
\end{cases}
\end{equation}
and the for the T-$\Omega$ one
\begin{equation}
\begin{cases}
\mu\frac{\partial\psi}{\partial \vec{n}}= 0, \, &\text{ on }\Gamma_\text{ebc}\;,
\\ 
\vec{n}\times \nabla \psi= 0, \, & \text{ on }\Gamma_\text{mbc}\;.
\end{cases}
\end{equation}

For the electric vector potential $\ensuremath{\vec{T}}$ in the $\ensuremath{\vec{T}}-\Omega$ formulation, boundary conditions have to be set on 
the corresponding subdomain where it is defined $\Gamma_{\!\textrm{c}} = \partial{\Omega}_{\textrm{c}}$. This leads to electric boundary
conditions
\begin{align*}
	\ensuremath{ \vec{n} }_\mathrm{c}\times\ensuremath{\vec{T}} = 0 \hspace{0.5cm}  \text{on } \Gamma_{\textrm{c}}\;,
\end{align*}
with, analogous to the cases before, $\ensuremath{ \vec{n} }_\mathrm{c}$ being the outward normal unit vector of $\Gamma_{\textrm{c}}$.

\section{Spatial Discretisation}\label{sec:space}
Starting from a differential formulation the Ritz-Galerkin the FE method can be applied using the appropriate Whitney basis functions 
\cite{Monk_2003aa}. Alternatively, FIT or similarly the Cell Method provide a spatial discretisation of Maxwell's equations based on the integral 
form \cite{Weiland_1977aa,Alotto_2006aa}. 
In the lowest order case FE and FIT only differ by quadrature, i.e., FIT uses the midpoint rule \cite{Bondeson_2005aa}. We derive in the 
following the discretisation of the partial differential operators in the terminology of FIT on an hexahedral grid since this allows a simple and explicit 
construction of divergence, curl and gradient matrices which will aid the following discussion.

\subsection{Domain and Grid}\label{sec:numbering}
The domain $\Omega$ is decomposed into an oriented simplicial complex that 
forms the 
computational grid. For the explanation, it is considered to be a brick and the 
grid is defined in cartesian coordinates as
\begin{align*}
	G = \{V(i_x, i_y, i_z)\subset \mathbbm{R}^3 |& V(i_x, i_y, i_z) = [x_{i_x}, 
	x_{i_x+1}]\times[y_{i_y}, 
	y_{i_y+1}]\times [z_{i_z}, 
	z_{i_z+1}],\text{ for }\\
	& i_x=1,...,n_x-1; \ i_y=1,...,n_y-1; \ i_z=1,...,n_z-1\}.
\end{align*}
The elements $V(i_x,i_y,i_z) = V(n)$ are numbered consecutively with an index $n$: 
$$n(i_x,i_y,i_z) = i_xk_x + (i_y-1)k_y + (i_z-1)k_z,$$
with $k_x =1$, $k_y = n_x$ and $k_z = n_xn_y$. Our discrete field quantities can be defined on several geometrical objects such as 
points $P(n),$ edges $L_\omega(n)$ or facets $A_\omega(n)$. An edge $L_\omega(n)$ connects points $P(n)$ and $P(n+k_\omega)$ in $\omega=\{x,y,z\}$ direction. 
\begin{wrapfigure}[10]{r}{0.45\textwidth}
	\centering
	\vspace{-0.6cm}
	\includegraphics{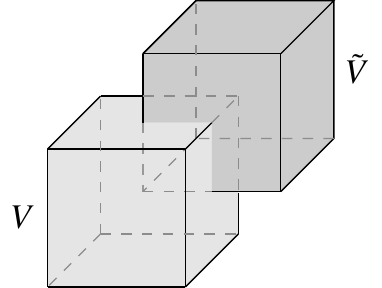}
		\caption{Primal and dual grid cells.}\label{fig:grid} 
\end{wrapfigure}
The 
facet $A_\omega(n)$ is defined by its smallest possible point $P(n)$ and directed such that its normal vector points towards 
$\omega$. There are $N=n_xn_yn_z$ points and as each point defines three edges and facets, there are in total $N_\text{dof}=3n_xn_yn_z$ edges and 
facets, ordered in $x$, $y$ and finally $z$-direction.

Nowadays, inspired by the notation of differential forms, it is well understood that a consistent mimetic discretisation of Maxwell's equations 
requires a primal/dual mesh pair. Even the discretisation with Whitney FEs implicitly constructs a dual mesh \cite{Bossavit_1998aa}. This can be 
traced back to the 
inherent structure of Maxwell's equations which are formed with quantities being dual to each other (see \cite[Section 6.11]{Jackson_1998aa}) that 
are linked by material properties (hodge operators in the terminology of differential forms). This concept is for example rigorously introduced in \cite{Hehl_2003aa}.

In contrast to FEM, both FIT and the Cell Method define the second (dual) grid $\tilde{G}$ explicitly. It is obtained by taking the centre of the 
cells in $G$ as dual grid points (see Figure \ref{fig:grid}). Now the dual quantities can be defined on the dual points $\widetilde{P}(n)$, edges 
$\widetilde{L}_\omega(n)$, facets $\widetilde{A}_\omega(n)$ and volumes $\widetilde{V}_n$. Dual edges and facets are truncated at the boundary 
\cite{Weiland_1996aa}.

\subsection{Maxwell's Grid Equations}

To illustrate the construction of the operator matrices, Faraday-Lenz's law in integral form, i.e. Equation \eqref{eq:afaraday_lenz_int},
\begin{equation*}
	\int\displaylimits_{\partial A}\ensuremath{\vec{E}} \cdot \mathrm{d}\vec{s} = -\int\displaylimits_{A}\frac{\partial{\ensuremath{\vec{B}}}}{\partial{{t}}}\cdot \mathrm{d}\ensuremath{\vec{A}}\;,
\end{equation*}
is used as an example. The equality must be fulfilled for all areas $A$, in particular for each facet $A_\omega(i,j,k)$ of the computational grid $G$. For the case $\omega = z$,
\begin{align*}
 \protect\bow{\mathrm{\mathbf{e}}}_x(i,j,k) + \protect\bow{\mathrm{\mathbf{e}}}_y(i+1,j,k) - \protect\bow{\mathrm{\mathbf{e}}}_x(i,j+1,k)- \protect\bow{\mathrm{\mathbf{e}}}_y(i,j,k) = - \frac{\mathrm{d}}{\mathrm{d}{}t}\,{}\protect\bbow{\mathrm{\mathbf{b}}}_z(i,j,k)\;,
\end{align*}
with 
$$\protect\bow{\mathrm{\mathbf{e}}}_\omega(i,j,k) = \displaystyle{\int\displaylimits_{L_\omega(i,j,k)}}\ensuremath{\vec{E}} \cdot \mathrm{d}\vec{s} \quad\text{ and }\quad \protect\bbow{\mathrm{\mathbf{b}}}_z(i,j,k) = 
-\displaystyle{\int\displaylimits_{A_z(i,j,k)}}\ensuremath{\vec{B}}\cdot \mathrm{d}\ensuremath{\vec{A}}\;.$$ 
This procedure is carried out for all the facets of $G$ and 
the following matrix equation
\begin{equation*}
\underbrace{\begin{bmatrix}
	& & & & {\vdots} & & & \\
	{\text{$\cdots$}} & 1 & {\text{$\cdots$}} & -1 & {\text{$\cdots$}} & -1 & 1 & {\text{$\cdots$}} \\
	& & & & {\vdots} & & & 
	\end{bmatrix}
}_\mathbf{C}
{\protect\bow{\mathrm{\mathbf{e}}}} 
= -\frac{\mathrm{d}}{\mathrm{d}{}t}\,{}\protect\bbow{\mathrm{\mathbf{b}}}
\end{equation*}
is obtained, which describes Faraday's law in our grid. The matrix $\mathbf{C}$ applies the curl operator on quantities integrated along edges. Similarly, the divergence matrix $\mathbf{S}$, acting on surface integrated degrees of freedom and the gradient matrix $\mathbf{G}$ are built. The same strategy is followed to obtain the matrices for the dual grid $\widetilde{\mathbf{C}}$, $\widetilde{\mathbf{S}}$ and $\widetilde{\mathbf{G}}$. It can be shown that the matrices mimic all classical identities of vector field on the discrete level, e.g. \cite{Schuhmann_2001aa} and \cite[Appendix A]{Schops_2011ac}. With this, the semi-discrete Maxwell's Grid Equations
\begin{subequations}\label{eq:maxwell_grid}
  \begin{align}
    \mathbf{C} \protect\bow{\mathrm{\mathbf{e}}} &= -\frac{\mathrm{d}}{\mathrm{d}t}\protect\bbow{\mathrm{\mathbf{b}}}\;,\label{eq:faraday_fit}\\
    \widetilde{\mathbf{C}} \protect\bow{\mathrm{\mathbf{h}}} &=\frac{\mathrm{d}}{\mathrm{d}{}t}\,{}\protect\bbow{\mathrm{\mathbf{d}}} + \protect\bbow{\mathrm{\mathbf{j}}}\;,\label{eq:ampere_fit}\\
    \mathbf{S} \protect\bbow{\mathrm{\mathbf{b}}} &= 0\;,\label{eq:nomono_fit}\\
    \widetilde{\mathbf{S}} \protect\bbow{\mathrm{\mathbf{d}}} &= \mathrm{\mathbf{q}}\label{eq:gauss_fit}
  \end{align}
\end{subequations}
are obtained which are closely resemble the system \eqref{eq:maxwell}. The matrices $\mathbf{C}$, $\widetilde{\mathbf{C}}\in\{-1,0,1\}^{N_\text{dof}\times N_\text{dof}}$ are the 
discrete curl operators, $\mathbf{S}$, $\widetilde{\mathbf{S}}\in\{-1,0,1\}^{N\times N_\text{dof}}$ the 
discrete divergence operators, which are all defined on the primal and dual grid, 
respectively. The fields are semi-discretely given by 
$\protect\bow{\mathrm{\mathbf{e}}}$, $\protect\bow{\mathrm{\mathbf{h}}}$, $\protect\bbow{\mathrm{\mathbf{d}}}$, $\protect\bbow{\mathrm{\mathbf{j}}}$, $\protect\bbow{\mathrm{\mathbf{b}}}:\mathcal{I}\to\ensuremath{\mathbb{R}}^{N_\text{dof}}$ and 
$\mathrm{\mathbf{q}}:\mathcal{I}\to\ensuremath{\mathbb{R}}^{N}$, and correspond to integrals of electric and magnetic voltages, 
electric fluxes, electric currents, magnetic fluxes and electric charges, respectively. 

\begin{lem}\label{prop:topmatrices}
	The operator matrices fulfil the following properties \cite{Weiland_1996aa}
	 \begin{itemize}
	 	\item divergence of the curl and curl of the gradient vanish on both grids
			\begin{equation}
				\mathbf{S}\mathbf{C} = 0\;,\;\;\widetilde{\mathbf{S}}\widetilde{\mathbf{C}} = 0
				\quad\text{and}\quad
				\mathbf{C}\mathbf{G} = 0\;,\;\;\widetilde{\mathbf{C}}\widetilde{\mathbf{G}} = 0
			\end{equation}
	 	\item primal (dual) gradient and dual (primal) divergence fulfill 
			\begin{equation}\label{eq:grad_div_fit}
				\mathbf{G} = -\widetilde{\mathbf{S}}^{\top} \quad\text{and}\quad \widetilde{\mathbf{G}} =-\mathbf{S}^{\top}
			\end{equation}
	 	\item curl and dual curl are related by 
			\begin{equation}
				\widetilde{\mathbf{C}} = \mathbf{C}^{\top}.
			\end{equation}
	 \end{itemize}
\end{lem}

Furthermore, potentials can be introduced on the primal grid, i.e.
\begin{align}
  \protect\bow{\mathrm{\mathbf{e}}} = -\frac{\mathrm{d}}{\mathrm{d}t}\protect\bow{\mathrm{\mathbf{a}}} - \mathbf{G}{\boldsymbol\Phi}\;,
\end{align}
where $\protect\bow{\mathrm{\mathbf{a}}}$ is the line-integrated magnetic vector potential and ${\boldsymbol\Phi}$ the electric scalar potential located on primary nodes. This is similar 
to the definition of the potentials in the continuous case, i.e., \eqref{eq:a-phi}. The properties stated in this Lemma have been proven in 
\cite{Schops_2011ac,Baumanns_2012ab,Bossavit_2000ab}.

\bigskip

The numbering scheme explained in Section \ref{sec:numbering} yields matrices with a simple banded structure. The sparsity pattern is such that an 
efficient implementation may not construct those matrices explicitly but apply the corresponding operations as such to vectors. However, the 
numbering scheme introduces superfluous objects allocated outside of the domain $\Omega$. For example in the case of the points located at the 
boundary where $i_x = n_x$, an edge in $x$ direction $L_x(n_x,i_y,i_z)\notin\Omega$. Those objects are called \emph{phantom objects}. However, the 
homogeneous Dirichlet boundary conditions explained in Section \ref{sec:bc}, as well as the deletion of the phantom objects can be incorporated 
either by removing them with (truncated) projection matrices or by setting the corresponding degrees of freedom to zero. For a more detailed 
description of the process, see \cite{Baumanns_2012ab} and \cite[Appendix A]{Schops_2011ac}. 

\begin{ass}[Boundary conditions]\label{ass:nophantoms}
	The degrees of freedom and all the operators are projected to an appropriate subspace considering the homogeneous Dirichlet boundary (`ebc') conditions 
	and disregarding any phantom objects in $\widetilde{\mathbf{S}}$, $\widetilde{\mathbf{S}}^{\top}\!$, $\mathbf{C}$ and $\widetilde{\mathbf{C}}$. Therefore $\ker\widetilde{\mathbf{S}}^{\top} = 0$.
\end{ass}
This assumption imposes boundary conditions directly on the system matrices and thus is a necessary condition to ensure uniqueness of solution. It is important to note that the reduced matrices keep the properties described in Lemma~\ref{prop:topmatrices}, see for example \cite[Section 3.2.4]{Baumanns_2012ab}.

Please note that identical operators (without phantom objects) are obtained when applying the FE method with lowest-order Whitney basis functions using the same primal grid \cite{Bossavit_2000ab,Bondeson_2005aa}.

\subsection{Material matrices}
The degrees of freedom have been introduced as integrals and thus the discretisation did not yet introduce any approximation error. This however 
happens when applying the matrices describing the material relations. In the FE case, the material matrices are given by the integrals 
$$\left[\mathbf{M}_\xi\right]_{n,m}=\int_\Omega \mathbf{w}_n\cdot\xi\mathbf{w}_m\; \mathrm{d}\Omega\;,$$
where $\xi\in\{\boldsymbol{\sigma},\boldsymbol{\nu},\boldsymbol{\varepsilon}\}$ and $\mathbf{w}_\star$ are from an appropriate space, i.e., tangentially continuous N\'ed\'elec 
vectorial shape functions \cite{Nedelec_1980aa,Bossavit_1988aa} related to the $n$-th edge of the grid for discretising $\boldsymbol{\varepsilon}$ and 
$\boldsymbol{\sigma}$ and normally continuous Raviart-Thomas vectorial shape functions \cite{Raviart_1977aa,Bossavit_1988aa} for discretizing $\boldsymbol{\nu}$. 

In FIT, the matrix construction is derived from the Taylor expansion of the material laws. In the following, only the construction of the conductivity matrix is explained. For simplicity of notation, the conductivity $\sigma(\ensuremath{ \vec{r} })$ 
is assumed to be isotropic and conforming to the primal grid, i.e. $\sigma^{(n)}=\sigma(\ensuremath{ \vec{r} }_n)$ is constant on each primal volume ($\ensuremath{ \vec{r} }_n\in V(n)$). Consider a primal edge $L_z(i,j,k)$ and its associated dual facet $\widetilde{A}_z(i,j,k)$ (Figure~\ref{fig:sigma}).
\begin{figure}
	\centering
\includegraphics{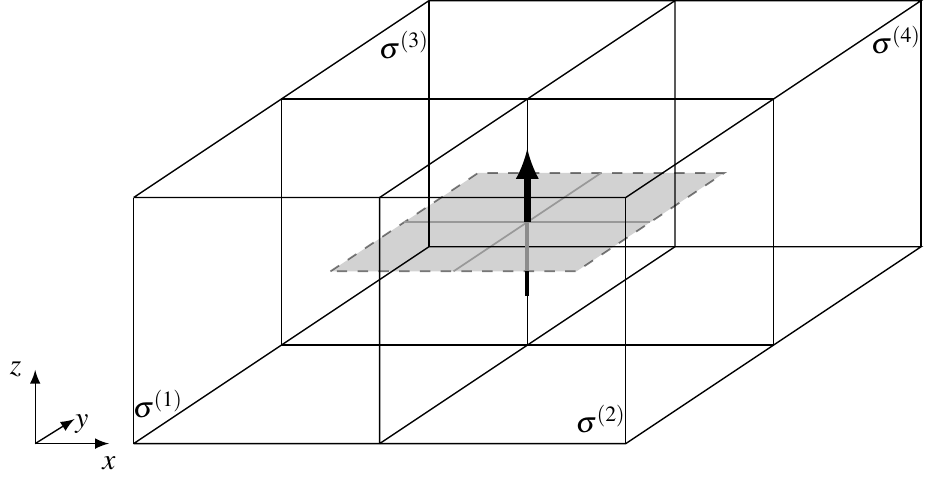}
	\caption{Sketch of dual facet $\widetilde{A}_z(i,j,k)$ with its normal vector.}\label{fig:sigma}
\end{figure}
The tangential component ${E}_z$ of the electric field strength is continuous along $L_z(i,j,k)$ and is found by approximation from
\begin{align*}
\protect\bow{\mathrm{\mathbf{e}}}_z(i,j,k) =& \int\displaylimits_{L_z(i,j,k)}\ensuremath{\vec{E}}\cdot\mathrm{d}\vec{s}
\approx {E}_z\;|L_z(i,j,k)|\;,
\end{align*}
where $|\cdot|$ denotes the length, area or volume depending on the object.
The current density integrated on the corresponding dual facet reads
\begin{align*}
\protect\bbow{\mathrm{\mathbf{j}}}_z(i,j,k) &= \int\displaylimits_{\widetilde{A}_z(i,j,k)}\ensuremath{ \vec{J} }\cdot\mathrm{d}\ensuremath{ \vec{A} } = 
\int\displaylimits_{\widetilde{A}_z(i,j,k)}J_z\;\mathrm{d}A =\mathlarger{\mathlarger{\sum}}_{q=1}^4 \ 
\int\displaylimits_{\widetilde{A}_z^{(q)}(i,j,k)}\!\!\!\!\sigma^{(q)}E_z\;\mathrm{d}A \\
&\approx \mathlarger{\mathlarger{\sum}}_{q=1}^4 \sigma^{(q)} E_z |\widetilde{A}_z^{(q)}(i,j,k)|
= \ensuremath{ \mathbf{M} }_{\sigma,i,j,k}\protect\bow{\mathrm{\mathbf{e}}}_z(i,j,k)\;,
\end{align*}
where the conductances $\ensuremath{ \mathbf{M} }_{\sigma,i,j,k}=\bar{\sigma}(i,j,k)\frac{|\widetilde{A}_z(i,j,k)|}{|L_z(i,j,k)|}$ include the conductivities
\begin{align*} 
	\bar{\sigma}(i,j,k) =\mathlarger{\mathlarger{\sum}}_{q=1}^4 \sigma^{(q)}\frac{|\widetilde{A}_z^{(q)}(i,j,k)|}{|\widetilde{A}_z(i,j,k)|}
\end{align*}
averaged according to the conductivities $\sigma^{(q)}$ of the primal grid cells $V^{(q)}$ surrounding $L_z(i,j,k)$ and the surface fractions 
$\widetilde{A}_z^{(q)}(i,j,k)=V^{(q)}\cap\widetilde{A}_z(i,j,k)$. Analogously, material matrices for $\varepsilon$ and $\nu$ are obtained, which lead to the 
discretised material relations
\begin{align*}
	\protect\bbow{\mathrm{\mathbf{d}}} = \mathbf{M}_{\varepsilon} \protect\bow{\mathrm{\mathbf{e}}} &\;,& \protect\bbow{\mathrm{\mathbf{j}}}_\mathrm{c} = \mathbf{M}_{\sigma} \protect\bow{\mathrm{\mathbf{e}}} &\;,& \protect\bow{\mathrm{\mathbf{h}}} = \mathbf{M}_{\nu} \protect\bbow{\mathrm{\mathbf{b}}}
\end{align*}
and 
\begin{align*}
  \protect\bbow{\mathrm{\mathbf{j}}} = \protect\bbow{\mathrm{\mathbf{j}}}_\mathrm{c}+\protect\bbow{\mathrm{\mathbf{j}}}_\mathrm{s}
\end{align*}
with the source current density $\protect\bbow{\mathrm{\mathbf{j}}}_\mathrm{s}$, which may be given by the discretisation $\mathbf{X}$ of the winding function \eqref{eq:winding}, 
such that $\protect\bbow{\mathrm{\mathbf{j}}}_\mathrm{s}=\sum_k\mathbf{X}_k\,i_k$ with currents $i_k$.

\bigskip

For the material matrices, one can show the following result \cite[Appendix A]{Schops_2011ac}.
\begin{lem}[Material Matrices]\label{prop:material}
	The material matrices $\ensuremath{ \mathbf{M} }_{\xi}$ are symmetric for all material properties $\xi = \{\sigma, \nu, \varepsilon\}$. If Assumption \ref{ass:material} holds, then the matrices $\ensuremath{ \mathbf{M} }_{\nu},\ensuremath{ \mathbf{M} }_{\varepsilon}$ are positive definite whereas $\ensuremath{ \mathbf{M} }_{\sigma}$ is only positive semidefinite.
\end{lem}

Finally, the discretised version of Maxwell's equations with its corresponding material laws can be visualised by `Maxwell's house' shown in 
Figure~\ref{fig:tonti}.
\begin{figure}[t]
\centering

\includegraphics{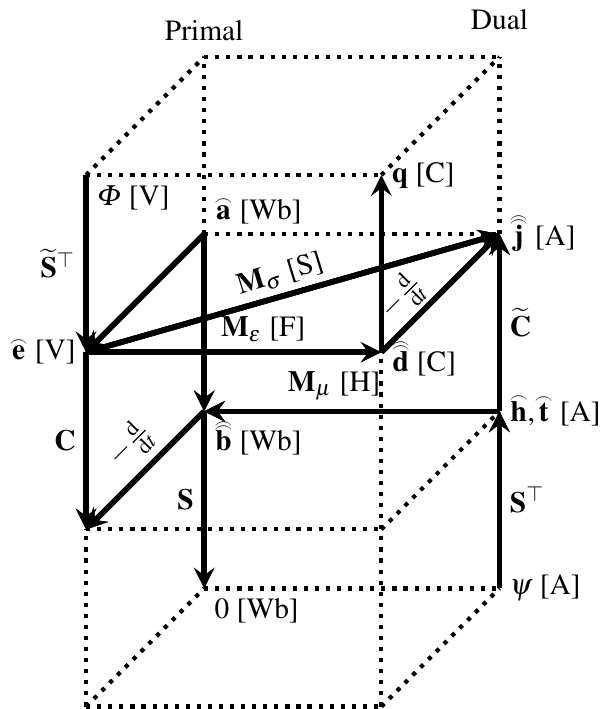}
 \caption{Maxwell's house after spatial discretisation.\label{fig:tonti}}
\end{figure}

\begin{remark}
	Both Lemma  \ref{prop:topmatrices} and \ref{prop:material}, as well as Assumption \ref{ass:nophantoms} hold for Finite Element discretisations with basis functions fulfilling a discrete de Rham sequence.
\end{remark}

\begin{remark}
	In many applications the material parameters, e.g. the reluctivity, in \eqref{eq:material} dependent nonlinearly on the fields. In these cases 
	one may consider the linearised system but since the differential material properties inherit the relevant properties, e.g.  \cite{Heise_1994aa}, 
	the characteristics of the DAE will also remain the same. In particular, there will be no change of nullspaces, see e.g. \cite{Bartel_2011aa}
\end{remark}

\section{Differential Algebraic Equations}\label{sec:DAE}
Starting from Maxwell's grid equations, various discrete time-domain formulations can be obtained. Depending on the choices made according to Definition~\ref{def:qs}, the 
resulting system is either static, first or second-order in time. In the dynamic cases, it can be written as a (linear) problem of the form 
\begin{align}
\label{eq:dae}
\mathbf{M}\frac{\mathrm{d}}{\mathrm{d}{}t}\,{}\mathbf{x}(t) + \mathbf{K} \mathbf{x}(t) = \mathbf{r}(t),
\quad\text{and}\quad\mathbf{x}(t_0)=\mathbf{x}_0
\end{align}
where $\mathbf{M}, \mathbf{K} \in \mathbbm{R}^{n\times n}$ are matrices, $\mathbf{x}: [t_0, T] \rightarrow \mathbbm{R}^n$
contains the time-dependent degrees of freedom and $\mathbf{r}: [t_0, T] \rightarrow \mathbbm{R}^n$ is an 
input.
\begin{definition}[DAE]
	Equation \eqref{eq:dae} is called a system of differential-algebraic
	equations (DAE) if $\mathbf{M}$ is singular.
\end{definition}

There are many options how to perform time-discretisation (`integration') of a DAE \eqref{eq:dae}, see for example \cite{Hairer_1996aa}. We suggest 
the 
simplest approach: implicit Euler's method, i.e.,
\begin{align}
\label{eq:euler}
\left(\mathbf{M}/{\Delta t}+\mathbf{K}\right) \mathbf{x}_{n+1} = \mathbf{r}(t_{n+1})+\mathbf{M}/{\Delta t}\;\mathbf{x}_{n}
\end{align}
where $\mathbf{x}_{n}\stackrel{\cdot}{=}\mathbf{x}(t_n)$ and $\Delta t = t_{n+1}-t_n$ is the time step.
DAEs are commonly classified according to their \emph{index}. Intuitively, it can be seen as a measure of the equations' 
sensitivity to perturbations of the input and the numerical difficulties when integrating. There are several competing index concepts. They 
essentially agree in the case of regular, linear problems, see \cite{Mehrmann_2015aa} for detailed discussion. Therefore, we employ the simplest 
concept
\begin{definition}[Differential Index, \cite{Brenan_1995aa}]
	If solvable and the right-hand-side is smooth enough, then the DAE \eqref{eq:dae} has differential index-$\vartheta$ if $\vartheta$ is the minimal number of analytical differentiations with respect
	to the time $t$ that are necessary to obtain an ODE for $d\mathbf{x}/dt$
	as a continuous function in $\mathbf{x}$ and $t$ by algebraic manipulations only.
\end{definition}

For $\vartheta\geq2$ the time-integration becomes difficult. Let us consider the classical educational index-2 problem to motivate analytically the sensitivity with respect to perturbations. The problem is described by
\begin{align}
  \label{eq:index2perturb}
  \frac{\mathrm{d}}{\mathrm{d}t}{x}_1 = x_2 
  \quad \text{and} \quad
     x_1  = \sin(t) + \delta(t)
\end{align}
where $\delta(t) = 10^{-k}\sin(10^{2k}t)$ is a small perturbation with $k\gg 1$. The solution $x_2 = \cos(t) + 10^{k}\cos(10^{2k}t)$ is easily 
obtained by the product and chain rules. It shows that a very small perturbation in an index-2 system (at a high frequency) can have a serious impact 
(in the order of $10^{k}$) on the solution when compared to the original solution  $x_2=\cos(t)$ of the unperturbed problem where $\delta=0$.

\begin{remark}
	For the index analysis in the following sections we assume that the right-hand sides are smooth enough.
\end{remark}

Furthermore, DAEs are known for the fact that solutions have to
fulfil certain constraints. One of the difficult parts in solving DAEs
numerically is to determine a consistent set of initial conditions in order
to start the integration \cite{Estevez-Schwarz_1999ab, Marz_2003aa, Baumanns_2010aa}.

\begin{remark} \cite{Lamour_2013aa}
	A vector $\mathbf{x}_0\in\mathbbm{R}^n$ is called a consistent initial value if there is a solution of \eqref{eq:dae} through $\mathbf{x}_0$ at 
	time $t_0$.
\end{remark}

The problems discussed in the following will have at most (linear) index-2 components. For this case it has be shown that if we are not interested in 
a consistent initialisation at time $t_0$ but accept a solution satisfying the DAE only after the first step, then one may apply the implicit Euler 
method starting with an operating point and still obtain the same solution after $t>t_0$ that one would have obtained using a particular consistent 
value 
\cite{Baumanns_2010aa,Baumanns_2012ab}.

\medskip

The aim of this paper is to study the index of the systems obtained with different formulations and approximations according to Definition~\ref{def:qs}. \section{Full-Wave Formulation}\label{sec:fullwave}
On first sight it seems optimal to analyse high-frequency electromagnetic phenomena, e.g. the radiation of antennas, in frequency domain. The 
right-hand-sides can often be assumed to vary sinusoidally and for a given frequency, the equations are linear as the materials are rather frequency 
than field-dependent. However, the solution of problems in frequency domain requires the resolution of very large systems of equations and becomes 
inconvenient if one is interested in many frequencies (\emph{broadband solution}). Therefore, often time-domain simulations are carried out with 
right-hand-sides that excite a large frequency spectrum.

\subsection{First-Order Formulation Time-Stepped by Leapfrog}
When solving Maxwell's grid equations for lossless ($\boldsymbol{\sigma}\equiv0$) wave propagation problems in time domain, a problem formulation based on 
the electric and magnetic field is commonly proposed. Assuming that the initial conditions fulfil the divergence relations of 
System~\eqref{eq:maxwell}, one starts with Faraday's and Amp\`ere's laws
\begin{align*}
	\frac{\partial{\ensuremath{\vec{B}}}}{\partial{{t}}} +\nabla\times\ensuremath{\vec{E}}  =0
	\quad\text{and}\quad
	\frac{\partial{\ensuremath{\vec{D}}}}{\partial{{t}}} -\nabla\times\ensuremath{\vec{H}}  =\ensuremath{\vec{J}}_\textrm{s}\;.
\end{align*}
After inserting the material laws, the system becomes
\begin{align*}
	\boldsymbol{\mu}\frac{\partial{\ensuremath{\vec{H}}}}{\partial{{t}}} + \nabla\times\ensuremath{\vec{E}} = 0 
	\quad\text{and}\quad
	\boldsymbol{\varepsilon}\frac{\partial{\ensuremath{\vec{E}}}}{\partial{{t}}} - \nabla\times\ensuremath{\vec{H}}  = \ensuremath{\vec{J}}_\textrm{s}\;,
\end{align*}
with $\boldsymbol{\nu}=\boldsymbol{\mu}^{-1}$. Using Maxwell's grid equations~\eqref{eq:maxwell_grid}, 
the semi-discrete initial value problem (IVP) has the form of equation \eqref{eq:dae}
with unknown voltages ${\mathbf{x}}^{\top} := [\protect\bow{\mathrm{\mathbf{h}}}^{\top},\protect\bow{\mathrm{\mathbf{e}}}^{\top}]$, right-hand-side ${\mathbf{r}}^{\top}:=[0,\protect\bbow{\mathrm{\mathbf{j}}}_\mathrm{s}^{\top}]$ and matrices
\begin{align}
    \mathbf{M} := 
    \begin{bmatrix}
        \mathbf{M}_\nu^{-1} &  0 \\
        0              & \mathbf{M}_{\varepsilon}  
    \end{bmatrix}
    \quad \text{and} \quad 
    \mathbf{K} := 
    \begin{bmatrix}
        0  & \mathbf{C} \\
        -\widetilde{\mathbf{C}}       & 0  
    \end{bmatrix}. \label{eq:fit_matrices}
\end{align}
If Assumptions~\ref{ass:material}  and \ref{ass:nophantoms} holds, all superfluous degrees of freedom are removed and the material matrices 
$\mathbf{M}_\nu$ and $\mathbf{M}_{\varepsilon}$ have 
full rank. With FIT the matrices are furthermore diagonal and thus easily inverted. A transformation by the matrices $\mathbf{M}_\nu^{-1/2}$ and 
$\mathbf{M}_\varepsilon^{1/2}$ allows us to rewrite \eqref{eq:dae} as
\begin{align}
    \frac{\mathrm{d}}{\mathrm{d}{}t}\,{}\bar{\mathbf{x}}(t) &= \bar{\mathbf{K}} \bar{\mathbf{x}}(t)+\bar{\mathbf{r}}(t) & \bar{\mathbf{x}}(t_0) &= \bar{\mathbf{x}}_0
    \label{eq:FIT_ode_2}
\end{align}
in the new unknowns $\bar{\mathbf{x}}^{\top}=[(\mathbf{M}_\nu^{-1/2}\protect\bow{\mathrm{\mathbf{h}}})^{\top},(\mathbf{M}_\epsilon^{1/2}\protect\bow{\mathrm{\mathbf{e}}})^{\top}]$
with the skew-symmetric stiffness matrix
\begin{align}
 \bar{\mathbf{K}}=\begin{bmatrix}
  0 & -\mathbf{M}_\nu^{1/2}\mathbf{C}\mathbf{M}_{\varepsilon}^{-1/2}\\
  \mathbf{M}_{\varepsilon}^{-1/2}\widetilde{\mathbf{C}}\mathbf{M}_\nu^{1/2} & 0
 \end{bmatrix}.
\end{align}
and right-hand-side $\bar{\mathbf{r}}^{\top}=[0,(\mathbf{M}_{\varepsilon}^{-1/2}\protect\bbow{\mathrm{\mathbf{j}}}_\text{s})^{\top}]$. Let us conclude this by the following result

\begin{theorem}
Let Assumptions~\ref{ass:dom}, \ref{ass:material} and \ref{ass:nophantoms} hold. Then, the semidiscrete full-wave Maxwell equations expressed in the 
field strengths, i.e., 
\eqref{eq:FIT_ode_2} are an explicit system of ordinary differential equations.
\end{theorem}

\begin{figure}[t]
  \centering
  \includegraphics[width=.8\textwidth]{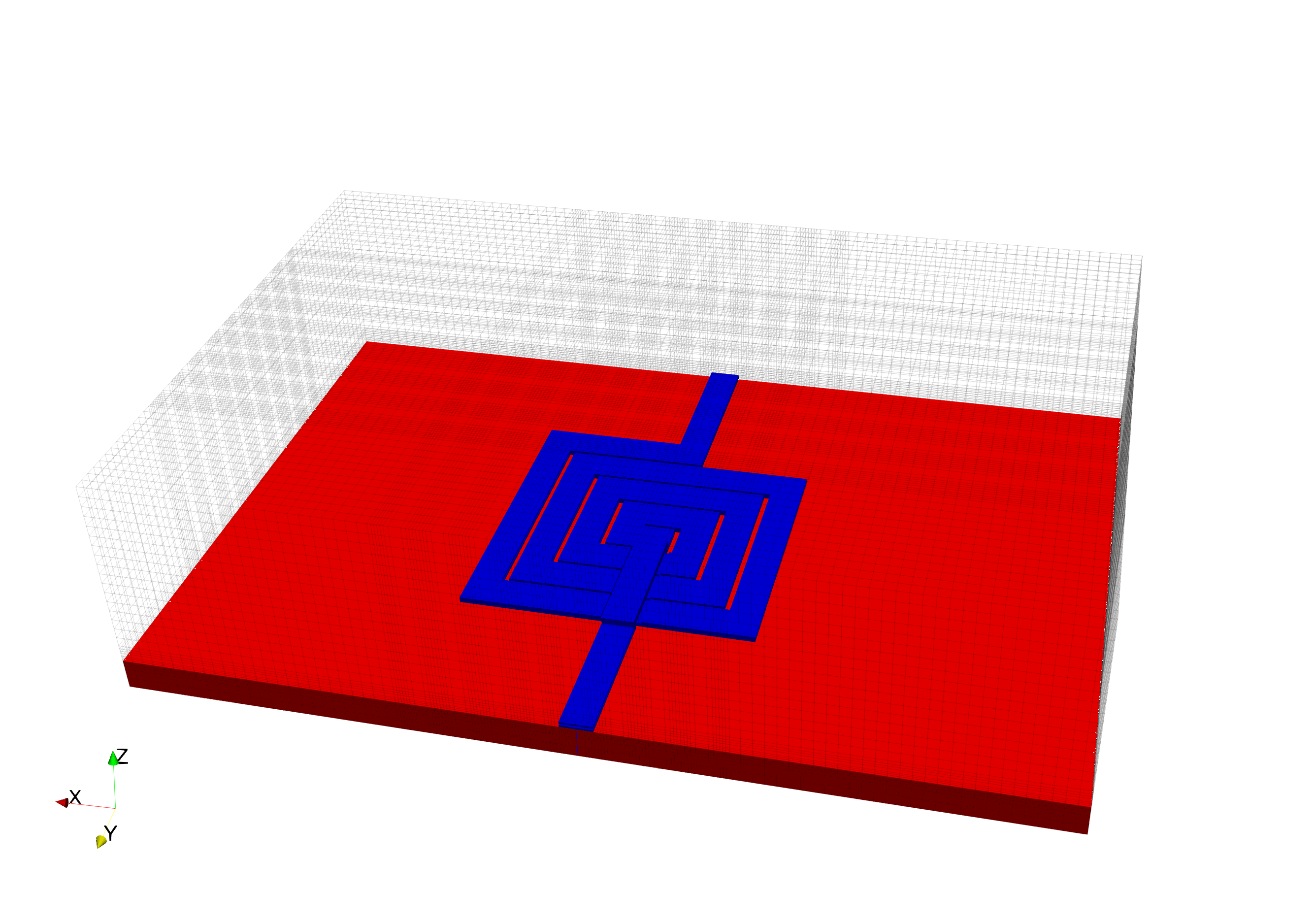}
  \caption{Spiral inductor model with coplanar lines located 
on a substrate layer with an air bridge (Benchmark \ref{bench:full})}
  \label{fig:spiral}
\end{figure}

The resulting IVP could be readily solved by the implicit Euler method \eqref{eq:euler} or any method that is tailored for second order differential equations. However, as explained above FIT allows to efficiently invert the mass matrix $\mathbf{M}$ and thus explicit methods become interesting. Typically the leapfrog scheme (or equivalently St\"{o}rmer-Verlet) are used \cite{Taflove_2007aa}. The restriction on the time step size related to the Courant–Friedrichs–Lewy-condition (CFL) is tolerable if the dynamics of the right-hand-side are in a similar order of magnitude. Leapfrog is second-order accurate and symplectic, which is particularly interesting if there is no damping, i.e., no conductors present ($\boldsymbol{\sigma}\equiv0$). Furthermore it can be shown that space and time errors are well balanced when using the leapfrog scheme with the a time step size close to the CFL limit ("magic time step") \cite[Chapters 2.4 and 4]{Taflove_1995aa}.

\bigskip

Let the initial conditions be
\begin{align*}
	\protect\bow{\mathrm{\mathbf{e}}}^{(0)}=\protect\bow{\mathrm{\mathbf{e}}}_{0} \quad\text{and\quad} \protect\bow{\mathrm{\mathbf{h}}}^{(\frac{1}{2})}=\protect\bow{\mathrm{\mathbf{h}}}_{1/2} \quad ,
\end{align*}
then the update equations for the leapfrog scheme read \cite{Weiland_1977aa,Weiland_1996aa,Taflove_1995aa}
\begin{align*}
	\protect\bow{\mathrm{\mathbf{e}}}^{(m+1)}&=\protect\bow{\mathrm{\mathbf{e}}}^{(m)}+\Delta t \mathbf{M}_{\varepsilon}^{-1}\left(\widetilde{\mathbf{C}}\protect\bow{\mathrm{\mathbf{h}}}^{(m+\frac{1}{2})}-\protect\bbow{\mathrm{\mathbf{j}}}^{(m+\frac{1}{2})}\right),\\
	\protect\bow{\mathrm{\mathbf{h}}}^{(m+\frac{3}{2})}&=\protect\bow{\mathrm{\mathbf{h}}}^{(m+\frac{1}{2})}-\Delta t \mathbf{M}_\nu \widetilde{\mathbf{C}}\protect\bow{\mathrm{\mathbf{e}}}^{(m+1)}
\end{align*}
for the electric and magnetic voltages $\protect\bow{\mathrm{\mathbf{e}}}^{(m)}$, $\protect\bow{\mathrm{\mathbf{h}}}^{(m+\frac{1}{2})}$ at time instants $t_m$ and $t_{m+\frac{1}{2}}$ with step size $\Delta t$. For equidistant grids, the resulting scheme is (up to scaling and interpretation) equivalent to Yee's FDTD scheme \cite{Yee_1966aa}. 

\begin{remark}
In practice, one may choose to violate Assumption~\ref{ass:nophantoms}. Instead one imposes the boundary conditions by setting the corresponding 
entries in the material matrices $\mathbf{M}_{\varepsilon}^{-1}$ and $\mathbf{M}_\nu$ to zero. In this case the system \eqref{eq:fit_matrices} comes with additional 
(trivial) equations when compared to a system that is projected to the lower dimensional subspace containing the boundary conditions. However, this 
preserves a simpler structure of the equation system and the topological grid operators, e.g. the discrete curl matrix $\mathbf{C}$, keep their banded 
structure. 
\end{remark}

\begin{figure}[t]
	\centering
	\subfigure[x-y cross section of spiral inductor.]{
		\includegraphics{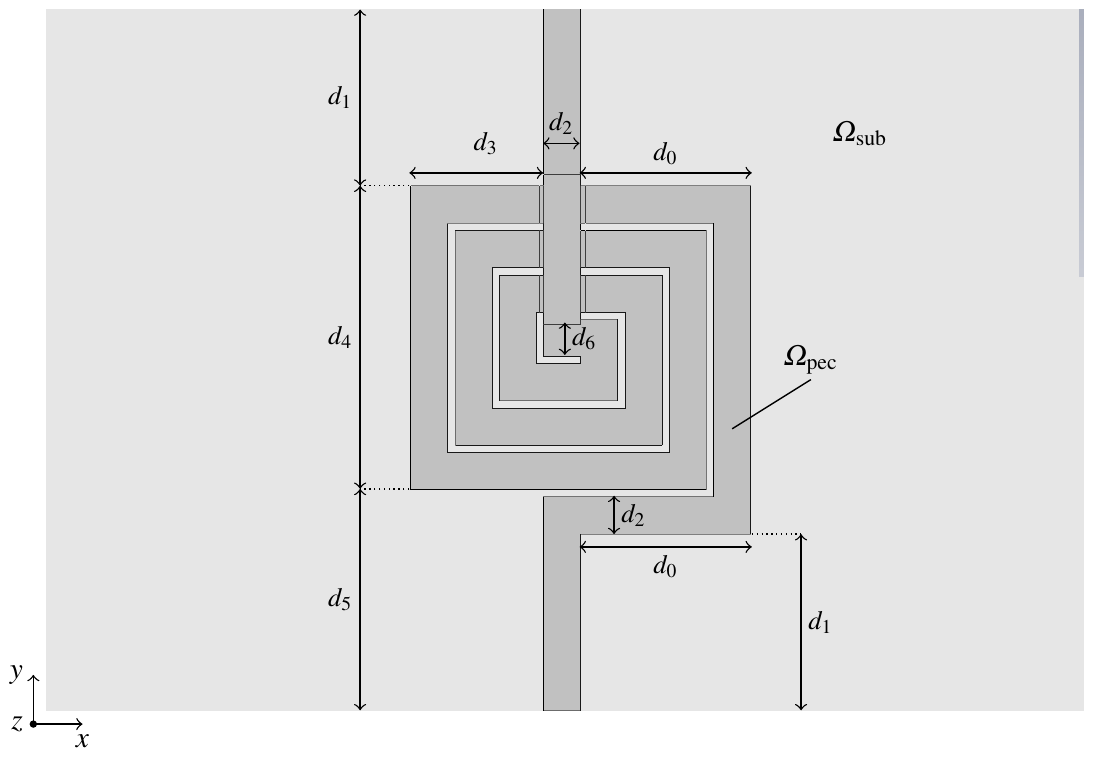}
			
	}
	\subfigure[y-z cross section of spiral inductor.]{
		\includegraphics{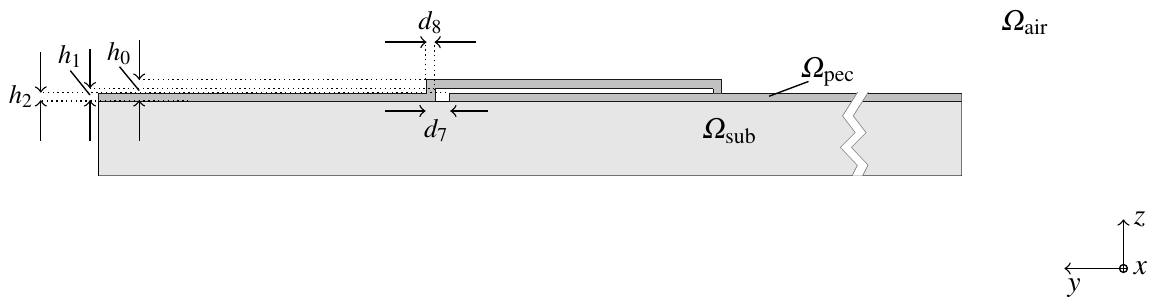}
			}
	\caption{Model of a spiral inductor of Benchmark \ref{bench:full}. The distances are $d_{0} = 1.15\cdot10^{-4}\,\si{\metre}$, 
		$d_{1} = 1.2\cdot10^{-4}\,\si{\metre}$, $d_{2} = 2.5\cdot10^{-5}\,\si{\metre}$, $d_{3} = 9\cdot10^{-5}\,\si{\metre}$, $d_{4} = 
		2.05\cdot10^{-4}\,\si{\metre}$, $d_{5} = 1.5\cdot10^{-4}\,\si{\metre}$, $d_{6} = 2.2\cdot10^{-5}\,\si{\metre}$, $d_{7} = 
		9\cdot10^{-6}\,\si{\metre}$, $d_{8} = 3\cdot10^{-6}\,\si{\metre}$, 
		$h_{0} = 8\cdot10^{-6}\,\si{\metre}$, $h_{1} = 5\cdot10^{-6}\,\si{\metre}$ and $h_{2} = 3\cdot10^{-6}\,\si{\metre}$.}
	\label{fig:spiral2}
\end{figure}

\begin{benchmark}\label{bench:full}
In \cite{Becks_1992aa} a spiral inductor model with coplanar lines located 
on a substrate layer with an air bridge was proposed as a benchmark example for high-frequency problems. The CST Microwave tutorial discusses the 
same model to advocate the usage of 3D field simulation instead of circuit models \cite{CST_2016aa}. A slightly simplified geometry is illustrated in 
Figure \ref{fig:spiral}. The dimensions of the layer are $\SI{7e-4}{m} \times \SI{4.75e-4}{m} \times \SI{2.5e-5}{m}$ and Figure \ref{fig:spiral2} 
illustrates the dimensions of the coil. 

The bottom of the substrate layer is constrained by ebc and the other five boundaries are by mbc. On each side of the bridge the coil is connected by 
a straight line of perfect conductor with the ebc bottom plane. One side is excited by a discrete port which is given by a current source 
$i(t)=\sin(2\pi f t)\mathrm{A}$ with $f=\SI{50e9}{1/s}$. 
The coil ($\Omega_{\mathrm{pec}}$) is assumed to be a perfect conductor, i.e., modelled by 
homogeneous electric boundary conditions, the substrate ($\Omega_{\mathrm{sub}}$) is given a relative permittivity of $\varepsilon_\text{r}=12$, in the 
air region $\Omega_{\mathrm{air}}$ $\varepsilon_\text{r}=1$  and vacuum permeability $\mu = 4\pi\cdot$\SI{1e-7}{H/m} is assumed everywhere else.

 The 
structure is discretised using FIT with $406493$ mesh cells and $1,283,040$ degrees of freedom. Leapfrog is used with a time step of $\Delta 
t=\SI{3.4331e-15}{s}$ based on the CFL condition  and zero initial condition, see Fig.~\ref{fig:spiral_iv}. The performance of leapfrog and 
exponential integrators for this model was recently discussed in 
\cite{Merkel_2017ac}.
\end{benchmark}

\begin{figure}
  \centering
  \subfigure[Current through the spiral inductor]{
    \includegraphics{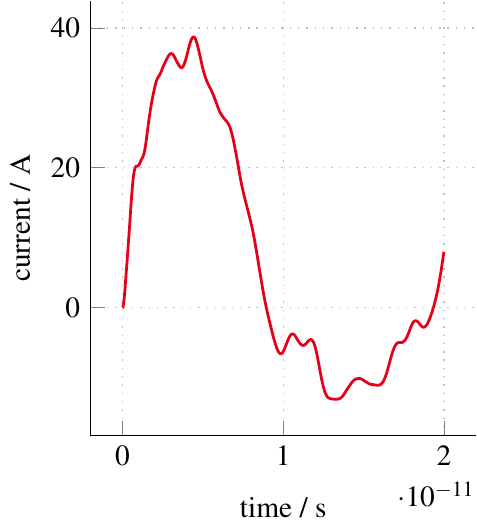}
      }
  \hspace{0.05\textwidth}
  \subfigure[voltage drop at the ports]{
    \includegraphics{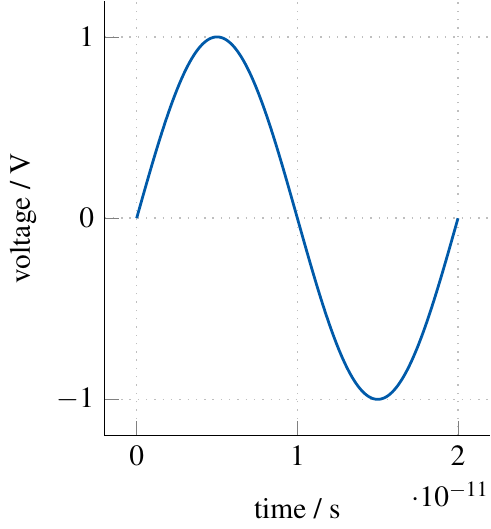}
      }
  \caption{Time domain simulation results for the Benchmark~\ref{bench:full}.\label{fig:spiral_iv}}
\end{figure}
\subsection{A-$\Phi$ formulations}
If there are small geometric features, slowly varying excitations, conducting or semiconducting materials \cite{Chen_2013aa}, the leapfrog scheme becomes inefficient. An alternative formulation is obtained if one rewrites Maxwell's equations as a second-order partial differential equation by combining Faraday's law, Amp\`ere's law and the material equations complemented by Gauss's law, i.e.,
\begin{align}
\boldsymbol{\varepsilon}\frac{\partial^2{}}{\partial{{t}}^2}\vec{E}
+
\boldsymbol{\sigma}\frac{\partial{}}{\partial{{t}}}\vec{E}
+
\nabla\times\nu\nabla\times\vec{E}
&=
\frac{\partial{}}{\partial{{t}}}\vec{J}_\text{s}.
\\
-\nabla\cdot\boldsymbol{\epsilon}\ensuremath{\vec{E}}
&=
\rho
\end{align}
The inconvenience of a time-differentiated source current density can be mitigated by exploiting the potentials as defined in \eqref{eq:a-phi}
\begin{align}\label{eq:curlcurl-a-phi}
\boldsymbol{\varepsilon}\frac{\partial^2{}}{\partial{{t}}^2}\vec{A}+\frac{\partial{}}{\partial{{t}}}\boldsymbol{\varepsilon}\nabla\phi
+
\boldsymbol{\sigma}\frac{\partial{}}{\partial{{t}}}\vec{A}+\boldsymbol{\sigma}\nabla\phi
+
\nabla\times\nu\nabla\times\vec{A}
&=
\vec{J}_\text{s}
\\
-\nabla\cdot\boldsymbol{\varepsilon}\frac{\partial{}}{\partial{{t}}}\vec{A}-\nabla\cdot\boldsymbol{\varepsilon}\nabla\phi
&=
\rho.
\end{align}
There is an ambiguity of the electromagnetic potentials since $\ensuremath{\vec{A}}$ is only fixed up to a gradient field \cite{Jackson_1998aa}. To this end, several gauging techniques have been introduced. For example grad-div formulations
that are based on the Coulomb gauge, have been introduced for low frequencies \cite{Bossavit_2001ab, Clemens_2002aa, Clemens_2011aa} and high-frequency applications \cite{Hahne_1992aa}. 

Let us define a regularisation of the electrodynamic potentials by the following gauge condition
\begin{align}
\label{eq:lorenz}
\xi_1\nabla\cdot\ensuremath{\vec{A}} + \xi_2 \phi  +\xi_3 \frac{\partial{}}{\partial{{t}}}\phi  = 0
\end{align}
which yields for $\xi_1=1$, $\xi_2=\xi_3=0$ the \emph{Coulomb} gauge and for $\xi_1=\nu$, $\xi_2=0$ and $\xi_3=\varepsilon$ the \emph{Lorenz} 
gauge if the considered materials are conducting, uniform, isotropic and linear. In the case $\xi_1=\nu$, $\xi_2=\sigma$ and $\xi_3=\varepsilon$ the curl-curl equation \eqref{eq:curlcurl-a-phi} can be written as a pair of damped wave equations 
\begin{align}
\label{eq:wave1}
\left[ \Delta - \mu\sigma\frac{\partial{}}{\partial{{t}}} -\mu\varepsilon\frac{\partial^2{}}{\partial{{t}}^2}\right]\ensuremath{\vec{A}} &= - \mu \ensuremath{\vec{J}}_\text{s}\\
\label{eq:wave2}
\left[ \Delta - \mu\sigma\frac{\partial{}}{\partial{{t}}} -\mu\varepsilon\frac{\partial^2{}}{\partial{{t}}^2}\right]\phi  &=  - \frac{\rho}{\varepsilon}
\end{align}
where $\Delta$ denotes the (scalar and vector) Laplace operators. In the undamped case ($\sigma=0$), this system reduces to the well-known d’Alembert equations, \cite{Jackson_1998aa}. The right-hand-sides are still coupled via the continuity equation \eqref{eq:continuity_diff}
\begin{align}
\label{eq:continuity}
\nabla\cdot\ensuremath{\vec{J}}_\text{s} + \frac{\sigma}{{\varepsilon}}\rho +\frac{\partial{}}{\partial{{t}}}\rho = 0
\end{align}  
where we have again exploited isotropy and homogeneity of $\sigma$ and $\ensuremath{\varepsilon}$ to obtain $-\nabla\cdot\sigma\ensuremath{\vec{E}}=\frac{\sigma}{\ensuremath{\varepsilon}}\rho$.
When solving the system \eqref{eq:wave1}-\eqref{eq:continuity}
we have to ensure that the (generalised) Lorenz gauge \eqref{eq:lorenz} is still fulfilled, which requires compatible boundary conditions for $\ensuremath{\vec{A}}$ and $\phi$ \cite{Baumanns_2013aa}.

\bigskip

Now, let us derive a similar semidiscrete formulation based on the spatial discretisation introduced above. We start with the $A-\phi$ formulation 
\eqref{eq:curlcurl-a-phi} using the discretised laws of Amp\`ere \eqref{eq:ampere_fit} and Gauss \eqref{eq:gauss_fit}:
\begin{align}
	\label{eq:curlcurl} 
	\!\!\!
	\tilde{\mathbf{C}}\mathbf{M}_{\nu}\mathbf{C}\protect\bow{\mathrm{\mathbf{a}}}\!+\!\mathbf{M}_{\sigma}\!\left[\frac{\mathrm{d}}{\mathrm{d}{}t}\,{}\protect\bow{\mathrm{\mathbf{a}}}+
	\mathbf{G}{\boldsymbol\Phi}\right]\!\! + \!
	\mathbf{M}_{\varepsilon}\!\left[\frac{\mbox{d}^2{}}{\mbox{d}{{t}}^2}\protect\bow{\mathrm{\mathbf{a}}}+\mathbf{G}\frac{\mathrm{d}}{\mathrm{d}{}t}\,{}{\boldsymbol\Phi}\right]\!
	&=\!\protect\bbow{\mathrm{\mathbf{j}}}_\mathrm{s}\;\!\\
	\label{eq:laplace}
	-\tilde{\mathbf{S}}\mathbf{M}_{\varepsilon}\frac{\mathrm{d}}{\mathrm{d}{}t}\,{}\protect\bow{\mathrm{\mathbf{a}}}+\mathbf{L}_{\varepsilon}{\boldsymbol\Phi}\!
	&=\!\mathrm{\mathbf{q}}\;\!
\end{align}
which contains the discrete Laplace operators
\begin{align}
	\label{eq:discrete_laplace} 
	\mathbf{L}_{\varepsilon}:=-\tilde{\mathbf{S}}\mathbf{M}_{\varepsilon}\mathbf{G}\quad\text{and}\quad
	\mathbf{L}_{\sigma}:=-\tilde{\mathbf{S}}\mathbf{M}_{\sigma}\mathbf{G}\;,
\end{align}
for permittivity and conductivity, respectively. 

\begin{lem}[Discrete Laplacians]
Let Assumptions~\ref{ass:material} and \ref{ass:nophantoms} hold true, the discrete Laplace operator $\mathbf{L}_{\varepsilon}$  in 
\eqref{eq:discrete_laplace} is symmetric positive definite and $\mathbf{L}_{\sigma}$ is symmetric positive semidefinite.
\end{lem}

\begin{proof}
	As we assume Dirichlet boundary conditions (ebc) in Assumption~\ref{ass:nophantoms} and $\mathbf{G}=-\tilde{\mathbf{S}}^{\top}$ due to 
	\eqref{eq:grad_div_fit}, the proof is straight forward. \qed
\end{proof} 

\medskip

Equations \eqref{eq:curlcurl} and \eqref{eq:laplace} are coupled by the potentials and right-hand-sides via the continuity equation
\begin{align}
	\label{eq:continuity_discrete}
	\tilde{\mathbf{S}}\protect\bbow{\mathrm{\mathbf{j}}}_\mathrm{s}
	+\mathbf{L}_{\sigma}\mathbf{L}_{\varepsilon}^{-1}\mathrm{\mathbf{q}}
	+\frac{\mathrm{d}}{\mathrm{d}{}t}\,{}\mathrm{\mathbf{q}}
	=\!\left[
	\tilde{\mathbf{S}}\mathbf{M}_{\sigma}
	-
	\mathbf{L}_{\sigma}\mathbf{L}_{\varepsilon}^{-1}
	\tilde{\mathbf{S}}\mathbf{M}_{\varepsilon}   
	\right]\!\frac{\mathrm{d}}{\mathrm{d}{}t}\,{}\protect\bow{\mathrm{\mathbf{a}}}\;,\!
\end{align}
that is obtained by a left multiplication of Amp\`ere’s law by $\tilde{\mathbf{S}}$ and inserting Gauss' law etc. The steps are the same as in the continuous case, e.g., applying the divergence operator.
Nonetheless, the discrete continuity equation \eqref{eq:continuity_discrete} is
more general than its continuous counterpart \eqref{eq:continuity} as it covers
anisotropic and non-homogeneous material distributions.

\bigskip

The ambiguity of the potentials is not yet fixed. The generalised discrete Lorenz gauge \eqref{eq:lorenz} for a conductive domain
in FIT notation is given by
\begin{align} 
\label{eq:mge_lorenz}
\mathbf{M}_{\varepsilon}\mathbf{G}\mathbf{M}_{\text{N}}\tilde{\mathbf{S}}\mathbf{M}_{\varepsilon}\protect\bow{\mathrm{\mathbf{a}}}
+\mathbf{M}_{\sigma}\mathbf{G}{\boldsymbol\Phi}
+\mathbf{M}_{\varepsilon}\mathbf{G}\frac{\mathrm{d}}{\mathrm{d}{}t}\,{}{\boldsymbol\Phi}
= 0\;
\end{align}
with a scaling matrix $\mathbf{M}_{\text{N}}$ which is mainly introduced to guarantee correct units. A consistent but rather inconvenient choice is
$$
 \mathbf{M}_{\text{N}}:=\mathbf{M}_{\varepsilon}^{-1/2}\mathbf{M}_{\nu}^{1/2}\mathbf{L}_{\varepsilon}^{-1}\mathbf{M}_{\nu}^{1/2}\mathbf{M}_{\varepsilon}^{-1/2}.
$$
This regularisation is similar to the Lagrange-multiplier formulation for the eddy-current problem, \cite{Clemens_2005aa}. 
Left-multiplication of \eqref{eq:mge_lorenz} by $\mathbf{M}_{\text{N}}^{-1}\mathbf{L}_{\varepsilon}^{-1}\tilde{\mathbf{S}}$ yields
\begin{align} 
	\label{eq:mge_lorenz2}
	\tilde{\mathbf{S}}\mathbf{M}_{\varepsilon}\protect\bow{\mathrm{\mathbf{a}}}
	+\mathbf{M}_{\text{N}}^{-1}\mathbf{L}_{\varepsilon}^{-1}\mathbf{L}_{\sigma}{\boldsymbol\Phi}
	+\mathbf{M}_{\text{N}}^{-1}\frac{\mathrm{d}}{\mathrm{d}{}t}\,{}{\boldsymbol\Phi}
	= 0.\;
\end{align}
which simplifies to Coulomb's gauge
\begin{align} 
	\label{eq:mge_coulomb}
	\tilde{\mathbf{S}}\mathbf{M}_{\varepsilon}\protect\bow{\mathrm{\mathbf{a}}}
    = 0
\end{align}
with respect to the permittivities if we set ${\boldsymbol\Phi}=0$.

\bigskip

To obtain a discrete version of the damped wave equation
\eqref{eq:wave1}-\eqref{eq:wave2}, we utilise \eqref{eq:mge_lorenz2}.
Now, using \eqref{eq:mge_lorenz} and \eqref{eq:mge_lorenz2} the system 
\eqref{eq:curlcurl}-\eqref{eq:laplace} becomes two discrete damped wave 
equations
\begin{align} \label{eq:wave_discrete1}
	\mathbf{L}_{\nu}\protect\bow{\mathrm{\mathbf{a}}}
	+ \mathbf{M}_{\sigma}\frac{\mathrm{d}}{\mathrm{d}{}t}\,{}\protect\bow{\mathrm{\mathbf{a}}} 
	+ \mathbf{M}_{\varepsilon} \frac{\mbox{d}^2{}}{\mbox{d}{{t}}^2}\protect\bow{\mathrm{\mathbf{a}}}
	&=\protect\bbow{\mathrm{\mathbf{j}}}_\mathrm{s}\\
	\label{eq:wave_discrete2}
	\mathbf{L}_{\varepsilon}{\boldsymbol\Phi}
	+\mathbf{M}_{\text{N}}^{-1}\mathbf{L}_{\varepsilon}^{-1}
	\mathbf{L}_{\sigma}\frac{\mathrm{d}}{\mathrm{d}{}t}\,{}{\boldsymbol\Phi}
	+\mathbf{M}_{\text{N}}^{-1}\frac{\mbox{d}^2{}}{\mbox{d}{{t}}^2}{\boldsymbol\Phi} &=\mathrm{\mathbf{q}}
\end{align}
with $\mathbf{L}_{\nu} := \widetilde{\mathbf{C}}\mathbf{M}_{\nu}\mathbf{C} - \mathbf{M}_{\varepsilon}\mathbf{G}\mathbf{M}_\mathrm{N}\widetilde{\mathbf{S}}\mathbf{M}_{\varepsilon}$ and given right-hand-sides $\protect\bbow{\mathrm{\mathbf{j}}}_\mathrm{s}$ 
and $\mathrm{\mathbf{q}}$ that fulfil the continuity equation 
\eqref{eq:continuity_discrete}. The resulting problem \eqref{eq:wave_discrete1}-\eqref{eq:wave_discrete2}
is a system of second-order ordinary differential equations:

\begin{theorem}
Let Assumptions~\ref{ass:dom}, \ref{ass:material} and \ref{ass:nophantoms} hold. Then, the A-$\Phi$-formulation with Lorenz gauge 
\eqref{eq:mge_lorenz2} and known charges $\mathrm{\mathbf{q}}$
leads to an ordinary differential equation (ODE) system which is given in \eqref{eq:wave_discrete1}-\eqref{eq:wave_discrete2}.
\end{theorem}

\subsubsection{Full Maxwell with Lorenz Gauge}
Let us now investigate the case where the charges $\mathrm{\mathbf{q}}$ are not known. We start from Lorenz' gauge \eqref{eq:mge_lorenz}. Left-multiplication of the 
equation by $-\tilde{\mathbf{S}}$ yields
\begin{align*} 
\mathbf{L}_{\varepsilon}\mathbf{M}_{\text{N}}\tilde{\mathbf{S}}\mathbf{M}_{\varepsilon}\protect\bow{\mathrm{\mathbf{a}}}
+\mathbf{L}_{\sigma}{\boldsymbol\Phi}
+\mathbf{L}_{\varepsilon}\frac{\mathrm{d}}{\mathrm{d}{}t}\,{}{\boldsymbol\Phi}
= 0\;.
\end{align*}
Following the notation of Schoenmaker, e.g. \cite{Schoenmaker_2017aa}, we denote the derivative of the magnetic vector potential by $\bow{\mathbf{\pi}} := \mathrm{d}\protect\bow{\mathrm{\mathbf{a}}}/\mathrm{d}t$. Then, the equations \eqref{eq:curlcurl}-\eqref{eq:laplace} can be rearranged 
as the following system of DAEs 
\begin{align} 
	\label{eq:loga-1}
	\mathbf{L}_{\varepsilon}\mathbf{M}_{\text{N}}\tilde{\mathbf{S}}\mathbf{M}_{\varepsilon}\protect\bow{\mathrm{\mathbf{a}}}
	+\mathbf{L}_{\sigma}{\boldsymbol\Phi}
	+\mathbf{L}_{\varepsilon}\frac{\mathrm{d}}{\mathrm{d}{}t}\,{}{\boldsymbol\Phi}
	& = 0\\
	\label{eq:loga-2}
	\tilde{\mathbf{C}}\mathbf{M}_{\nu}\mathbf{C}\protect\bow{\mathrm{\mathbf{a}}}\!+\!\mathbf{M}_{\sigma}\!\left[\bow{\mathbf{\pi}}+
	\mathbf{G}{\boldsymbol\Phi}\right]\!\! + \!
	\mathbf{M}_{\varepsilon}\!\left[\frac{\mathrm{d}}{\mathrm{d}{}t}\,{}\bow{\mathbf{\pi}}+\mathbf{G}\frac{\mathrm{d}}{\mathrm{d}{}t}\,{}{\boldsymbol\Phi}\right]\!
	&=\!\protect\bbow{\mathrm{\mathbf{j}}}_\mathrm{s}\!\\
	\label{eq:loga-3}
	\tilde{\mathbf{S}}\mathbf{M}_{\varepsilon}\bow{\mathbf{\pi}}-\mathbf{L}_{\varepsilon}{\boldsymbol\Phi}\!
	+ \!\mathrm{\mathbf{q}}\!&= 0\\
	\label{eq:loga-4}
	\frac{\mathrm{d}}{\mathrm{d}{}t}\,{}\protect\bow{\mathrm{\mathbf{a}}} - \bow{\mathbf{\pi}} &= 0
\end{align}
with $\mathbf{x}^{\top}=(\mathrm{\mathbf{q}}^{\top}, {\boldsymbol\Phi}^{\top}, \protect\bow{\mathrm{\mathbf{a}}}^{\top}, \bow{\mathbf{\pi}}^{\top})$ such that we can write \eqref{eq:loga-1}-\eqref{eq:loga-4} in the form of \eqref{eq:dae} with the definitions
$$
\mathbf{M}=
\begin{bmatrix}
0 & \mathbf{L}_{\varepsilon} & 0 & 0\\
0 & \mathbf{M}_{\varepsilon}\mathbf{G} & 0 & \mathbf{M}_{\varepsilon}\\
0 & 0 & 0 & 0\\
0 & 0 & \mathbf{I} & 0\\
\end{bmatrix}
,
\qquad
\mathbf{K}=
\begin{bmatrix}
0 & \mathbf{L}_{\sigma} & \mathbf{L}_{\varepsilon}\mathbf{M}_{\text{N}}\tilde{\mathbf{S}}\mathbf{M}_{\varepsilon} & 0\\
0 & \mathbf{M}_{\sigma}\mathbf{G} & \tilde{\mathbf{C}}\mathbf{M}_{\nu}\mathbf{C} & \mathbf{M}_{\sigma}\\
\mathbf{I} & -\mathbf{L}_{\varepsilon} & 0 & \tilde{\mathbf{S}}\mathbf{M}_{\varepsilon}\\
0 & 0 & 0 & -\mathbf{I}\\
\end{bmatrix}
\quad
\text{and}
\quad
\mathbf{r}=
\begin{bmatrix}
0\\
\protect\bbow{\mathrm{\mathbf{j}}}_\mathrm{s}\\
0\\
0\\
\end{bmatrix}.
$$
Now, any standard time integrator, e.g. the implicit Euler method \eqref{eq:euler}, can be applied.

\bigskip

Next we determine the
differential index of the system \eqref{eq:loga-1}-\eqref{eq:loga-4}.
Equation \eqref{eq:loga-1} is an ODE for ${\boldsymbol\Phi}$
\begin{align}
	\label{eq:ode_phi_lorenz}
	\frac{\mathrm{d}}{\mathrm{d}{}t}\,{}{\boldsymbol\Phi}
	=
	-\mathbf{M}_{\text{N}}\tilde{\mathbf{S}}\mathbf{M}_{\varepsilon}\protect\bow{\mathrm{\mathbf{a}}}
	-\mathbf{L}_{\varepsilon}^{-1}\mathbf{L}_{\sigma}{\boldsymbol\Phi}.
\end{align}
Then, we deduce from \eqref{eq:loga-2} and \eqref{eq:ode_phi_lorenz} an ODE for $\bow{\mathbf{\pi}}$:
\begin{align*}
	\frac{\mathrm{d}}{\mathrm{d}{}t}\,{}\bow{\mathbf{\pi}} 
	& = -\mathbf{M}_{\varepsilon}^{-1}\left[\mathbf{L}_{\nu}\protect\bow{\mathrm{\mathbf{a}}}\!+\!\mathbf{M}_{\sigma}\!\left[\bow{\mathbf{\pi}}+
	\mathbf{G}{\boldsymbol\Phi}\right]\!\! - \!
	\mathbf{M}_{\varepsilon}\mathbf{G}\mathbf{L}_{\varepsilon}^{-1}\mathbf{L}_{\sigma}{\boldsymbol\Phi}
	- \protect\bbow{\mathrm{\mathbf{j}}}_\mathrm{s}\right]
\end{align*}
Finally, only one differentiation with respect to time of \eqref{eq:loga-3} is needed 
to obtain an ordinary differential equation for $\mathrm{\mathbf{q}}$:
\begin{align*}
	\frac{\mathrm{d}}{\mathrm{d}{}t}\,{}\mathrm{\mathbf{q}} 
	&= \tilde{\mathbf{S}}\mathbf{M}_{\sigma}\bow{\mathbf{\pi}}
	  -\mathbf{L}_{\sigma}{\boldsymbol\Phi}
      - \tilde{\mathbf{S}}\protect\bbow{\mathrm{\mathbf{j}}}_\mathrm{s}\;.
\end{align*}
Hence we conclude the following result \cite{Baumanns_2013aa}
\begin{theorem}
  Let Assumptions~\ref{ass:dom}, \ref{ass:material} and  \ref{ass:nophantoms} hold. The system 
  \eqref{eq:loga-1}-\eqref{eq:loga-4} has differential index-$1$
	and the initial vector $\mathbf{x}_0^{\top}=(\mathrm{\mathbf{q}}_0^{\top}, {\boldsymbol\Phi}_0^{\top}, \protect\bow{\mathrm{\mathbf{a}}}_0^{\top}, \bow{\mathbf{\pi}_0}^{\top})$ 
	is a consistent initial value if 
		$\mathrm{\mathbf{q}}_0 = \mathbf{L}_{\varepsilon}{\boldsymbol\Phi}_0
        -\tilde{\mathbf{S}}\mathbf{M}_{\varepsilon}\bow{\mathbf{\pi}_0}$
	is fulfilled.
\end{theorem}
\subsubsection{Full Maxwell with Coulomb Gauge}
Instead of augmenting the equations by a Lorenz gauge, one can choose the Coulomb gauge \eqref{eq:mge_coulomb}.
Starting by left-multiplying Coulomb's gauge by $\mathbf{M}_{\varepsilon}\mathbf{G}\mathbf{M}_{\text{N}}$, we obtain
\begin{align} 
	\label{eq:mge_coulomb_mod}
	\mathbf{M}_{\varepsilon}\mathbf{G}\mathbf{M}_{\text{N}}\tilde{\mathbf{S}}\mathbf{M}_{\varepsilon}\protect\bow{\mathrm{\mathbf{a}}}
	= 0\;.
\end{align} 
Using \eqref{eq:mge_coulomb} and \eqref{eq:mge_coulomb_mod}, the system
\eqref{eq:curlcurl}-\eqref{eq:laplace} becomes a semi-discrete damped wave 
equation accompanied by a Laplace equation, i.e.,
\begin{align*}
	\!\!\!
	\mathbf{L}_{\nu}\protect\bow{\mathrm{\mathbf{a}}}\!+\!\mathbf{M}_{\sigma}\!\left[\frac{\mathrm{d}}{\mathrm{d}{}t}\,{}\protect\bow{\mathrm{\mathbf{a}}}+
	\mathbf{G}{\boldsymbol\Phi}\right]\!\! + \!
	\mathbf{M}_{\varepsilon}\!\left[\frac{\mbox{d}^2{}}{\mbox{d}{{t}}^2}\protect\bow{\mathrm{\mathbf{a}}}+\mathbf{G}\frac{\mathrm{d}}{\mathrm{d}{}t}\,{}{\boldsymbol\Phi}\right]\!
	&=\!\protect\bbow{\mathrm{\mathbf{j}}}_\mathrm{s}\;\!\\
	\mathbf{L}_{\varepsilon}{\boldsymbol\Phi}\!
	&=\!\mathrm{\mathbf{q}}\!
\end{align*}
with right-hand-sides that fulfil the continuity equation \eqref{eq:continuity_discrete} and
thus for given $\protect\bbow{\mathrm{\mathbf{j}}}_\mathrm{s}$ the resulting semi-discrete problem is again a system of
DAEs.
The Coulomb-gauged system reads
\begin{align} 
	\label{eq:coga-1}
	\tilde{\mathbf{S}}\mathbf{M}_{\varepsilon}\protect\bow{\mathrm{\mathbf{a}}}
	& = 0\\
	\label{eq:coga-2}
	\tilde{\mathbf{C}}\mathbf{M}_{\nu}\mathbf{C}\protect\bow{\mathrm{\mathbf{a}}}\!+\!\mathbf{M}_{\sigma}\!\left[\bow{\mathbf{\pi}}+
	\mathbf{G}{\boldsymbol\Phi}\right]\!\! + \!
	\mathbf{M}_{\varepsilon}\!\left[\frac{\mathrm{d}}{\mathrm{d}{}t}\,{}\bow{\mathbf{\pi}}+\mathbf{G}\frac{\mathrm{d}}{\mathrm{d}{}t}\,{}{\boldsymbol\Phi}\right]\!
	&=\!\protect\bbow{\mathrm{\mathbf{j}}}_\mathrm{s}\!\\
	\label{eq:coga-3}
	\tilde{\mathbf{S}}\mathbf{M}_{\varepsilon}\bow{\mathbf{\pi}}-\mathbf{L}_{\varepsilon}{\boldsymbol\Phi}\!
	+ \!\mathrm{\mathbf{q}}\!&= 0\\
	\label{eq:coga-4}
	\frac{\mathrm{d}}{\mathrm{d}{}t}\,{}\protect\bow{\mathrm{\mathbf{a}}} -\bow{\mathbf{\pi}} &= 0
\end{align}
with $\mathbf{x}^{\top}=(\mathrm{\mathbf{q}}^{\top}, {\boldsymbol\Phi}^{\top}, \protect\bow{\mathrm{\mathbf{a}}}^{\top}, \bow{\mathbf{\pi}}^{\top})$. Similarly as before we can identify a first order DAE 
system of form \eqref{eq:dae} and apply for example the implicit Euler method.

\bigskip

Next, we determine the
differential index of the system \eqref{eq:coga-1}-\eqref{eq:coga-4}.
Differentiating \eqref{eq:coga-1} twice with respect to time and inserting \eqref{eq:coga-4} leads to
\begin{align}
	\label{eq:coga-5}
	\tilde{\mathbf{S}}\mathbf{M}_{\varepsilon}\frac{\mathrm{d}}{\mathrm{d}{}t}\,{}\bow{\mathbf{\pi}} = 0\;.
\end{align}
This indicates already that the differential index is at least $\vartheta\geq2$.
Left-multiplying \eqref{eq:coga-2} by $\tilde{\mathbf{S}}$
and applying \eqref{eq:coga-5} yields:
\begin{align*}
	\frac{\mathrm{d}}{\mathrm{d}{}t}\,{}{\boldsymbol\Phi} = -\mathbf{L}_{\varepsilon}^{-1}\!\left[
	\mathbf{L}_{\sigma}{\boldsymbol\Phi}
	- \tilde{\mathbf{S}}\mathbf{M}_{\sigma}\bow{\mathbf{\pi}}
	+ \tilde{\mathbf{S}}\protect\bbow{\mathrm{\mathbf{j}}}_\mathrm{s}
	\right]\!\;
\end{align*}
Furthermore, from \eqref{eq:coga-2}, we obtain:
\begin{align*}
\frac{\mathrm{d}}{\mathrm{d}{}t}\,{}\bow{\mathbf{\pi}} 
& = 
-\mathbf{M}_{\varepsilon}^{-1}\!\left[ (\mathbf{M}_{\sigma}\mathbf{G}-\mathbf{M}_{\varepsilon}\mathbf{G}\mathbf{L}_{\varepsilon}^{-1}\mathbf{L}_{\sigma}){\boldsymbol\Phi}
+\tilde{\mathbf{C}}\mathbf{M}_{\nu}\mathbf{C}\protect\bow{\mathrm{\mathbf{a}}}\right. \\
& \hphantom{=}\left. +(\mathbf{M}_{\sigma}+ \mathbf{M}_{\varepsilon}\mathbf{G}\mathbf{L}_{\varepsilon}^{-1}\tilde{\mathbf{S}}\mathbf{M}_{\sigma})\bow{\mathbf{\pi}}
-(\mathbf{I} +\mathbf{M}_{\varepsilon}\mathbf{G}\mathbf{L}_{\varepsilon}^{-1}\tilde{\mathbf{S}} )\protect\bbow{\mathrm{\mathbf{j}}}_\mathrm{s}
\right]
\end{align*}
Finally, one differentiation with respect to time of \eqref{eq:coga-3} results in
\begin{align*}	
	\frac{\mathrm{d}}{\mathrm{d}{}t}\,{}\mathrm{\mathbf{q}}
	&= \tilde{\mathbf{S}}\!\mathbf{M}_{\sigma}\bow{\mathbf{\pi}}
	  - \mathbf{L}_{\sigma}{\boldsymbol\Phi}
      - \tilde{\mathbf{S}}\protect\bbow{\mathrm{\mathbf{j}}}_\mathrm{s}\;
\end{align*}
and thus the overall problem has a differential index-$2$  \cite{Baumanns_2013aa}.
\begin{theorem}
    Let Assumptions~\ref{ass:dom}, \ref{ass:material} and \ref{ass:nophantoms} hold. The system \eqref{eq:coga-1}-\eqref{eq:coga-4} has differential 
    index-$2$	and the initial vector 
    $\mathbf{x}_0^{\top}=(\mathrm{\mathbf{q}}_0^{\top}, 
  {\boldsymbol\Phi}_0^{\top}, \protect\bow{\mathrm{\mathbf{a}}}_0^{\top}, \bow{\mathbf{\pi}_0}^{\top})$ is a consistent initial value if
		$\tilde{\mathbf{S}}\mathbf{M}_{\varepsilon}\protect\bow{\mathrm{\mathbf{a}}}_0 = 0$, 
		$\tilde{\mathbf{S}}\mathbf{M}_{\varepsilon}\bow{\mathbf{\pi}_0} = 0$
        and
		$\mathrm{\mathbf{q}}_0=\mathbf{L}_{\varepsilon}{\boldsymbol\Phi}_0-\tilde{\mathbf{S}}\mathbf{M}_{\varepsilon}\bow{\mathbf{\pi}_0}$
	are fulfilled.
\end{theorem}

Lorenz' and Coulomb's gauge lead to systems that describe the same
phenomena and have eventually, i.e. in the mesh size limit, the same electromagnetic fields (strengths or fluxes) as solutions. 
On the other hand, the
structural properties are different, i.e., the Lorenz gauge yields an index-1
problem whereas the Coulomb gauge gives index-2. Hence, the latter formulation will be
much more affected by perturbations and the computation of consistent initial values is more cumbersome. This has been observed in simulations \cite{Baumanns_2013aa,Baumanns_2012ab}.

\begin{figure}[t]
	\centering
	\includegraphics[width=.5\textwidth]{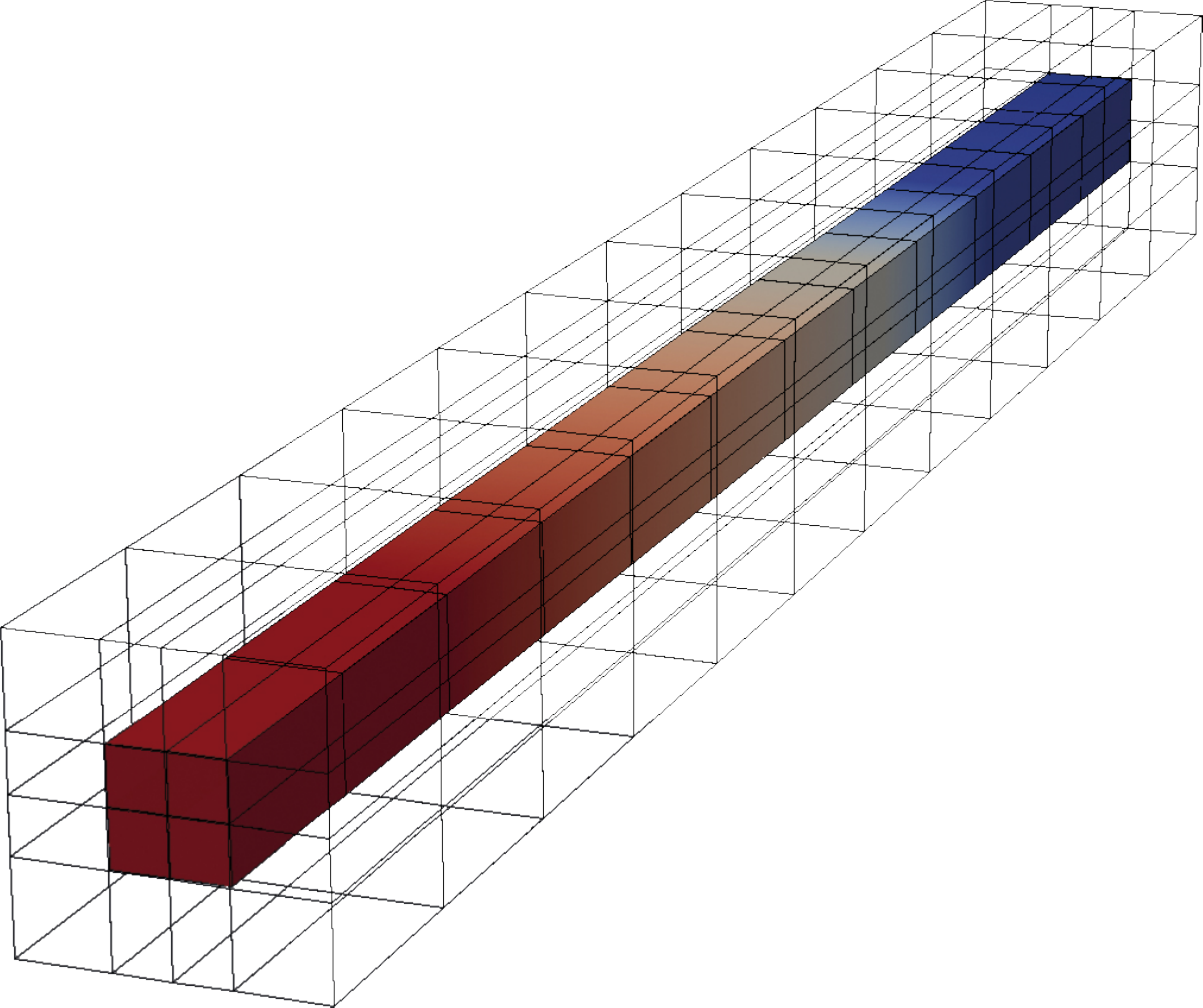}
	\caption{Copper bar in air, excited by sinusoidal source (Benchmark \ref{bench:sascha})}
	\label{fig:sascha}
\end{figure}

\begin{figure}[t]
	\centering
	\subfigure[Squared cross-section of the copper bar at the z=0 plane.]{
		\includegraphics{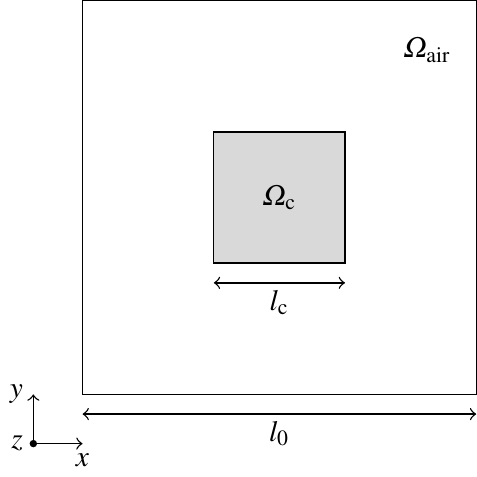}
			}
	\hspace{6em}
	\subfigure[Cross-section of the copper bar at the x=0 plane.]{
		\includegraphics{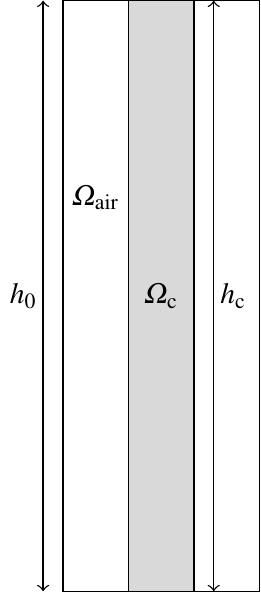}
			}
	\caption{Copper bar with square cross-section surrounded by air. The distances are \mbox{$l_\mathrm{c} = (1+\sqrt{1.5})\cdot$\SI{e-3}{m}}, $l_0 = 
	(3+\sqrt{1.5})\cdot$\SI{e-3}{m}, 
	and $h_0 = h_\mathrm{c} = $\SI{3}{m}.}
	\label{fig:sascha2}
\end{figure}

\begin{benchmark}\label{bench:sascha}
The benchmark example Fig.~\ref{fig:sascha} was proposed in \cite{Baumanns_2013aa} to numerically analyse the DAE index of the two gauged $A-\phi$ 
formulations. The model is a copper bar with a cross-sectional area of $0.25$ mm$^2$ surrounded by air and discretised by FIT. A detailed 
characterisation of the dimensions can be seen in Figure \ref{fig:sascha2}. 

On the copper bar $\Omega_{\mathrm{c}}$, a conductivity of $\sigma_\mathrm{c} = $\SI{5.7e7}{S/m} is set and on the air region $\sigma_\mathrm{air} = 
0\,  \mathrm{S/m}$. Vacuum permeability $\mu = 4\pi\cdot$\SI{e-7}{H/m} and relative permittivity $\varepsilon_\text{r}=1$ is assumed in the entire 
domain. 

One contact is excited by a sinusoidal voltage $v=\sin(2\pi t)\SI{}{V}$, the other contact is grounded~(ebc) and the remaining boundary is 
set to mbc. 

The structure is discretised using FIT with $325$ mesh cells and $845$ degrees of freedom. The implicit Euler method is applied with a time step of 
$\Delta t=$\SI{1e-4}{s} and zero initial condition, see Fig.~\ref{fig:sascha_iv}.
\end{benchmark}

\begin{figure}
  \centering
  \subfigure[Current through the bar]{
    \includegraphics{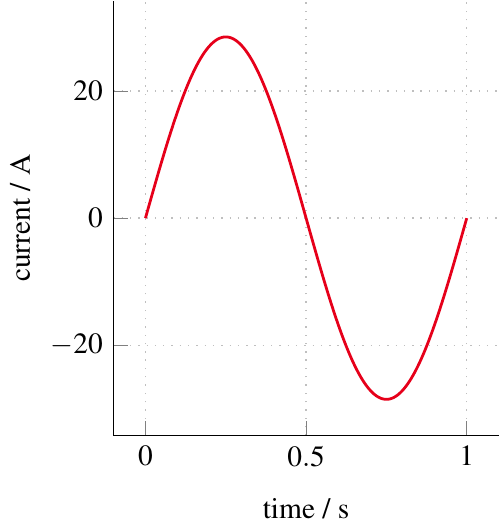}
      }
  \hspace{0.05\textwidth}
  \subfigure[voltage drop at the ports]{
    \includegraphics{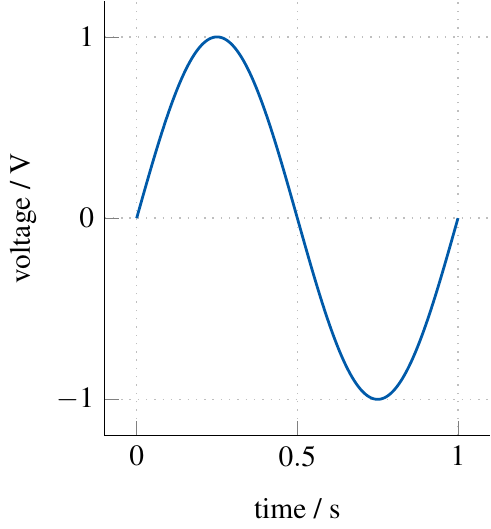}
      }
  \caption{Time domain simulation results for the Benchmark~\ref{bench:sascha} with Lorenz gauge $\Delta t=\SI{1e-4}{s}$\label{fig:sascha_iv}.}
\end{figure} 
\section{Quasistatic Maxwell's Equations}\label{sec:qs}
In the case of slowly time-varying fields, certain time-derivatives of Maxwell's equations can be disregarded with respect to other phenomena, see Definition~\ref{def:qs}. This is convenient to simplify the numerical treatment. However, the resulting (quasi-) static approximations have different structural properties and their differential algebraic index is studied next.

\subsection{Electroquasistatic Maxwell's Equations}
We start with the index study of the electroquasistatic approximation, that is given in Definition 
\hyperref[eqs]{\ref*{def:qs} \ref*{eqs}}. Maxwell's equations can be rewritten as
\begin{align*}
	  \nabla\times\ensuremath{\vec{E}} &= 0\;,&
	\nabla\times\ensuremath{\vec{H}} &= \frac{\partial\ensuremath{\vec{D}}}{\partial t} + \ensuremath{\vec{J}}\;, &
	\nabla\cdot\ensuremath{\vec{D}}  &= \ensuremath{\rho}\;, &
	\nabla\cdot\ensuremath{\vec{B}}  &= 0\;.
\end{align*}
As the curl of electric field $\ensuremath{\vec{E}}$ vanishes, it can be described as the gradient of the electric scalar potential $\phi$ 
\begin{equation}\label{eq:ESP}
	\ensuremath{\vec{E}} = - \nabla \phi\;,
\end{equation}
i.e., the magnetic vector potential's contribution to $\ensuremath{\vec{E}}$ in \eqref{eq:a-phi} is negligible.

\subsubsection{Electric Scalar Potential $\phi$-Formulation}
Using equation \eqref{eq:ESP}, Maxwell's equations for electroquasistatic fields and the material laws, the following potential equation can be 
obtained to compute $\phi$
\begin{equation*}
	\nabla \cdot \boldsymbol{\sigma}\nabla \phi + \frac{\partial}{\partial{}t}\,{}\nabla \cdot \boldsymbol{\varepsilon} \nabla \phi = 0\;. 
\end{equation*}
Eventually, spatial discretisation leads to a system of DAEs 
\begin{equation}\label{eq:eqsphi}
	\widetilde{\mathbf{S}}\mathbf{M}_{\sigma}\widetilde{\mathbf{S}}^{\top}{\boldsymbol\Phi} + \widetilde{\mathbf{S}}\mathbf{M}_{\varepsilon}\widetilde{\mathbf{S}}^{\top}\frac{\mathrm{d}}{\mathrm{d}t}{\boldsymbol\Phi} = 0\;,
\end{equation}
where ${\boldsymbol\Phi}$ contains the degrees of freedom of our problem, i.e. the electric scalar potential on the nodes of the primal grid. 
However, if boundary conditions are properly set, then one can show
\begin{theorem}\label{theorem:eqs}
	The system \eqref{eq:eqsphi} under Assumptions~\ref{ass:dom}, \ref{ass:material} and \ref{ass:nophantoms} is an ODE.
\end{theorem}
Theorem \ref{theorem:eqs} follows immediately from Lemma \ref{prop:material}.

\begin{figure}
	\centering
	\includegraphics{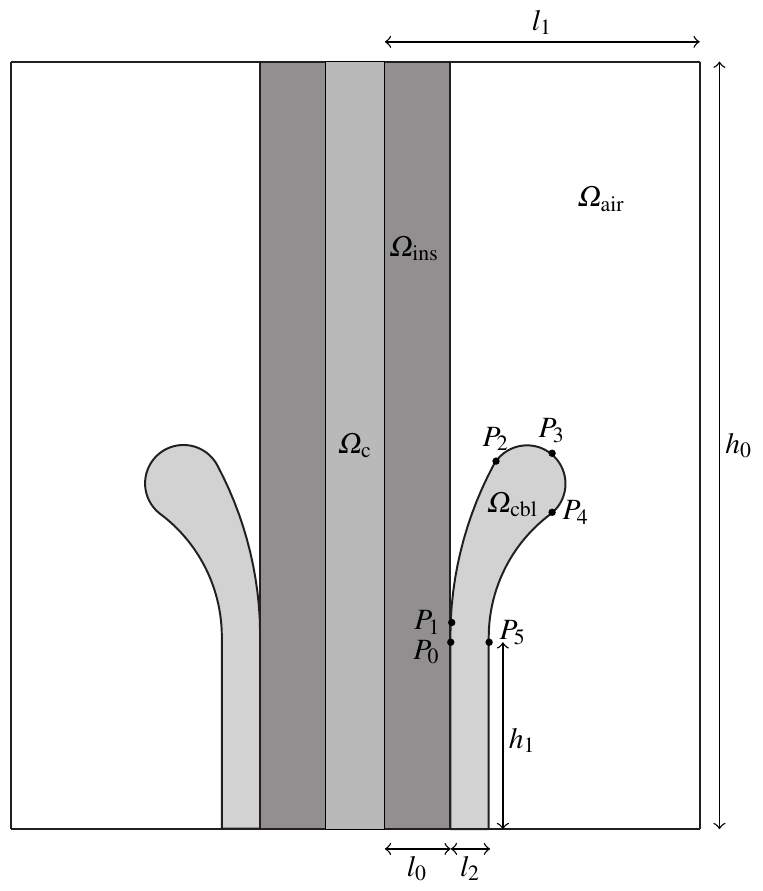}
		\caption{Sketch of electroquasistatic benchmark domain. The distances are $l_0 = $\SI{1.2e-2}{m}, $l_1 = $\SI{2.5e-2}{m}, $l_2 = $\SI{2e-3}{m}, 
	$h_0 = $\SI{4e-2}{m} and $h_1 = 
	$\SI{1e-2}{m}. To describe the arc segments, points $P_0 =(12, \ 10)$, $P_1 = (12.03, \ 11)$, $P_2 =(14.24, \ 18.94)$, $P_3 =(17.3, \ 19.52)$, 
	$P_4 = (17.2, \ 16,4)$ and $P_5 =(14, \ 10)$ are 
	defined. The first arc segment from point $P_2$ to $P_1$ has an angle of \ang{25.06} and is described by straight segments with a grid spacing of 
	\ang{3.02}. Both arc segments from $P_3$ to $P_2$ as well as from $P_4$ to $P_3$ have \ang{102.53} and \ang{28.65} spacing. The last one from 
	$P_4$ to $P_5$ has \ang{53.13} and \ang{7.16} spacing.}\label{fig:eqs_sketch}
\end{figure}
\begin{benchmark}\label{bench:eqs}
	In a DC high-voltage cable, the insulation between the inner high-voltage electrode and the outer shielding layer carries a large electric field 
	strength. At the end of the cable ($\Omega_{\mathrm{cbl}}$), the voltage has to drop along the surface of the insulation layer 
	($\Omega_{\mathrm{ins}}$) with a 
	substantially smaller electric field 
	strength. This necessitates the design of a so-called cable termination with field-shaping capability. A sketch of the domain and its distances 
	can be seen in 
	Fig.~\ref{fig:eqs_sketch}.
	
	The computational domain is $\Omega = \Omega_{\mathrm{air}}\cup\Omega_{\mathrm{ins}}$ and the rest of the domain is modelled via boundary  
	conditions and thus not considered by the discretisation.
	In the air region $\Omega_{\mathrm{air}}$ the conductivity $\sigma$ is set to zero  and the permittivity of vacuum $\varepsilon_0 = 
	$\SI{8.85e-12}{F/m} is assumed. The insulating domain $\Omega_{\mathrm{ins}}$ has conductivity \SI{1e0}{S/m} and permittivity  
	$6\varepsilon_0$. 
	
	Due to symmetry reasons, only an axisymmetric cross-section (i.e. the right half of the domain sketched 
	in Fig.~\ref{fig:eqs_sketch}) is 
	simulated.
	The cable endings $\Omega_{\mathrm{cbl}}$ are modelled by zero Dirichlet boundary conditions (ebc)  on the  
	boundary $\Gamma_{\mathrm{cbl}} = \partial \Omega_{\mathrm{cbl}}\cap \left(\overline{\Omega}_\mathrm{ins}\cup 
	\overline{\Omega}_\mathrm{air}\right)$. Similarly, the domain $\Omega_{\mathrm{c}}$ is modelled by non-homogeneous Dirichlet boundary conditions 
	that set the potential $\phi$ to a time dependent value $f(t)$ on $\Gamma_\mathrm{c} = \partial \Omega_{\mathrm{ins}}\cap 
	\overline{\Omega}_\mathrm{c}$, see Fig.~\ref{fig:eqs_sol}a. At the rest of the boundary zero Neumann boundary conditions (mbc) are set. 
	
	Using the Finite Element Method yields 2892 
	number of nodes and 
	2078 degrees of freedom. Time integration is carried out with the implicit Euler method from time $t_0 =$\SI{0}{s} to $t_\mathrm{end} = 
	$\SI{2e-3}{s} with step size $\Delta t =$\SI{1e-5}{s}. The steady state solution is set as initial condition.
	Fig.~\ref{fig:eqs_sol} shows the excitation function $f(t)$ and the electric energy $E_\mathrm{elec} = 
	\frac{1}{2}\int_{\Omega}\ensuremath{\vec{E}}\cdot\ensuremath{\vec{D}}\;\mathrm{d}\Omega\approx\frac{1}{2}\protect\bow{\mathrm{\mathbf{e}}}^{\top}\mathbf{M}_{\varepsilon}\protect\bow{\mathrm{\mathbf{e}}}$ over time.
\end{benchmark}

\begin{figure}
	\centering
	\subfigure[Excitation function f(t).]{
		\includegraphics{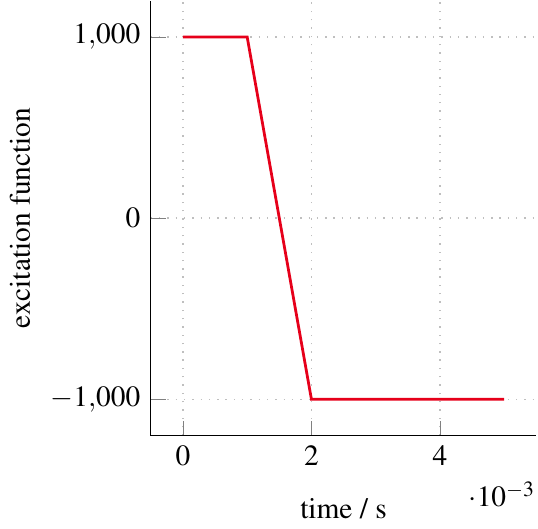}
		\label{fig:Wmagn}}
	\subfigure[Electric energy.]{
		\includegraphics{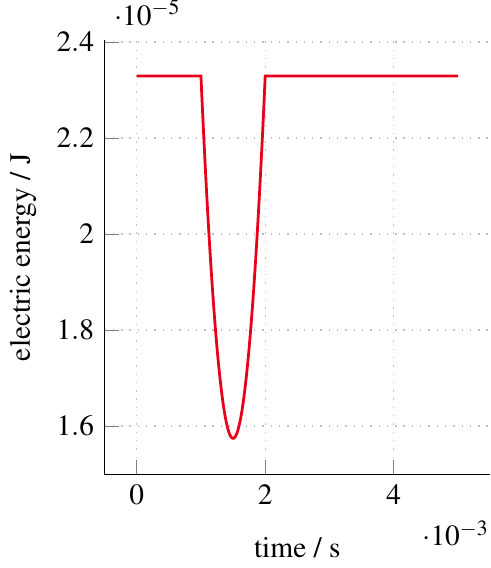}
		\label{fig:Uelec}}
	
	\caption{Electric energy and excitation function of benchmark~\ref{bench:eqs}.\label{fig:eqs_sol}}
	
\end{figure}

For the electroquasistatic problem other formulations are not common as \eqref{eq:eqsphi} has convenient properties, e.g. a low-number of degrees of freedom, since no vectorial fields are needed and ordinary differential character. Rarely, a mixed charge/potential formulation \cite{Ouedraogo_2017aa}
\begin{align*}
	\widetilde{\mathbf{S}}\mathbf{M}_{\sigma}\widetilde{\mathbf{S}}^{\top}{\boldsymbol\Phi} + \frac{\mathrm{d}}{\mathrm{d}t}\mathrm{\mathbf{q}} &= 0\;,\\
	\mathrm{\mathbf{q}}-\widetilde{\mathbf{S}}\mathbf{M}_{\varepsilon}\widetilde{\mathbf{S}}^{\top}{\boldsymbol\Phi}&=0\;. 
\end{align*}
is employed which is a simple and easy to solve DAE index-1 system.

\subsection{Magnetoquasistatic Maxwell's Equations}
Now the magnetoquasistatic case is studied. Following Definition \hyperref[mqs]{\ref*{def:qs} \ref*{mqs}} Maxwell's equations take the form  
\begin{align*}
\nabla\times\ensuremath{\vec{E}} &=  -\frac{\partial\ensuremath{\vec{B}}}{\partial t}\;,&
\nabla\times\ensuremath{\vec{H}} &= \ensuremath{\vec{J}}\;, &
\nabla\cdot\ensuremath{\vec{D}}  &= \ensuremath{\rho}\;, &
\nabla\cdot\ensuremath{\vec{B}}  &= 0\;.
\end{align*}
Due to the non-vanishing curl of the electric field strength, a scalar potential formulation is no longer possible. Several competing vector potential formulations are common, see e.g. \cite{Biro_1995aa}.

\subsubsection{Magnetic Vector Potential A-Formulations}
Using the definition of the magnetic vector and the electric scalar potentials $\ensuremath{\vec{A}}$  and $\phi$ (see equation \ref{eq:a-phi}) and inserting the material laws and magnetoquasistatic equations into each other, one finds the curl-curl equation
\begin{equation*}
	\boldsymbol{\sigma} \left(\frac{\partial\ensuremath{\vec{A}}}{\partial t}+\nabla\phi\right) + \nabla\times\left(\boldsymbol{\nu}\nabla\times\ensuremath{\vec{A}}\right) = 
	\ensuremath{\vec{J}}_\mathrm{s}\;.
\end{equation*}
In a three dimensional domain, a gauge condition is necessary to ensure uniqueness of solution, due to the kernel of the curl-operator ($\mathbf{C} 
\widetilde{\mathbf{S}}^{\top} = 0$). The so-called A* formulation exploits this freedom of choice and assumes that the gradient of the electric scalar potential is 
zero 
($\nabla \phi = 
0$). Applying it yields the spatially discretised magnetoquasistatic curl-curl equation
\begin{equation}\label{eq:c-c}
	\mathbf{M}_{\sigma} \frac{\mathrm{d}}{\mathrm{d}t}\protect\bow{\mathrm{\mathbf{a}}} + \mathbf{C}^{\top}\mathbf{M}_{\nu}\mathbf{C}\protect\bow{\mathrm{\mathbf{a}}} = \protect\bbow{\mathrm{\mathbf{j}}}_\mathrm{s} 
\end{equation}
with $\protect\bbow{\mathrm{\mathbf{j}}}_\mathrm{s}$ known and $\protect\bow{\mathrm{\mathbf{a}}}$ containing the degrees of freedom (the magnetic vector potential integrated on the edges of the primal 
grid). Alternatively, mixed formulations incorporating a gauging conditions have been proposed, e.g. \cite{Clemens_2005aa}
\begin{equation*}
	\begin{pmatrix}
		\mathbf{M}_{\sigma} & 0\\
		0 & 0
	\end{pmatrix}
	\frac{\mathrm{d}}{\mathrm{d}t}
	\begin{pmatrix}
		\protect\bow{\mathrm{\mathbf{a}}}\\
		{\boldsymbol\Phi}
	\end{pmatrix}
	+
	\begin{pmatrix}
		\mathbf{C}^{\top}\mathbf{M}_{\nu}\mathbf{C} & \mathbf{M}_1\widetilde{\mathbf{S}}^{\top}\\
		-\widetilde{\mathbf{S}}\mathbf{M}_1 & - \ensuremath{ \mathbf{M} }_\text{N}^{-1}
	\end{pmatrix}
	\begin{pmatrix}
		\protect\bow{\mathrm{\mathbf{a}}}\\
		{\boldsymbol\Phi}
	\end{pmatrix}
	=
	\begin{pmatrix}
		\protect\bbow{\mathrm{\mathbf{j}}}_\mathrm{s}\\
		0
	\end{pmatrix}
\end{equation*}
where $\mathbf{M}_1$ is a regularised version of $\mathbf{M}_{\sigma}$ and $\ensuremath{ \mathbf{M} }_\text{N}$ is a regular matrix to ensure the correct physical units as in \eqref{eq:mge_lorenz}. Using the Schur complement, one derives a grad-div regularisation, e.g. \cite{Clemens_2002aa}, where
\begin{equation*}
		\ensuremath{ \mathbf{Z} }_\sigma = \ensuremath{ \mathbf{M} }_1\widetilde{\mathbf{S}}^{\top}\ensuremath{ \mathbf{M} }_\text{N}\widetilde{\mathbf{S}}\ensuremath{ \mathbf{M} }_1.
\end{equation*}
can be used to finally arrive at
\begin{ass}\label{ass:regmatrix}
 	Let us assume that system \eqref{eq:c-c} is rewritten as
	\begin{equation}\label{eq:mqs_reg}
	\mathbf{M}_{\sigma}\frac{\mathrm{d}}{\mathrm{d}t}\protect\bow{\mathrm{\mathbf{a}}} + \ensuremath{ \mathbf{K} }_\nu\protect\bow{\mathrm{\mathbf{a}}} = \protect\bbow{\mathrm{\mathbf{j}}}_\mathrm{s},
	\end{equation}
	where $\ensuremath{ \mathbf{K} }_\nu = \mathbf{C}^{\top}\mathbf{M}_{\nu}\mathbf{C} + \ensuremath{ \mathbf{Z} }_\sigma$, provided  $\ensuremath{ \mathbf{Z} }_\sigma$ is a positive semidefinite matrix that enforces the matrix 
	pencil 
	$\lambda\mathbf{M}_{\sigma} + \ensuremath{ \mathbf{K} }_\nu$ to be positive definite for 
	$\lambda > 0$. 
\end{ass}

\begin{theorem}\label{thm:grad-div}
	Under Assumptions~\ref{ass:dom} ,\ref{ass:material}, \ref{ass:nophantoms} and \ref{ass:regmatrix}, system \eqref{eq:mqs_reg} has Kroenecker and 
	tractability 
	index 1.
\end{theorem}
	The proof of the \textit{Kronecker index} of system \eqref{eq:mqs_reg} has been originally given in \cite{Nicolet_1996aa}. More recently, 
	\cite{Kerler-Back_2017aa} obtained the same result for the tractability index and
	\cite{Bartel_2011aa} used the tractability index concept to analyse the DAE index of this formulation with an attached network description. 
	
Instead of using $\ensuremath{ \mathbf{Z} }_\sigma$, various gauging techniques have been proposed for this formulation, such as the tree-cotree gauge \cite{Munteanu_2002aa} or the weak gauging property of iterative linear solvers \cite{Clemens_1999ac}. Due to the simple structure 
of $\mathbf{M}_{\sigma}$ and as long as the gauge leads to a positive definite matrix pencil, the index can be derived analogously as before.

\begin{figure}[t]
	\pgfmathsetmacro{\BBd}{3e-3}
	\pgfmathsetmacro{\Alw}{4e-3}
	\pgfmathsetmacro{\Alh}{8e-3}
	\pgfmathsetmacro{\Clra}{6e-3}
	\pgfmathsetmacro{\Clrb}{8e-3}
	\pgfmathsetmacro{\Clh}{12e-3}
	\centering
	\subfigure[Cross-section of inductor at the z=0 plane.]{\includegraphics{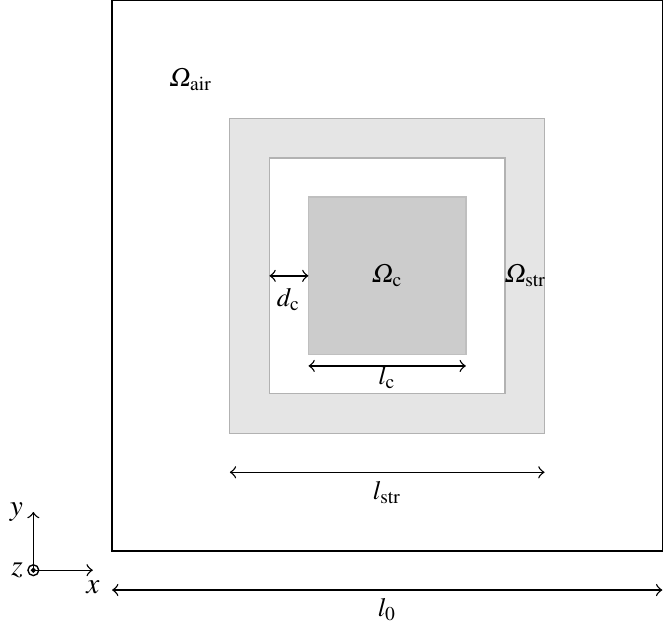}
	}
	\subfigure[Cross-section of inductor at the x=0 plane.]{\includegraphics{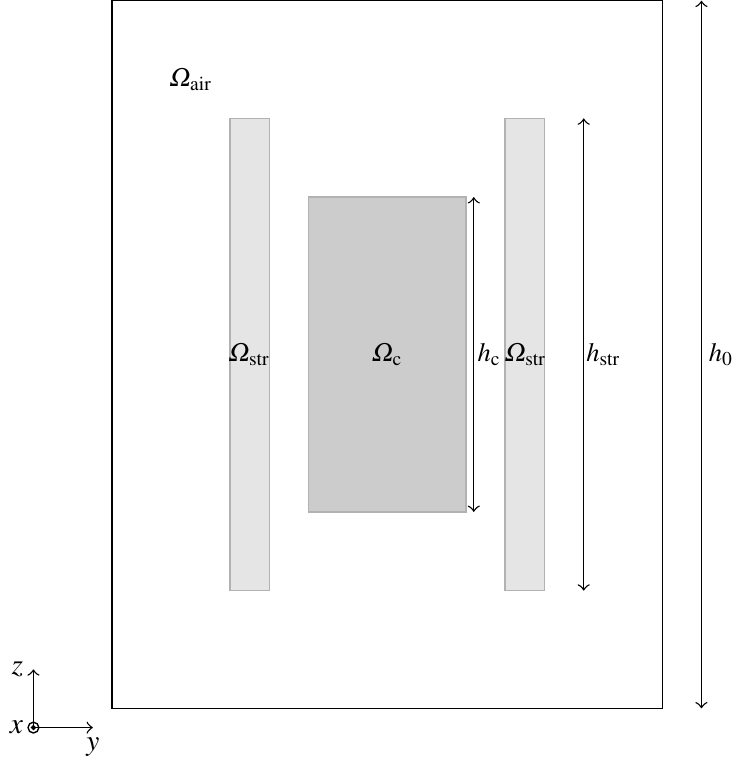}
}
	\caption{Magnetoquasistatic model of an inductor with metal core. The distances are $l_{0} = 1.4\cdot10^{-2}\,\si{\metre}$, 
		$l_\mathrm{str} = 8\cdot10^{-3}\,\si{\metre}$
		$l_\mathrm{c} =4\cdot10^{-3}\,\si{\metre} $, $d_\mathrm{c} = 1\cdot10^{-3}\,\si{\metre}$, $h_0 = 1.8\cdot10^{-2}\,\si{\metre}$, 
		$l_\mathrm{str} = 1.2\cdot10^{-2}\,\si{\metre}$ and 
		$l_\mathrm{c} = 8\cdot10^{-3}\,\si{\metre}$.}
	\label{fig:mqs}
\end{figure}

\begin{figure}[t]
	\centering
	\subfigure[Iron core (red) with surrounding coil (transparent grey)]{\includegraphics[width=.45\linewidth]{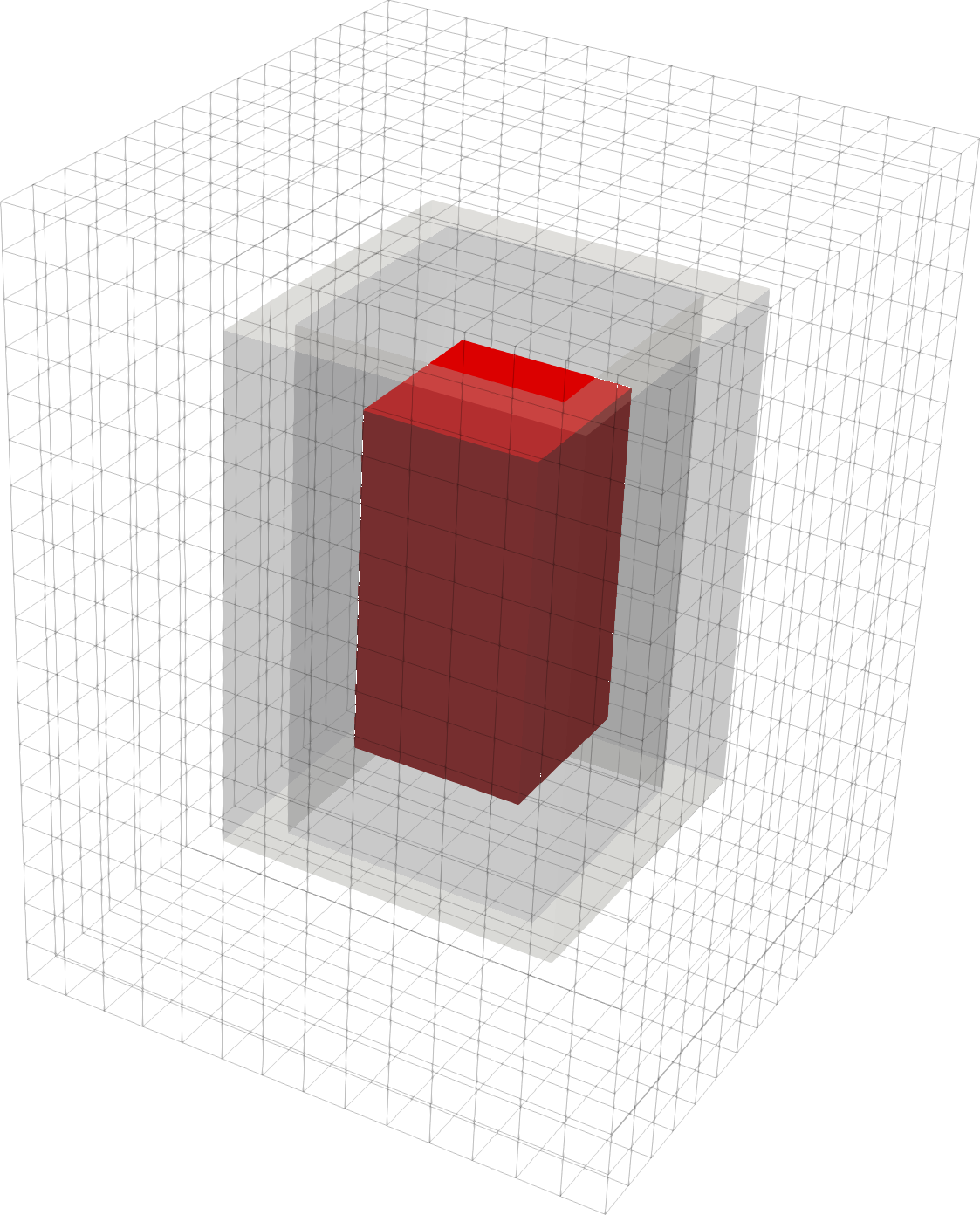}}
	\hspace{0.05\linewidth}
	\subfigure[Source current density given by winding function $\ensuremath{\vec{J}}_\text{s} = 
	\boldsymbol{\chi}_\text{s}i_\text{s}$]{\includegraphics[width=.45\linewidth]{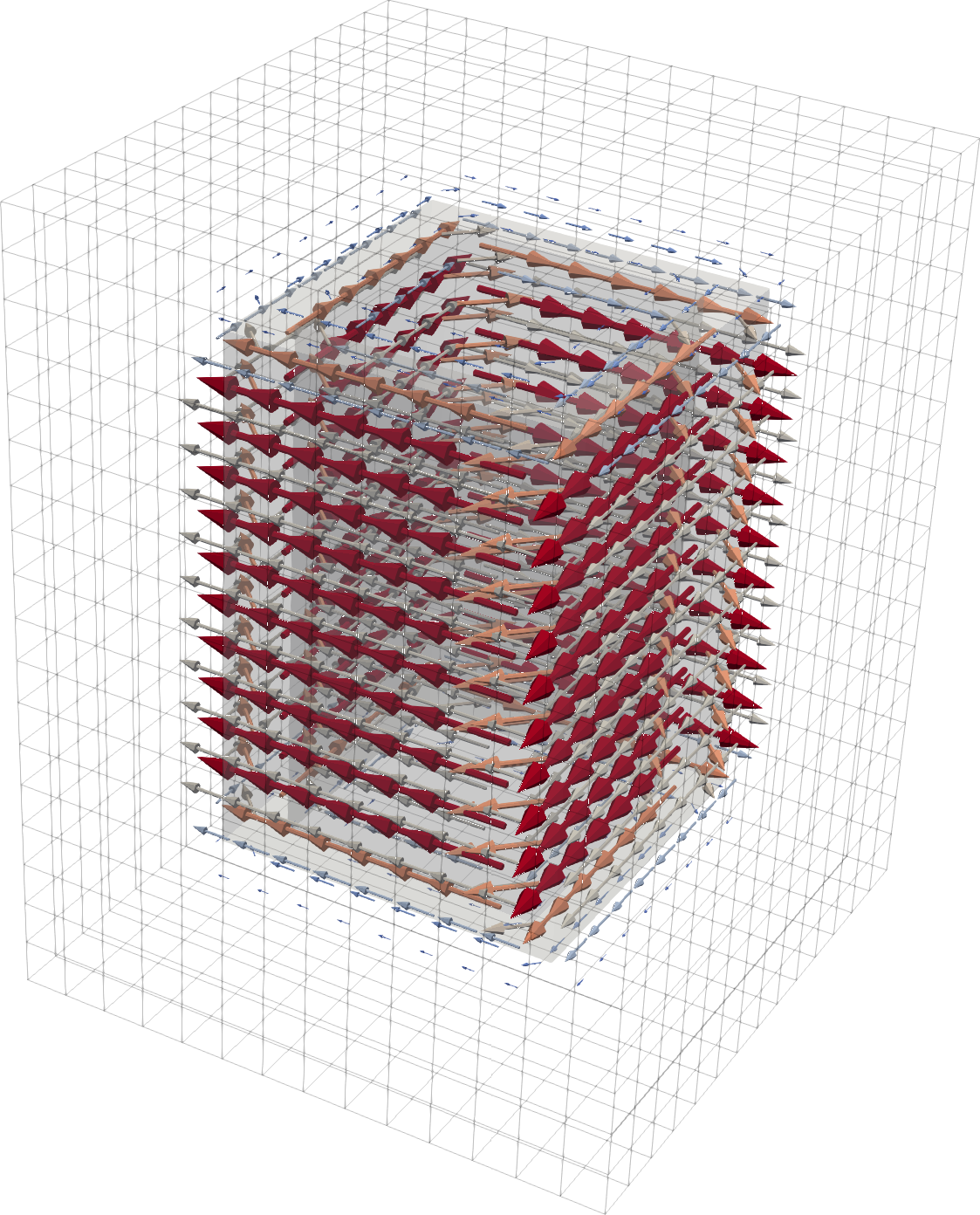}}
	\caption{Simple magnetoquasistatic model of an inductor with metal core. Coil is given by the stranded conductor model (Benchmark 
		\ref{bench:mqs_a})}
	\label{fig:mqs_a}
\end{figure}

\begin{benchmark}\label{bench:mqs_a}
A common benchmark for magnetoquasistatic models are inductors with a metal core. The example in Fig.~\ref{fig:mqs} from \cite{Cortes-Garcia_2018ab} 
features an aluminium core, i.e. 
$\sigma=$\SI{35e6}{S/m} in $\Omega_\text{c}$ with a copper coil $\Omega_\mathrm{str}$ surrounded by air $\Omega_\mathrm{air}$. The coil is given by 
the stranded conductor model consisting of 
120 turns with 
conductivity $\sigma=$\SI{1e6}{S/m}. The conductivity in $\Omega_\mathrm{air}$ is zero and disregarded in $\Omega_\mathrm{str}$ in the construction 
of $\mathbf{M}_{\sigma}$ as eddy currents are assumed 
negligible in the 
windings. Vacuum permeability $\mu = 4\pi\cdot$\SI{1e-7}{H/m} is assumed everywhere ($\Omega$) and electric boundary conditions are enforced 
$\Gamma=\Gamma_\text{ebc}$, cf. \eqref{eq:BC}. 
The FIT discretisation uses an equidistant hexahedral grid with step size $10^{-3}$, which leads to  $3528$ elements.

The discretisation of the winding function for the A* formulation can be visualised in Fig.~\ref{fig:mqs_a}.
As no gauging is performed, $9958$ degrees of freedom arise.

For the simulation, the source current density is $\protect\bbow{\mathrm{\mathbf{j}}}_\mathrm{s} = \mathbf{X}_\mathrm{s}\mathbf{i}_\mathrm{s}$, where
$\mathbf{X}_\mathrm{s}$ is the discretisation of the winding function in Fig.~\ref{fig:mqs_a}, 
$\mathbf{i}_\mathrm{s} = \sin(2\pi f_\mathrm{s} t)$ is the source current and the frequency $f_\mathrm{s}$ is $500\,$\si{\hertz}. Time integration is 
performed with the implicit Euler method with step size $\Delta t = $\SI{2e-5}{s} from time $t_0 = $\SI{0}{s} to $t_\mathrm{end} = $\SI{2e-3}{s} 
and zero 
initial condition. 
Fig.~\ref{fig:aWmagn} 
shows the magnetic energy $E_\mathrm{mag} = \frac{1}{2}\int_{\Omega}\ensuremath{\vec{H}}\cdot\ensuremath{\vec{B}}\;\mathrm{d}\Omega\approx\frac{1}{2}\protect\bbow{\mathrm{\mathbf{b}}}^{\top}\mathbf{M}_{\nu}\protect\bbow{\mathrm{\mathbf{b}}}$ 
over time.
\end{benchmark}

\begin{figure}
	\centering
	\includegraphics{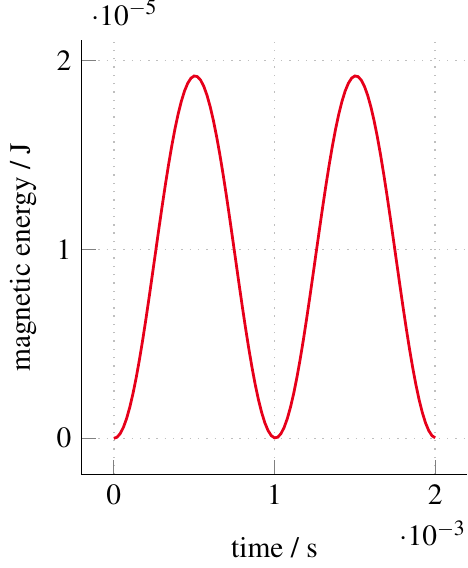}
		\caption{Magnetic energy of the inductor in Fig.~\ref{fig:mqs} simulated with the A* formulation.}
	\label{fig:aWmagn}
\end{figure}

\subsubsection{Electric Vector Potential T-$\Omega$-Formulation}
Maxwell's equations and their material laws in magnetoquasistatics can also be expressed in the T-$\Omega$ formulation. Using the electric vector and 
magnetic scalar potentials from equation \eqref{eq:to_pot}, the system reads
\begin{align*}
	\nabla\times\left(\boldsymbol{\rho}\nabla\times\mbox{\textbf{T}}\right) + \boldsymbol{\mu}\frac{\partial \mbox{\textbf{T}}}{\partial t}-\boldsymbol{\mu}\nabla\frac{\partial\psi}{\partial 
	t} &= 
	-\boldsymbol{\mu}\frac{\partial\ensuremath{\vec{H}}_{\mathrm{s}}}{\partial t}-\nabla\times\ensuremath{\vec{E}}_\mathrm{s}\\
	\nabla\cdot\left(\boldsymbol{\mu}\mbox{\textbf{T}}\right)-\nabla\cdot\left(\boldsymbol{\mu}\nabla\psi\right) &= -\nabla\cdot\left(\boldsymbol{\mu}\ensuremath{\vec{H}}_{\mathrm{s}}\right),
\end{align*}
 with $\ensuremath{\vec{E}}_\mathrm{s}$ being the source electric field strength when solid conductors are present. Here, $\boldsymbol{\rho}$ is the electrical 
 resistivity, 
which 
corresponds to the inverse of $\boldsymbol{\sigma}$ wherever $\boldsymbol{\sigma}\neq0$. As in the case of the A-$\phi$ formulation, the gauge 
condition is set for the spatially discretised version of our system
\begin{align*}
	\mathbf{C}^{\top}{\mathbf{M}_{\rho}}\mathbf{C}\protect\bow{\mathrm{\mathbf{t}}} + \mathbf{M}_{\mu}\frac{\mathrm{d}}{\mathrm{d}t}\protect\bow{\mathrm{\mathbf{t}}}+\mathbf{M}_{\mu}\widetilde{\mathbf{S}}^{\top}\frac{\mathrm{d}}{\mathrm{d}t}\Psi 
	&=-\mathbf{M}_{\mu}\frac{\mathrm{d}}{\mathrm{d}t}\protect\bow{\mathrm{\mathbf{h}}}_\mathrm{s}-\mathbf{C}\protect\bow{\mathrm{\mathbf{e}}}_\mathrm{s} \\
	\widetilde{\mathbf{S}}\mathbf{M}_{\mu}\protect\bow{\mathrm{\mathbf{t}}}+\widetilde{\mathbf{S}}\mathbf{M}_{\mu}\widetilde{\mathbf{S}}^{\top}\Psi &= -\widetilde{\mathbf{S}}\mathbf{M}_{\mu}\protect\bow{\mathrm{\mathbf{h}}}_\mathrm{s},
\end{align*}
with degrees of freedom $\protect\bow{\mathrm{\mathbf{t}}}$, which contains the on the dual grid's edges integrated electric vector potential and $\Psi$, which consists of the 
magnetic 
scalar potential on the dual grid's nodes. 
\begin{definition}\label{def:togaug}
	For gauging, a tree $T$ is generated on the dual grid's edges inside the conducting region $\Omega_{\textrm{c}}$.
	We define a projector $\tilde{\ensuremath{ \mathbf{P} }}_\mathrm{t}$ onto the cotree  of $T$ and truncate it by deleting all the linearly dependent columns to obtain 
	$\ensuremath{ \mathbf{P} }_\mathrm{t}$.
	As gauging condition we set $\ensuremath{ \mathbf{Q} }_\mathrm{t}\protect\bow{\mathrm{\mathbf{t}}} = 0$, where  $\ensuremath{ \mathbf{Q} }_\mathrm{t}$ spans the kernel of $\ensuremath{ \mathbf{P} }_\mathrm{t}^{\top}$. This corresponds to 
	setting to zero the values of the electric vector potential on the edges $T$.
\begin{align}
		\ensuremath{ \mathbf{P} }_\mathrm{t}^{\top}\mathbf{C}^{\top}{\mathbf{M}_{\rho}}\mathbf{C}\ensuremath{ \mathbf{P} }_\mathrm{t}\protect\bow{\mathrm{\mathbf{t}}} + 
		\ensuremath{ \mathbf{P} }_\mathrm{t}^{\top}\mathbf{M}_{\mu}\ensuremath{ \mathbf{P} }_\mathrm{t}\frac{\mathrm{d}}{\mathrm{d}t}\protect\bow{\mathrm{\mathbf{t}}}+\ensuremath{ \mathbf{P} }_\mathrm{t}^{\top}\mathbf{M}_{\mu}\widetilde{\mathbf{S}}^{\top}\frac{\mathrm{d}}{\mathrm{d}t}\psi \label{eq:to1}
		&=-\ensuremath{ \mathbf{P} }_\mathrm{t}^{\top}\mathbf{M}_{\mu}\frac{\mathrm{d}}{\mathrm{d}t}\protect\bow{\mathrm{\mathbf{h}}}_\mathrm{s}-\ensuremath{ \mathbf{P} }_\mathrm{t}^{\top}\mathbf{C}\protect\bow{\mathrm{\mathbf{e}}}_\mathrm{s} \\
	\widetilde{\mathbf{S}}\mathbf{M}_{\mu}\ensuremath{ \mathbf{P} }_\mathrm{t}\protect\bow{\mathrm{\mathbf{t}}}+\widetilde{\mathbf{S}}\mathbf{M}_{\mu}\widetilde{\mathbf{S}}^{\top}\psi &= -\widetilde{\mathbf{S}}\mathbf{M}_{\mu}\protect\bow{\mathrm{\mathbf{h}}}_\mathrm{s}.\label{eq:to2}
\end{align}
\end{definition}
\begin{property}\label{prop:p}
	The matrix $\ensuremath{ \mathbf{P} }_\mathrm{t}$ fulfils 
	\begin{enumerate}
		\item $
		\det \left(\ensuremath{ \mathbf{P} }_\mathrm{t}^{\top}\mathbf{C}^{\top}{\mathbf{M}_{\rho}}\mathbf{C}\ensuremath{ \mathbf{P} }_\mathrm{t}\right) \neq 0$,\label{prop:cotree1}
		\item $\ima\ensuremath{ \mathbf{P} }_\mathrm{t} \cap \ima\widetilde{\mathbf{S}}^{\top} = \emptyset$. \label{prop:cotree}
	\end{enumerate}
\end{property}
Property \hyperref[prop:cotree1]{\ref*{prop:p}.\ref*{prop:cotree1}} is a consequence of the tree-cotree gauge (see \cite{Munteanu_2002aa}) and
Property \hyperref[prop:cotree]{\ref*{prop:p}.\ref*{prop:cotree}} follows from Property \hyperref[prop:cotree1]{\ref*{prop:p}.\ref*{prop:cotree1}} 
and the 
fact that 
$\ima\widetilde{\mathbf{S}}^{\top} \subseteq \ker\mathbf{C}$ 
(Lemma \ref{prop:topmatrices}).
\begin{property}[\cite{Cortes-Garcia_2018ab}]\label{prop:span}
	Every $\mathbf{x}\in \mathbb{R}^n$ can be written as $\mathbf{x} = \mathbf{M}_{\mu}^{1/2}\widetilde{\mathbf{S}}^{\top} \mathbf{x}_1 + 
	\mathbf{M}_{\mu}^{-1/2}\mathbf{W}^{\top} \mathbf{x}_2$, where $n \coloneqq \rank{\mathbf{M}_{\mu}}$ and $\mathbf{W}$ is the matrix whose columns span $\ker \widetilde{\mathbf{S}}$.
\end{property}
\begin{proof}
	As $\ensuremath{ \mathbf{M} }_{\mu}$ has full rank and is symmetric, $\rank 
	(\ensuremath{ \mathbf{M} }_{\mu}^{1/2}\widetilde{\mathbf{S}}^{\top}) = \rank \widetilde{\mathbf{S}}^{\top}$ and \\
	$\rank(\ensuremath{ \mathbf{M} }_{\mu}^{-1/2}\mathbf{W}^{\top}) = \rank\mathbf{W}^{\top}$.
	Using the rank-nullity theorem together with the fact that both subspaces are orthogonal and thus linear independent, we obtain that their direct 
	sum spans $\mathbb{R}^n$. \qed
\end{proof}
\begin{theorem}[\cite{Cortes-Garcia_2018ab}]
	Under Assumptions~\ref{ass:dom}, \ref{ass:material} and \ref{ass:nophantoms}, the system of DAEs 
	\eqref{eq:to1}-\eqref{eq:to2} has differential index-1. 
\end{theorem}
\begin{proof}
	 As the system has an algebraic constraint, it has at least index-1. Equation \eqref{eq:to2} is differentiated and $\frac{\mathrm{d}}{\mathrm{d}t}\Psi$
	is extracted as
	\begin{equation*}
		\frac{\mathrm{d}}{\mathrm{d}t}\Psi  = -(\widetilde{\mathbf{S}}\mathbf{M}_{\mu}\widetilde{\mathbf{S}}^{\top})^{-1}\widetilde{\mathbf{S}}\mathbf{M}_{\mu}\ensuremath{ \mathbf{P} }_\mathrm{t}\frac{\mathrm{d}}{\mathrm{d}t}\mathbf{t} - 
		(\widetilde{\mathbf{S}}\mathbf{M}_{\mu}\widetilde{\mathbf{S}}^{\top})^{-1}\widetilde{\mathbf{S}}\mathbf{M}_{\mu}\protect\bow{\mathrm{\mathbf{h}}}_\mathrm{s}.
	\end{equation*}
	This can be inserted into equation \eqref{eq:to1} and now it is sufficient to see that \mbox{$\det\left(\ensuremath{ \mathbf{P} }_\mathrm{t}\ensuremath{ \mathbf{Z} }\ensuremath{ \mathbf{P} }_\mathrm{t}\right) 
	\neq 
	0$} for
	\begin{equation*}
		\ensuremath{ \mathbf{Z} } = (\mathbf{M}_{\mu}-\mathbf{M}_{\mu}\widetilde{\mathbf{S}}^{\top}(\widetilde{\mathbf{S}}\mathbf{M}_{\mu}\widetilde{\mathbf{S}}^{\top})^{-1}\widetilde{\mathbf{S}}\mathbf{M}_{\mu}).
	\end{equation*}
    	We can write $\mathbf{M}_{\mu}^{1/2}\ensuremath{ \mathbf{P} }_\mathrm{t} \mathbf{x} =\mathbf{M}_{\mu}^{1/2}\widetilde{\mathbf{S}}^{\top} \mathbf{x}_1 + 
    	\mathbf{M}_{\mu}^{-1/2}\mathbf{W}^{\top} 
    \mathbf{x}_2$ (Property \ref{prop:span}). As $\mathbf{M}_{\mu}^{1/2}$ is invertible, \mbox{$\ensuremath{ \mathbf{P} }_\mathrm{t} \mathbf{x} =\widetilde{\mathbf{S}}^{\top} 
    \mathbf{x}_1 + 
    \mathbf{M}_{\mu}^{-1}\mathbf{W}^{\top} 
    \mathbf{x}_2$} and, as $\ensuremath{ \mathbf{P} }_\mathrm{t}\mathbf{x} \neq \widetilde{\mathbf{S}}^{\top} \mathbf{x}_1$ (Property \ref{prop:p}), $\mathbf{M}_{\mu}^{-1}\mathbf{W}^{\top} 
    \mathbf{x}_2\neq 0$.
    Thus
 \begin{align*}
 \ensuremath{ \mathbf{x} }^{\top} \ensuremath{ \mathbf{P} }_\mathrm{t}^{\top}\ensuremath{ \mathbf{Z} }\ensuremath{ \mathbf{P} }_\mathrm{t} \ensuremath{ \mathbf{x} } = \ensuremath{ \mathbf{x} }_2^{\top}\mathbf{W}\mathbf{M}_{\mu}^{-1}\mathbf{W}^{\top}\ensuremath{ \mathbf{x} }_2 > 0,
 \end{align*}
 as long as $\ensuremath{ \mathbf{x} } \neq 0$. We conclude that  $\ensuremath{ \mathbf{P} }_\mathrm{t}^{\top}\ensuremath{ \mathbf{Z} }\ensuremath{ \mathbf{P} }_\mathrm{t}$ is positive definite. \qed
\end{proof}

\begin{figure}[t]
	\centering
	\subfigure[Iron core (red) with surrounding coil (transparent grey)]{\includegraphics[width=.45\linewidth]{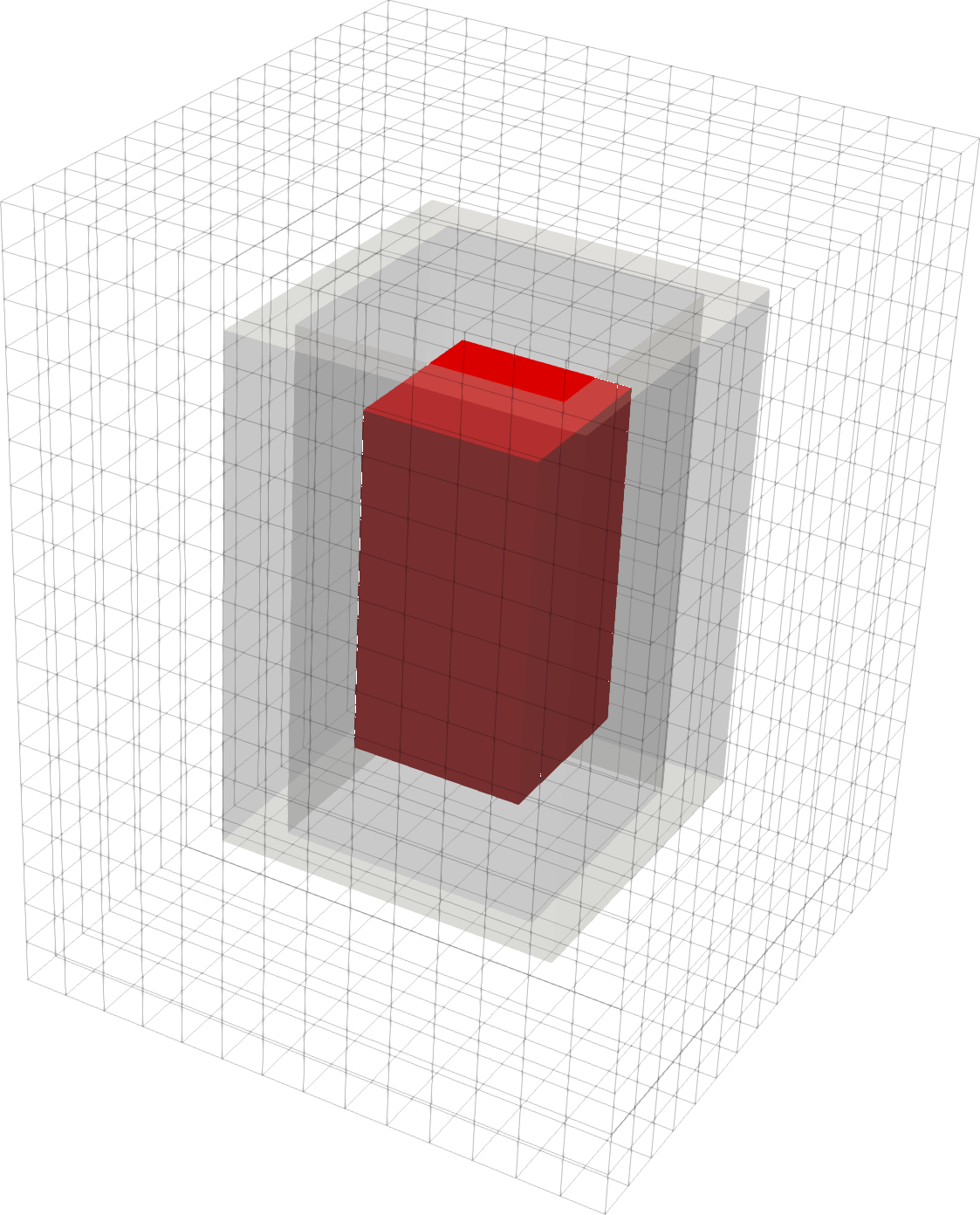}}
	\hspace{0.05\linewidth}
	\subfigure[Source magnetic field strength $\ensuremath{\vec{H}}_\text{s}$ such that $\nabla\times\ensuremath{\vec{H}}_\text{s} = \boldsymbol{\chi}_\text{s}i_\text{s}$]{\includegraphics[width=.45\linewidth]{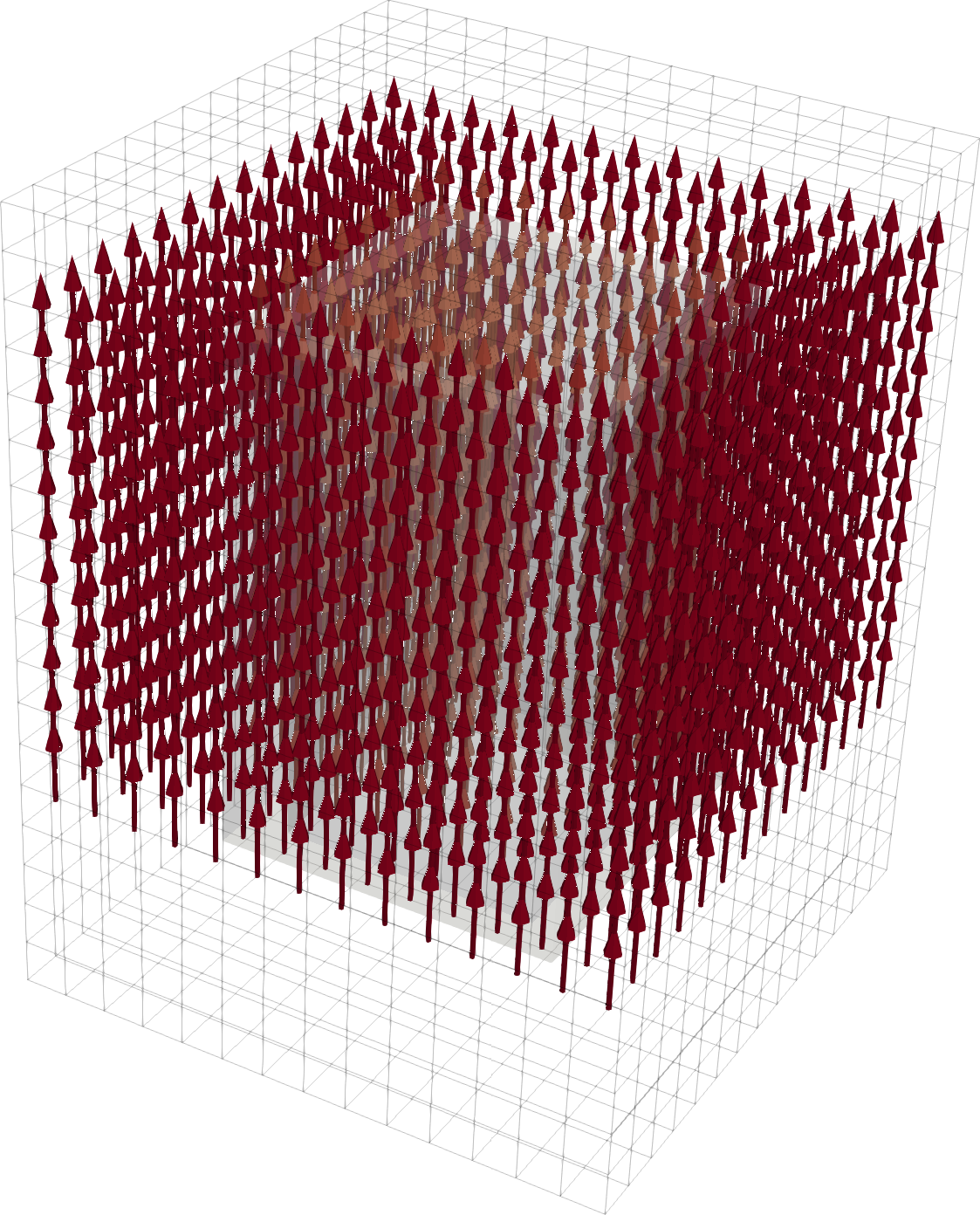}}
	\caption{Simple magnetoquasistatic model of an inductor with metal core. Coil is given by the stranded conductor model (Benchmark \ref{bench:mqs_t})}
	\label{fig:mqs_t}
\end{figure}

\begin{benchmark}\label{bench:mqs_t}
The physical specifications and discretisation of the benchmark example of the T-$\Omega$ formulation is equivalent to Benchmark~\ref{bench:mqs_a}.

The construction of the winding function is different and it is depicted in Fig.~\ref{fig:mqs_t}.
This time a tree-cotree gauge is performed, which yield only $4610$ degrees of freedom.

Again for the simulation the source current is set to $\mathbf{i}_\mathrm{s} = \sin(2\pi f_\mathrm{s} t)$ with frequency 
$f_\mathrm{s}=500\,$\si{\hertz}. The source magnetic field is $\protect\bow{\mathrm{\mathbf{h}}}_\mathrm{s} = \mathbf{Y}_\mathrm{s}\mathbf{i}_\mathrm{s}$, 
with $\mathbf{Y}_\mathrm{s}$ being the discretisation of the winding function in Fig.~\ref{fig:mqs_t}. Like in Benchmark~\ref{bench:mqs_a}, 
time integration is performed with implicit Euler with step size $\Delta t =$\SI{2e-5}{s} from time $t_0 = $\SI{0}{s} to $t_\mathrm{end} = 
$\SI{2e-3}{s} and with zero initial condition.
The resulting magnetic energy is depicted in Figure~\ref{fig:tWmagn}.

\end{benchmark}
\begin{figure}
	\centering
	\includegraphics{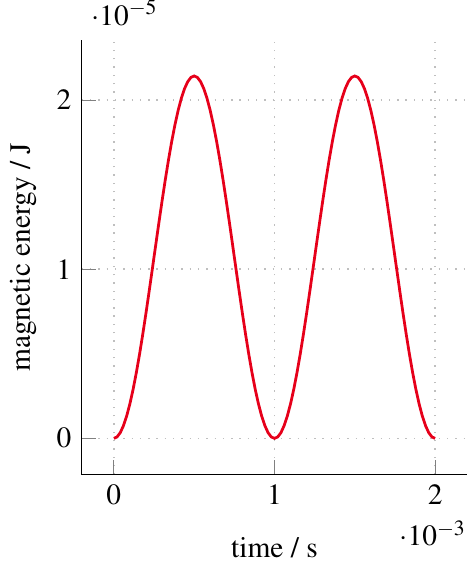}
		\caption{Magnetic energy of the inductor in Fig.~\ref{fig:mqs} simulated with the T-$\Omega$ formulation.}
	\label{fig:tWmagn}
\end{figure}

\begin{remark}
	Note that the magnetic energy obtained with the A* formulation in Figure \ref{fig:aWmagn} and the one obtained with the T-$\Omega$ one in Figure 
	\ref{fig:tWmagn} differ. Both formulations are dual to each other, i.e., their degrees of freedom are on dual sides of Maxwell's House in Figure 
	\ref{fig:tonti}. They converge to the unique physical solution from below and above. This property can be used in order to study the error of the 
	spatial discretisation (see \cite{Albanese_1991aa}).
\end{remark}

\subsection{Darwin Model}
In many situations, phenomena related to electric energy, magnetic energy and Joule losses coincide, while at the same time, for the considered operating frequencies, the wave lengths are much larger than the model size, which indicates that wave propagation effects can be neglected. An example thereof is a filter design with concentrated elements, i.e., coils and capacitors connected by strip lines, arranged on a printed circuit board \cite{Schuhmacher_2017aa}. In such models, resonances between the discrete elements and Joule losses in all conducting parts need to be considered, whereas for the considered frequency range, cross-talk due to electromagnetic radiation is irrelevant. For such configurations, a combination of electroquasistatics and magnetoquasistatics in the form of the \textit{Darwin} approximation of the Maxwell equations is the appropriate formulation.

The Darwin approximation starts from the decomposition of the electric field strength $\ensuremath{\vec{E}}=\ensuremath{\vec{E}}_\text{irr}+\ensuremath{\vec{E}}_\text{rem}$ into an \textit{irrotational part} $\ensuremath{\vec{E}}_\text{irr}$ and a \textit{remainder part} $\ensuremath{\vec{E}}_\text{rem}$. This decomposition is not unique, in contrast to the Helmholtz decomposition, which can be considered as a special case enforcing $\ensuremath{\vec{E}}_\text{rem}$ to be solenoidal. In this paper, the non-uniqueness of the decomposition will be resolved when choosing a gauge condition below. The Darwin approximation consists of removing the displacement currents related to $\ensuremath{\vec{E}}_\text{rem}$ from the law of Amp\`{e}re-Maxwell. This affects the set of Maxwell equations in the sense that, in the formulations derived below, second derivatives with respect to time vanish and the overall PDE looses its hyperbolic character, which is equivalent to neglecting wave propagation effects. The set of relevant equations is
\begin{align}
	\label{eq:darwinfaraday}\nabla\times\ensuremath{\vec{E}}_\text{rem} &= -\frac{\partial{\ensuremath{\vec{B}}}}{\partial{{t}}} \quad ;\\
	\label{eq:darwinampere}\nabla\times\left(\boldsymbol{\nu}\ensuremath{\vec{B}}\right) &= \ensuremath{\vec{J}_{\mathrm{s}}} +\boldsymbol{\sigma}\ensuremath{\vec{E}}_\text{irr} +\boldsymbol{\sigma}\ensuremath{\vec{E}}_\text{rem} +\frac{\partial{}}{\partial{{t}}}\left(\varepsilon\ensuremath{\vec{E}}_\text{irr}\right) \quad ;\\
	\label{eq:darwinmgs}\nabla\cdot\ensuremath{\vec{B}} &= 0 \quad ;\\
	\label{eq:darwinegs}\nabla\cdot\left(\varepsilon\ensuremath{\vec{E}}_\text{irr}\right) +\nabla\cdot\left(\varepsilon\ensuremath{\vec{E}}_\text{rem}\right) &= \rho \quad .
\end{align}

Almost all publications on the Darwin approximation are limited to the case with homogeneous materials. They choose the Helmholtz decomposition, i.e., $\nabla\cdot\ensuremath{\vec{E}}_\text{rem}=0$, which causes the term $\nabla\cdot\left(\varepsilon\ensuremath{\vec{E}}_\text{rem}\right)$ in \eqref{eq:darwinegs} to vanish \cite{Dirks_1996aa,Larsson_2007aa}. In the more general case considered here, however, this term is important as it models the charges accumulating at material interfaces due to the induced electric field strength \cite{Koch_2011aa}.

An ambiguity arises when expressing the continuity equation $\nabla\cdot\ensuremath{\vec{J}}+\frac{\partial{\rho}}{\partial{{t}}}=0$. When inserting Gauss law, the result reads
\begin{equation}\label{eq:darwincontinuity}
	\nabla\cdot\ensuremath{\vec{J}_{\mathrm{s}}} +\nabla\cdot\left(\boldsymbol{\sigma}\ensuremath{\vec{E}}_\text{irr}\right) +\nabla\cdot\left(\boldsymbol{\sigma}\ensuremath{\vec{E}}_\text{rem}\right) 
	+\frac{\partial{}}{\partial{{t}}}\left(\nabla\cdot\varepsilon\ensuremath{\vec{E}}_\text{irr}\right) +\frac{\partial{}}{\partial{{t}}}\left(\nabla\cdot\varepsilon\ensuremath{\vec{E}}_\text{rem}\right) =0 \quad ,
\end{equation}
whereas when applying applying the divergence operator to \eqref{eq:darwinampere}, the last term in \eqref{eq:darwincontinuity} is missing. Hence, the Darwin model introduces an anomaly in the continuity equation, which can only be alleviated by also neglecting the displacement currents due to $\ensuremath{\vec{E}}_\text{rem}$ in the continuity equation.

\subsubsection{$\phi$-A-$\psi$ formulation of the Darwin model}

The irrotational part in of the electric field strength is represented by an electric scalar potential $\phi$, i.e., $\ensuremath{\vec{E}}_\text{irr}=-\nabla\phi$. The magnetic Gauss law \eqref{eq:darwinmgs} is resolved by the definition of the magnetic vector potential $\ensuremath{\vec{A}}$, i.e., $\ensuremath{\vec{B}}=\nabla\times\ensuremath{\vec{A}}$, whereas Faraday's law \eqref{eq:darwinfaraday} is fulfilled by the definition of an additional scalar potential $\psi$ such that $\ensuremath{\vec{E}}_\text{rem}=-\frac{\partial{\ensuremath{\vec{A}}}}{\partial{{t}}}-\nabla\psi$. The law of Amp\`{e}re-Darwin, Gauss' law and the continuity equation become
\begin{align}
\label{eq:darwin1}\nabla\times\left(\boldsymbol{\nu}\nabla\times\ensuremath{\vec{A}}\right) +\boldsymbol{\sigma}\frac{\partial{\ensuremath{\vec{A}}}}{\partial{{t}}} +\boldsymbol{\sigma}\nabla\psi +\boldsymbol{\sigma}\nabla\phi +\varepsilon\nabla\frac{\partial{\phi}}{\partial{{t}}} &=\ensuremath{\vec{J}_{\mathrm{s}}} \quad ;\\
\label{eq:darwin2}-\nabla\cdot\left(\varepsilon\frac{\partial{\ensuremath{\vec{A}}}}{\partial{{t}}}\right) -\nabla\cdot\left(\varepsilon\nabla\psi\right) -\nabla\cdot\left(\varepsilon\nabla\phi\right) &=\rho \quad ;\\
\label{eq:darwin3}-\nabla\cdot\left(\boldsymbol{\sigma}\frac{\partial{\ensuremath{\vec{A}}}}{\partial{{t}}}\right) -\nabla\cdot\left(\boldsymbol{\sigma}\nabla\psi\right) -\nabla\cdot\left(\boldsymbol{\sigma}\nabla\phi\right) -\nabla\cdot\left(\varepsilon\nabla\frac{\partial{\phi}}{\partial{{t}}}\right) &= 0 \quad .
\end{align}

The introduced potentials $\phi$, $\ensuremath{\vec{A}}$ and $\psi$ lead to too many degrees of freedom. The electric field strength is decomposed as
\begin{equation}
	\ensuremath{\vec{E}} =\quad\underbrace{\overbrace{-\nabla\phi}^{\ensuremath{\vec{E}}_\phi}}_{\ensuremath{\vec{E}}_\text{irr}} \quad\underbrace{\overbrace{-\frac{\partial{\ensuremath{\vec{A}}}}{\partial{{t}}}}^{\ensuremath{\vec{E}}_{\ensuremath{\vec{A}}}} \quad\overbrace{-\nabla\psi}^{\ensuremath{\vec{E}}_\psi}}_{\ensuremath{\vec{E}}_\text{rem}}
\end{equation}
Irrotational parts of the electric field strength can be attributed to $\ensuremath{\vec{E}}_\phi$, $\ensuremath{\vec{E}}_{\ensuremath{\vec{A}}}$ or $\ensuremath{\vec{E}}_\psi$. However, the splitting determines which displacement currents are considered in the law of Amp\`{e}re-Maxwell and which are only considered in the electric law of Gauss. All components of $\ensuremath{\vec{E}}$ contribute to the Joule losses and contribute in the electric Gauss law, whereas only $\ensuremath{\vec{E}}_\phi$ contributes to the displacement currents. The component $\ensuremath{\vec{E}}_\psi$ can be fully integrated into $\ensuremath{\vec{E}}_{\ensuremath{\vec{A}}}$.

The choice of $\psi$ determines the fraction of the displacement currents which is neglected in the Darwin model. In the case of a homogeneous material distribution, the choice $\psi=0$ implies a Helmholtz splitting of $\ensuremath{\vec{E}}$ into the irrotational part $\ensuremath{\vec{E}}_\text{irr}$ and a solenoidal part $\ensuremath{\vec{E}}_\text{rem}$. Then, the Darwin approximation amounts to neglecting the displacement currents related to the solenoidal part of the electric field strength. Other choices for $\psi$ lead to other Darwin models yielding different results. Hence, Darwin's approximation to the Maxwell equations is not gauge-invariant \cite{Larsson_2007aa}.

\subsubsection{$\psi$-$\ensuremath{\vec{A}}^*$-formulation of the Darwin model}

A straightforward choice for $\psi$ is $\psi=0$ in which case, it makes sense to combine \eqref{eq:darwin1} and \eqref{eq:darwin2}, i.e.,
\begin{align}
	\label{eq:darwin4}\nabla\times\left(\boldsymbol{\nu}\nabla\times\ensuremath{\vec{A}}\right) +\boldsymbol{\sigma}\frac{\partial{\ensuremath{\vec{A}}}}{\partial{{t}}} +\boldsymbol{\sigma}\nabla\phi +\varepsilon\nabla\frac{\partial{\phi}}{\partial{{t}}} &=\ensuremath{\vec{J}_{\mathrm{s}}} \quad ;\\
	\label{eq:darwin5}-\nabla\cdot\left(\varepsilon\frac{\partial{\ensuremath{\vec{A}}}}{\partial{{t}}}\right) -\nabla\cdot\left(\varepsilon\nabla\phi\right) &=\rho \quad .
\end{align}
The drawback of this formulation is that the charge density must be prescribed, which may be cumbersome in the presence of metallic parts at floating potentials.

\subsubsection{Alternative $\psi$-$\ensuremath{\vec{A}}^*$-formulation of the Darwin model}

Another possible formulation \cite{Koch_2010aa} arises with $\psi=0$, but by combining \eqref{eq:darwin1} and \eqref{eq:darwin3}, i.e.,
\begin{align}
\label{eq:darwin6}\nabla\times\left(\boldsymbol{\nu}\nabla\times\ensuremath{\vec{A}}\right) +\boldsymbol{\sigma}\frac{\partial{\ensuremath{\vec{A}}}}{\partial{{t}}} +\boldsymbol{\sigma}\nabla\phi +\varepsilon\nabla\frac{\partial{\phi}}{\partial{{t}}} &=\ensuremath{\vec{J}_{\mathrm{s}}} \quad ;\\
\label{eq:darwin7}-\nabla\cdot\left(\boldsymbol{\sigma}\frac{\partial{\ensuremath{\vec{A}}}}{\partial{{t}}}\right) -\nabla\cdot\left(\boldsymbol{\sigma}\nabla\phi\right) -\nabla\cdot\left(\varepsilon\nabla\frac{\partial{\phi}}{\partial{{t}}}\right) &=0 \quad .
\end{align}
Here, however, the magnetic vector potential $\ensuremath{\vec{A}}$ is not uniquely defined in the non-conductive model parts. Then, a gauge is necessary is fix the irrotational part of $\ensuremath{\vec{A}}$ in the non-conductive model parts. During post-processing, one should discard the irrotational part of $\ensuremath{\vec{A}}$ when calculating the electric field strength or displacement current density in the non-conducting parts.

\subsubsection{Discretisation of the Darwin model}

The discrete counterpart of \eqref{eq:darwin4} and \eqref{eq:darwin5} reads
\begin{align}
	\label{eq:darwin8}\mathbf{C}^{\top}\mathbf{M}_{\nu}\mathbf{C}\protect\bow{\mathrm{\mathbf{a}}} +\mathbf{M}_{\sigma}\frac{\mathrm{d}}{\mathrm{d}t}\protect\bow{\mathrm{\mathbf{a}}} -\mathbf{M}_{\sigma}\widetilde{\mathbf{S}}^{\top}\Phi -\mathbf{M}_{\varepsilon}\widetilde{\mathbf{S}}^{\top}\frac{\mathrm{d}}{\mathrm{d}t}\Phi &= \protect\bbow{\mathrm{\mathbf{j}}}_\mathrm{s} \quad ;\\
	\label{eq:darwin9}\widetilde{\mathbf{S}}\mathbf{M}_{\varepsilon}\frac{\mathrm{d}}{\mathrm{d}t}\protect\bow{\mathrm{\mathbf{a}}} +\widetilde{\mathbf{S}}\mathbf{M}_{\varepsilon}\widetilde{\mathbf{S}}^{\top}\Phi &=\mathrm{\mathbf{q}} \quad ,
\end{align}
whereas the discrete counterpart of \eqref{eq:darwin6} and \eqref{eq:darwin7} reads
\begin{align}
	\label{eq:darwin10}\mathbf{C}^{\top}\mathbf{M}_{\nu}\mathbf{C}\protect\bow{\mathrm{\mathbf{a}}} +\mathbf{M}_{\sigma}\frac{\mathrm{d}}{\mathrm{d}t}\protect\bow{\mathrm{\mathbf{a}}} -\mathbf{M}_{\sigma}\widetilde{\mathbf{S}}^{\top}\Phi -\mathbf{M}_{\varepsilon}\widetilde{\mathbf{S}}^{\top}\frac{\mathrm{d}}{\mathrm{d}t}\Phi &= \protect\bbow{\mathrm{\mathbf{j}}}_\mathrm{s} \quad ;\\
	\label{eq:darwin11}\widetilde{\mathbf{S}}\mathbf{M}_{\sigma}\frac{\mathrm{d}}{\mathrm{d}t}\protect\bow{\mathrm{\mathbf{a}}} +\widetilde{\mathbf{S}}\mathbf{M}_{\sigma}\widetilde{\mathbf{S}}^{\top}\Phi +\widetilde{\mathbf{S}}\mathbf{M}_{\varepsilon}\widetilde{\mathbf{S}}^{\top}\frac{\mathrm{d}}{\mathrm{d}t}\Phi &=0 \quad ,
\end{align}
with all matrices defined as above. The second formulation needs a gauge for the non-conducting model parts as previously discussed in Assumption~\ref{ass:regmatrix}.

\bigskip

In \cite{Koch_2011aa}, simulation results for the Darwin approximation, the electroquasistatic approximation and the magnetoquasistatic approximation are compared to results obtained by a full-wave solver, serving as a reference. This comparison clearly shows the limitations (the neglected wave propagation effects) and the advantages (more accurate than electroquasistatics and magnetoquasistatics and faster than a full-wave solver) of the Darwin model.

\bigskip

A DAE index analysis of the systems \eqref{eq:darwin8}-\eqref{eq:darwin9} and \eqref{eq:darwin10}-\eqref{eq:darwin11} is an open research question. Also, the following benchmark is merely a literature reference. 

\begin{figure}[t]
  \centering
  \includegraphics[width=.58\linewidth]{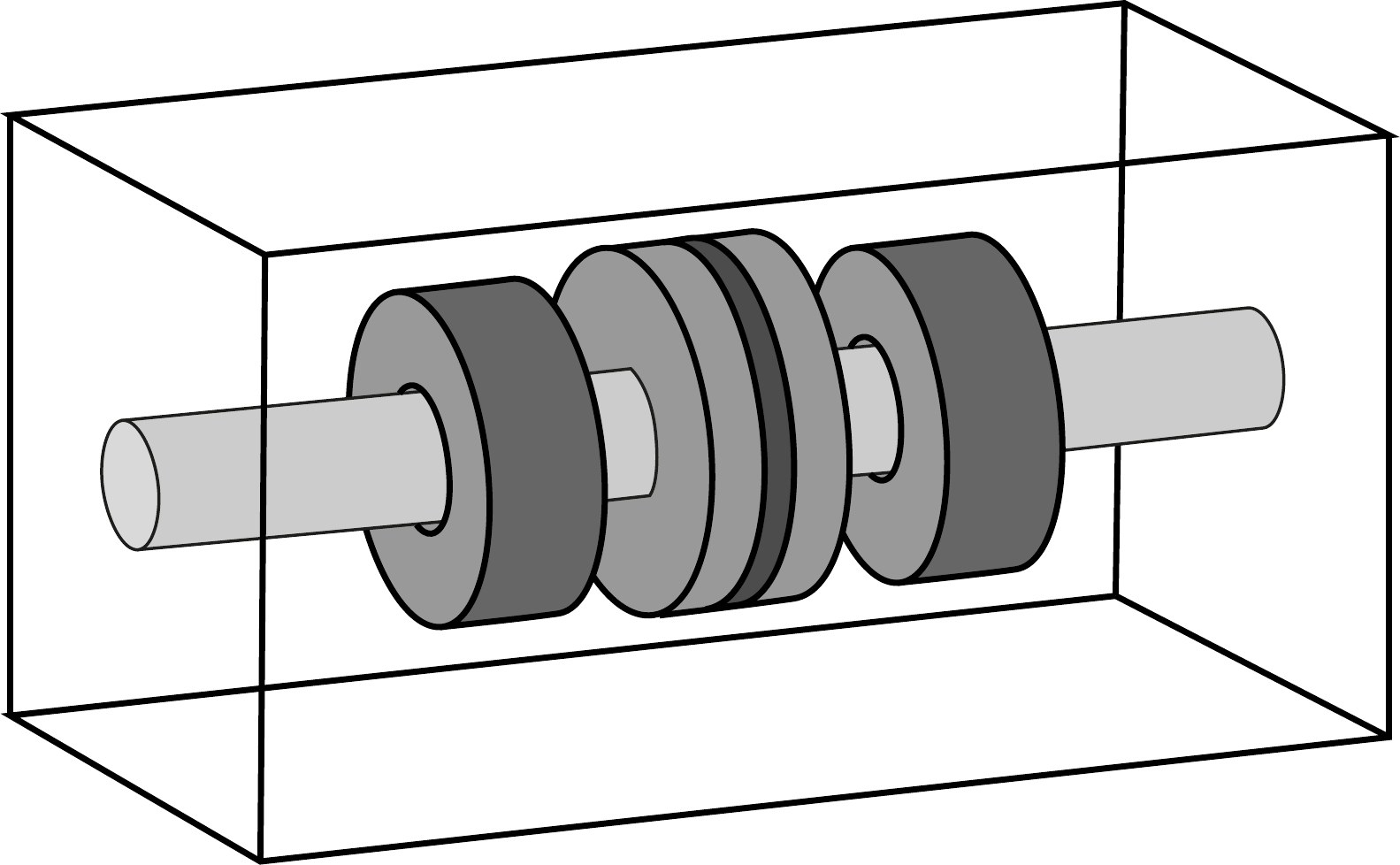}
  \caption{Benchmark example for Darwin model, originally proposed in \cite{Koch_2011aa}, see Benchmark \ref{bench:darwin}.}
	\label{fig:darwin}
\end{figure}

\begin{benchmark}\label{bench:darwin}
The Darwin benchmark example shown in Fig.~\ref{fig:darwin} is taken from \cite{Koch_2011aa}. It is axisymmetric and consists of conductive, 
dielectric and magnetic pieces. A solid conductor is connected to two plates which form a capacitor in combination with the dielectric material layer 
in between. Two solid ferrite rings surround the conductor. The problem is excited by potentials $\phi_1$  and $\phi_0$ at the ends of the conductor, 
the remaining boundaries are mbc. Geometry specifications and material data can be found in \cite{Koch_2011aa}. Koch et al. propose to regularise the 
model by an artificial conductivity.
\end{benchmark} 
\section{Conclusions}\label{sec:conclu}
This paper has discussed various formulation for low and high-frequency problems in computational electromagnetics. In contrast to electric circuit simulation, e.g. \cite{Tischendorf_1999aa}, most formulations are rather harmless, i.e., they lead to a systems of low DAE index ($\leq2$). More precisely, the only index-2 problem arises in the case of considering Maxwell's equations in A-$\phi$-potential formulation with a Coulomb gauge. It has been shown that this can be mitigated by a reformulation based on Lorenz' gauge. Obviously, when coupling various formulations with each other or with electric circuits, the situation becomes more complex.


\vspace{-1em} 

\begin{thebibliography}{100}
\providecommand{\url}[1]{{#1}}
\providecommand{\urlprefix}{URL }
\expandafter\ifx\csname urlstyle\endcsname\relax
  \providecommand{\doi}[1]{DOI~\discretionary{}{}{}#1}\else
  \providecommand{\doi}{DOI~\discretionary{}{}{}\begingroup
  \urlstyle{rm}\Url}\fi

\bibitem{Albanese_1991aa}
Albanese, R., Coccorese, E., Martone, R., Miano, G., Rubinacci, G.: On the
  numerical solution of the nonlinear three-dimensional eddy current problem.
\newblock IEEE Trans. Magn. \textbf{27}(5), 3990--3995 (1991).
\newblock \doi{10.1109/20.104976}

\bibitem{Alonso-Rodriguez_2003aa}
Alonso~Rodríguez, A., Raffetto, M.: Unique solvability for electromagnetic
  boundary value problems in the presence of partly lossy inhomogeneous
  anisotropic media and mixed boundary conditions.
\newblock M3AS \textbf{13}(04), 597--611 (2003).
\newblock \doi{10.1142/S0218202503002672}

\bibitem{Alonso-Rodriguez_2010aa}
Alonso~Rodríguez, A., Valli, A.: Eddy Current Approximation of {Maxwell}
  Equations, \emph{Modeling, Simulation and Applications}, vol.~4.
\newblock Springer, Heidelberg (2010).
\newblock \doi{10.1007/978-88-470-1506-7}

\bibitem{Alotto_2006aa}
Alotto, P., De~Cian, A., Molinari, G.: A time-domain 3-{D} full-{Maxwell}
  solver based on the cell method.
\newblock IEEE Trans. Magn. \textbf{42}(4), 799--802 (2006).
\newblock \doi{10.1109/tmag.2006.871381}

\bibitem{Assous_2018aa}
Assous, F., Ciarlet, P., Labrunie, S.: Mathematical foundations of
  computational electromagnetism.
\newblock Springer (2018)

\bibitem{Auserhofer_2009ab}
Außerhofer, S., Bíró, O., Preis, K.: Discontinuous {Galerkin} finite
  elements in time domain eddy-current problems.
\newblock IEEE Trans. Magn. \textbf{45}(3), 1300--1303 (2009)

\bibitem{Bartel_2011aa}
Bartel, A., Baumanns, S., Schöps, S.: Structural analysis of electrical
  circuits including magnetoquasistatic devices.
\newblock APNUM \textbf{61}, 1257--1270 (2011).
\newblock \doi{10.1016/j.apnum.2011.08.004}

\bibitem{Baumanns_2012ab}
Baumanns, S.: Coupled electromagnetic field/circuit simulation: Modeling and
  numerical analysis.
\newblock Ph.D. thesis, Universität zu Köln, Köln (2012)

\bibitem{Baumanns_2013aa}
Baumanns, S., Clemens, M., Schöps, S.: Structural aspects of regularized full
  {Maxwell} electrodynamic potential formulations using {FIT}.
\newblock In: G.~Manara (ed.) Proceedings of 2013 {URSI} International
  Symposium on Electromagnetic Theory ({EMTS}), pp. 1007--1010. IEEE (2013).

\bibitem{Baumanns_2010aa}
Baumanns, S., Selva~Soto, M., Tischendorf, C.: Consistent initialization for
  coupled circuit-device simulation.
\newblock In: J.~Roos, L.R.J. Costa (eds.) Scientific Computing in Electrical
  Engineering {SCEE} 2008, \emph{Mathematics in Industry}, vol.~14, pp.
  297--304. Springer, Berlin (2010).
\newblock \doi{10.1007/978-3-642-12294-1_38}

\bibitem{Becks_1992aa}
Becks, T., Wolff, I.: Analysis of 3-d metallization structures by a full-wave
  spectral-domain technique.
\newblock IEEE Trans. Microw. Theor. Tech. \textbf{40}(12), 2219--2227 (1992).
\newblock \doi{10.1109/22.179883}

\bibitem{Bedrosian_1993aa}
Bedrosian, G.: A new method for coupling finite element field solutions with
  external circuits and kinematics.
\newblock IEEE Trans. Magn. \textbf{29}(2), 1664--1668 (1993).
\newblock \doi{10.1109/20.250726}

\bibitem{Boffi_2010aa}
Boffi, D.: Finite element approximation of eigenvalue problems.
\newblock Acta. Num. \textbf{19}, 1--120 (2010).
\newblock \doi{10.1017/S0962492910000012}

\bibitem{Bondeson_2005aa}
Bondeson, A., Rylander, T., Ingelström, P.: Computational Electromagnetics.
\newblock Texts in Applied Mathematics. Springer Berlin Heidelberg (2005).
\newblock \doi{10.1007/b136922}

\bibitem{Bossavit_1988aa}
Bossavit, A.: Whitney forms: a class of finite elements for three-dimensional
  computations in electromagnetism.
\newblock IEE Proceedings \textbf{135}(8), 493--500 (1988).
\newblock \doi{10.1049/ip-a-1:19880077}

\bibitem{Bossavit_1991aa}
Bossavit, A.: Differential geometry for the student of numerical methods in
  electromagnetism.
\newblock Tech. rep., Électricité de France (1991).

\bibitem{Bossavit_1998aa}
Bossavit, A.: Computational Electromagnetism: Variational Formulations,
  Complementarity, Edge Elements.
\newblock Academic Press, San Diego (1998).

\bibitem{Bossavit_1999ae}
Bossavit, A.: On the geometry of electromagnetism. (4): '{Maxwell}'s house'.
\newblock JSAEM \textbf{6}(4), 318--326 (1999)

\bibitem{Bossavit_2001ab}
Bossavit, A.: Stiff problems in eddy-current theory and the regularization of
  {Maxwell}'s equations.
\newblock IEEE Trans. Magn. \textbf{37}(5), 3542--3545 (2001).
\newblock \doi{0018–9464/01$10.00}

\bibitem{Bossavit_1999af}
Bossavit, A., Kettunen, L.: Yee-like schemes on a tetrahedral mesh, with
  diagonal lumping.
\newblock Int. J. Numer. Model. Electron. Network. Dev. Field \textbf{12}(1-2),
  129--142 (1999).
\newblock
  \doi{10.1002/(SICI)1099-1204(199901/04)12:1/2<129::AID-JNM327>3.0.CO;2-G}

\bibitem{Bossavit_2000ab}
Bossavit, A., Kettunen, L.: {Yee}-like schemes on staggered cellular grids: a
  synthesis between {FIT} and {FEM} approaches.
\newblock IEEE Trans. Magn. \textbf{36}(4), 861--867 (2000).
\newblock \doi{10.1109/20.877580}

\bibitem{Brenan_1995aa}
Brenan, K.E., Campbell, S.L., Petzold, L.R.: Numerical solution of
  initial-value problems in differential-algebraic equations.
\newblock SIAM, Philadelphia (1995)

\bibitem{Biro_1989aa}
Bíró, O., Preis, K.: On the use of the magnetic vector potential in the
  finite-element analysis of three-dimensional eddy currents.
\newblock IEEE Trans. Magn. \textbf{25}(4), 3145--3159 (1989).
\newblock \doi{10.1109/20.34388}

\bibitem{Biro_1990aa}
Bíró, O., Preis, K.: Finite element analysis of 3-d eddy currents.
\newblock IEEE Trans. Magn. \textbf{26}(2), 418--423 (1990).
\newblock \doi{10.1109/20.106343}

\bibitem{Biro_1995aa}
Bíró, O., Preis, K., Richter, K.R.: Various {FEM} formulations for the
  calculation of transient 3d eddy currents in nonlinear media.
\newblock IEEE Trans. Magn. \textbf{31}(3), 1307--1312 (1995).
\newblock \doi{10.1109/20.376269}

\bibitem{Carpenter_1980aa}
Carpenter, C.J.: Comparison of alternative formulations of 3-dimensional
  magnetic-field and eddy-current problems at power frequencies.
\newblock IEE Proceedings B Electric Power Applications \textbf{127}(5), 332
  (1980).
\newblock \doi{10.1049/ip-b:19800045}

\bibitem{Chen_2013aa}
Chen, Q., Schoenmaker, W., Chen, G., Jiang, L., Wong, N.: A numerically
  efficient formulation for time-domain electromagnetic-semiconductor
  cosimulation for fast-transient systems.
\newblock IEEE Trans. Comput. Aided. Des. Integrated Circ. Syst.
  \textbf{32}(5), 802--806 (2013).
\newblock \doi{10.1109/TCAD.2012.2232709}

\bibitem{Clemens_2005aa}
Clemens, M.: Large systems of equations in a discrete electromagnetism:
  formulations and numerical algorithms.
\newblock IEE. Proc. Sci. Meas. Tech. \textbf{152}(2), 50--72 (2005).
\newblock \doi{10.1049/ip-smt:20050849}

\bibitem{Clemens_2011aa}
Clemens, M., Schöps, S., De~Gersem, H., Bartel, A.: Decomposition and
  regularization of nonlinear anisotropic curl-curl {DAE}s.
\newblock COMPEL \textbf{30}(6), 1701--1714 (2011).
\newblock \doi{10.1108/03321641111168039}

\bibitem{Clemens_1999ac}
Clemens, M., Weiland, T.: Transient eddy-current calculation with the
  {FI}-method.
\newblock IEEE Trans. Magn. \textbf{35}(3), 1163--1166 (1999).
\newblock \doi{10.1109/20.767155}

\bibitem{Clemens_2002aa}
Clemens, M., Weiland, T.: Regularization of eddy-current formulations using
  discrete grad-div operators.
\newblock IEEE Trans. Magn. \textbf{38}(2), 569--572 (2002).
\newblock \doi{10.1109/20.996149}

\bibitem{Clemens_2003aa}
Clemens, M., Wilke, M., Weiland, T.: Linear-implicit time-integration schemes
  for error-controlled transient nonlinear magnetic field simulations.
\newblock IEEE Trans. Magn. \textbf{39}(3), 1175--1178 (2003).
\newblock \doi{10.1109/TMAG.2003.810221}

\bibitem{Cortes-Garcia_2018ab}
Cortes~Garcia, I., De~Gersem, H., Schöps, S.: A structural analysis of
  field/circuit coupled problems based on a generalised circuit element.
\newblock Submitted,  2018.\newblock arXiv: 1801.07081.

\bibitem{CST_2016aa}
{CST AG}: {CST} {STUDIO} {SUITE} 2016 (2016).
\newblock \urlprefix\url{https://www.cst.com}

\bibitem{De-Gersem_2001aa}
De~Gersem, H., Hameyer, K.: A finite element model for foil winding simulation.
\newblock IEEE Trans. Magn. \textbf{37}(5), 3472--3432 (2001).
\newblock \doi{10.1109/20.952629}

\bibitem{De-Gersem_2004ab}
De~Gersem, H., Hameyer, K., Weiland, T.: Field-circuit coupled models in
  electromagnetic simulation.
\newblock J. Comput. Appl. Math. \textbf{168}(1-2), 125--133 (2004).
\newblock \doi{10.1016/j.cam.2003.05.008}

\bibitem{De-Gersem_2004ac}
De~Gersem, H., Weiland, T.: Field-circuit coupling for time-harmonic models
  discretized by the finite integration technique.
\newblock IEEE Trans. Magn. \textbf{40}(2), 1334--1337 (2004).
\newblock \doi{10.1109/TMAG.2004.824536}

\bibitem{Deschamps_1981aa}
Deschamps, G.A.: Electromagnetics and differential forms.
\newblock Proc. IEEE \textbf{69}(6), 676--696 (1981).
\newblock \doi{dx.doi.org/10.1109/PROC.1981.12048}

\bibitem{Dirks_1996aa}
Dirks, H.K.: Quasi-stationary fields for microelectronic applications.
\newblock Electr. Eng. \textbf{79}(2), 145--155 (1996).
\newblock \doi{10.1007/BF01232924}

\bibitem{Dutine_2017ae}
Dutiné, J.S., Richter, C., Jörgens, C., Schöps, S., Clemens, M.: Explicit
  time integration techniques for electro- and magneto-quasistatic field
  simulations.
\newblock In: R.D. Graglia (ed.) Proceedings of the International Conference on
  Electromagnetics in Advanced Applications ({ICEAA}) 2017. IEEE (2017).
\newblock \doi{10.1109/ICEAA.2017.8065562}

\bibitem{Dyck_2004aa}
Dyck, D.N., Webb, J.P.: Solenoidal current flows for filamentary conductors.
\newblock IEEE Trans. Magn. \textbf{40}(2), 810--813 (2004).
\newblock \doi{10.1109/TMAG.2004.824594}

\bibitem{Eller_2017aa}
Eller, M., Reitzinger, S., Schöps, S., Zaglmayr, S.: A symmetric low-frequency
  stable broadband {Maxwell} formulation for industrial applications.
\newblock SIAM J. Sci. Comput. \textbf{39}(4), B703--B731 (2017).
\newblock \doi{10.1137/16M1077817}

\bibitem{Estevez-Schwarz_1999ab}
Estévez~Schwarz, D.: Consistent initialization of differential-algebraic
  equations in circuit simulation.
\newblock Tech. Rep. 99-5, Humboldt Universität Berlin, Berlin (1999)

\bibitem{Griffiths_1999aa}
Griffiths, D.F.: Introduction to Electrodynamics.
\newblock Prentice-Hall, New Jersey (1999)

\bibitem{Godel_2010ab}
Gödel, N., Schomann, S., Warburton, T., Clemens, M.: {GPU} accelerated
  {Adams}–{Bashforth} multirate discontinuous {Galerkin} {FEM} simulation of
  high-frequency electromagnetic fields.
\newblock IEEE Trans. Magn. \textbf{46}(8), 2735--2738 (2010)

\bibitem{Hahne_1992aa}
Hahne, P., Weiland, T.: 3d eddy current computation in the frequency domain
  regarding the displacement current.
\newblock IEEE Trans. Magn. \textbf{28}(2), 1801--1804 (1992).
\newblock \doi{10.1109/20.124056}

\bibitem{Hairer_1996aa}
Hairer, E., Nørsett, S.P., Wanner, G.: Solving Ordinary Differential Equations
  {II}: Stiff and Differential-Algebraic Problems, 2 edn.
\newblock Springer Series in Computational Mathematics. Springer, Berlin (2002)

\bibitem{Harrington_1993aa}
Harrington, R.F.: Field Computation by Moment Methods.
\newblock Wiley-IEEE Press (1993)

\bibitem{Haus_1989aa}
Haus, H.A., Melcher, J.R.: Electromagnetic Fields and Energy.
\newblock Prentice-Hall (1989).

\bibitem{Heaviside_1891aa}
Heaviside, O.: On the forces, stresses, and fluxes of energy in the
  electromagnetic field.
\newblock Royal Society of London Proceedings Series I \textbf{50}, 126--129
  (1891)

\bibitem{Hehl_2003aa}
Hehl, F.W., Obukhov, Y.N.: Foundations of Classical Electrodynamics -- Charge,
  Flux, and Metric.
\newblock Progress in Mathematical Physics. Birkhäuser, Basel (2003)

\bibitem{Heise_1994aa}
Heise, B.: Analysis of a fully discrete finite element method for a nonlinear
  magnetic field problem.
\newblock SIAM J. Numer. Anal. \textbf{31}(3), 745--759 (1994).

\bibitem{Hesthaven_2008aa}
Hesthaven, J.S., Warburton, T.: Nodal discontinuous {Galerkin} methods:
  algorithms, analysis, and applications.
\newblock Texts in applied mathematics. Springer, New York and Berlin (2008)

\bibitem{Jackson_1998aa}
Jackson, J.D.: Classical Electrodynamics, 3rd edn.
\newblock Wiley and Sons, New York (1998)

\bibitem{Jochum_2015aa}
Jochum, M.T., Farle, O., Dyczij-Edlinger, R.: A new low-frequency stable
  potential formulation for the finite-element simulation of electromagnetic
  fields.
\newblock IEEE Trans. Magn. \textbf{51}(3), 7402,304 (2015).
\newblock \doi{10.1109/TMAG.2014.2360080}

\bibitem{Kameari_1990aa}
Kameari, A.: Calculation of transient {3D} eddy-current using edge elements.
\newblock IEEE Trans. Magn. \textbf{26}(5), 466–--469 (1990).
\newblock \doi{10.1109/20.106354}

\bibitem{Kerler-Back_2017aa}
Kerler-Back, J., Stykel, T.: Model reduction for linear and nonlinear
  magneto-quasistatic equations.
\newblock Int. J. Numer. Meth. Eng. \textbf{111}(13), 1274--1299 (2017).
\newblock \doi{10.1002/nme.5507}

\bibitem{Koch_2010aa}
Koch, S., Weiland, T.: Time domain methods for slowly varying fields.
\newblock In: {URSI} International Symposium on Electromagnetic Theory ({EMTS}
  2010), pp. 291--294 (2010).
\newblock \doi{10.1109/URSI-EMTS.2010.5636991}

\bibitem{Koch_2011aa}
Koch, S., Weiland, T.: Different types of quasistationary formulations for time
  domain simulations.
\newblock Radio Science \textbf{46}(5) (2011).
\newblock \doi{10.1029/2010RS004637}

\bibitem{Lamour_2013aa}
Lamour, R., März, R., Tischendorf, C.: Differential-Algebraic Equations: A
  Projector Based Analysis.
\newblock Differential-Algebraic Equations Forum. Springer, Heidelberg (2013).
\newblock \doi{10.1007/978-3-642-27555-5}

\bibitem{Larsson_2007aa}
Larsson, J.: Electromagnetics from a quasistatic perspective.
\newblock Am. J. Phys. \textbf{75}(3), 230--239 (2007).
\newblock \doi{10.1119/1.2397095}

\bibitem{Manges_1997aa}
Manges, J.B., Cendes, Z.J.: Tree-cotree decompositions for first-order complete
  tangential vector finite elements.
\newblock Int. J. Numer. Meth. Eng. \textbf{40}(9), 1667--1685 (1997).
\newblock
  \doi{10.1002/(SICI)1097-0207(19970515)40:9<1667::AID-NME133>3.0.CO;2-9}

\bibitem{Maxwell_1864aa}
Maxwell, J.C.: A dynamical theory of the electromagnetic field.
\newblock Phil. Trans. R. Soc. London \textbf{CLV}, 459--512 (1864)

\bibitem{Mehrmann_2015aa}
Mehrmann, V.: Index Concepts for Differential-Algebraic Equations, pp.
  676--681.
\newblock Springer Berlin Heidelberg, Berlin, Heidelberg (2015).
\newblock \doi{10.1007/978-3-540-70529-1_120}

\bibitem{Merkel_2017ac}
Merkel, M., Niyonzima, I., Schöps, S.: Paraexp using leapfrog as integrator
  for high-frequency electromagnetic simulations.
\newblock Radio Science \textbf{52}(12), 1558--1569 (2017).
\newblock \doi{10.1002/2017RS006357}

\bibitem{Monk_2003aa}
Monk, P.: Finite Element Methods for {Maxwell}'s Equations.
\newblock Oxford University Press, Oxford (2003)

\bibitem{Monk_1994aa}
Monk, P., Süli, E.: A convergence analysis of {Yee}'s scheme on nonuniform
  grids.
\newblock SIAM J. Numer. Anal. \textbf{31}(2), 393--412 (1994).
\newblock \doi{10.1137/0731021}

\bibitem{Munteanu_2002aa}
Munteanu, I.: Tree-cotree condensation properties.
\newblock ICS Newsletter (International Compumag Society) \textbf{9}, 10--14
  (2002).
\newblock
  \urlprefix\url{http://www.compumag.org/jsite/images/stories/newsletter/ICS-02-09-1-Munteanu.pdf}

\bibitem{Marz_2003aa}
März, R.: Differential algebraic systems with properly stated leading term and
  {MNA} equations.
\newblock In: K.~Anstreich, R.~Bulirsch, A.~Gilg, P.~Rentrop (eds.) Modelling,
  Simulation and Optimization of Integrated Circuits, pp. 135--151.
  Birkhäuser, Berlin (2003)

\bibitem{Nagel_1973aa}
Nagel, L.W., Pederson, D.O.: Simulation program with integrated circuit
  emphasis.
\newblock Tech. rep., University of California, Berkeley, Electronics Research
  Laboratory, ERL-M382 (1973)

\bibitem{Nicolet_1996aa}
Nicolet, A., Delincé, F.: Implicit {Runge}-{Kutta} methods for transient
  magnetic field computation.
\newblock IEEE Trans. Magn. \textbf{32}(3), 1405--1408 (1996).
\newblock \doi{0.1109/20.497510}

\bibitem{Nedelec_1980aa}
Nédélec, J.C.: Mixed finite elements in $r^3$.
\newblock Numerische Mathematik \textbf{35}(3), 315--341 (1980).
\newblock \doi{10.1007/BF01396415}

\bibitem{Ostrowski_2012aa}
Ostrowski, J., Hiptmair, R., Krämer, F., Smajic, J., Steinmetz, T.: Transient
  full {Maxwell} computation of slow processes.
\newblock In: B.~Michielsen, J.R. Poirier (eds.) Scientific Computing in
  Electrical Engineering {SCEE} 2010, \emph{Mathematics in Industry}, vol.~16,
  pp. 87--95. Springer, Berlin (2012).
\newblock \doi{10.1007/978-3-642-22453-9_10}

\bibitem{Ouedraogo_2017aa}
Ouédraogo, Y., Gjonaj, E., Weiland, T., De~Gersem, H., Steinhausen, C.,
  Lamanna, G., Weigand, B., Preusche, A., Dreizler, A., Schremb, M.:
  Electrohydrodynamic simulation of electrically controlled droplet generation.
\newblock International Journal of Heat and Fluid Flow \textbf{64}, 120--128
  (2017)

\bibitem{Petzold_1982aa}
Petzold, L.R.: Differential/algebraic equations are not {ODE}'s.
\newblock SIAM J. Sci. Stat. Comput. \textbf{3}(3), 367--384 (1982).
\newblock \doi{10.1137/0903023}

\bibitem{Rapetti_2014aa}
Rapetti, F., Rousseaux, G.: On quasi-static models hidden in {Maxwell}'s
  equations.
\newblock APNUM \textbf{79}, 92–--106 (2014).
\newblock \doi{10.1016/j.apnum.2012.11.007}

\bibitem{Rautio_2014aa}
Rautio, J.C.: The long road to {Maxwell}'s equations.
\newblock IEEE Spectrum \textbf{51}(12), 36--56 (2014).
\newblock \doi{10.1109/MSPEC.2014.6964925}

\bibitem{Raviart_1977aa}
Raviart, P.A., Thomas, J.M.: Primal hybrid finite element methods for 2nd order
  elliptic equations.
\newblock Math. Comput. \textbf{31}(138), 391--413 (1977)

\bibitem{Ruehli_1974aa}
Ruehli, A.E.: Equivalent circuit models for three-dimensional multiconductor
  systems.
\newblock IEEE Trans. Microw. Theor. Tech. \textbf{22}(3), 216--221 (1974)

\bibitem{Ruehli_2015aa}
Ruehli, A.E., Antonini, G., Jiang, L.: The Partial Element Equivalent Circuit
  Method for Electro-Magnetic and Circuit Problems.
\newblock Wiley and Sons, Hoboken, New Jersey (2015)

\bibitem{Schilders_2005aa}
Schilders, W.H.A., Ciarlet, P., ter Maten, E.J.W. (eds.): Handbook of Numerical
  Analysis. Numerical Methods in Electromagnetics, \emph{Handbook of Numerical
  Analysis}, vol.~13.
\newblock North-Holland (2005)

\bibitem{Schmidt_2008aa}
Schmidt, K., Sterz, O., Hiptmair, R.: Estimating the eddy-current modeling
  error.
\newblock IEEE Trans. Magn. \textbf{44}(6), 686--689 (2008).
\newblock \doi{10.1109/TMAG.2008.915834}

\bibitem{Schoenmaker_2017aa}
Schoenmaker, W.: Computational Electrodynamics.
\newblock River Publishers Series in Electronic Materials and Devices. River
  Publishers (2017)

\bibitem{Schuhmacher_2017aa}
Schuhmacher, S., Klaedtke, A., Keller, C., Ackermann, W., De~Gersem, H.:
  Optimizing the inductance cancellation behavior in an {EMI} filter design
  with the help of a sensitivity analysis.
\newblock In: {EMC} Europe. Angers, France (2017)

\bibitem{Schuhmann_2001aa}
Schuhmann, R., Weiland, T.: Conservation of discrete energy and related laws in
  the finite integration technique.
\newblock PIER \textbf{32}, 301--316 (2001).
\newblock \doi{10.2528/PIER00080112}.

\bibitem{Schops_2011ac}
Schöps, S.: Multiscale modeling and multirate time-integration of
  field/circuit coupled problems.
\newblock Dissertation, Bergische Universität Wuppertal \& Katholieke
  Universiteit Leuven, Düsseldorf (2011).
\newblock {VDI} Verlag. Fortschritt-Berichte {VDI}, Reihe 21

\bibitem{Schops_2012aa}
Schöps, S., Bartel, A., Clemens, M.: Higher order half-explicit time
  integration of eddy current problems using domain substructuring.
\newblock IEEE Trans. Magn. \textbf{48}(2), 623--626 (2012).
\newblock \doi{10.1109/TMAG.2011.2172780}

\bibitem{Schops_2013aa}
Schöps, S., De~Gersem, H., Weiland, T.: Winding functions in transient
  magnetoquasistatic field-circuit coupled simulations.
\newblock COMPEL \textbf{32}(6), 2063--2083 (2013).
\newblock \doi{10.1108/COMPEL-01-2013-0004}

\bibitem{Steinmetz_2011aa}
Steinmetz, T., Kurz, S., Clemens, M.: Domains of validity of quasistatic and
  quasistationary field approximations.
\newblock COMPEL \textbf{30}(4), 1237--1247 (2011).
\newblock \doi{10.1108/03321641111133154}

\bibitem{Taflove_1995aa}
Taflove, A.: Computational Electrodynamics: The Finite-Difference
  Time-Domain-Method.
\newblock Artech House, Dedham, MA (1995)

\bibitem{Taflove_2007aa}
Taflove, A.: A perspective on the 40-year history of {FDTD} computational
  electrodynamics.
\newblock Appl. Comput. Electromagn. \textbf{22}(1), 1--21 (2007)

\bibitem{Tischendorf_1999aa}
Tischendorf, C.: Topological index calculation of {DAE}s in circuit simulation.
\newblock Tech. Rep. 3-4, Humboldt Universität Berlin, Berlin (1999)

\bibitem{Tonti_1975aa}
Tonti, E.: On the formal structure of physical theories.
\newblock Tech. rep., Politecnico di Milano, Milano, Italy (1975)

\bibitem{Tsukerman_2002aa}
Tsukerman, I.A.: Finite element differential-algebraic systems for eddy current
  problems.
\newblock Numer. Algorithm. \textbf{31}(1), 319--335 (2002).
\newblock \doi{10.1023/A:1021112107163}

\bibitem{Webb_1993aa}
Webb, J.P., Forghani, B.: The low-frequency performance of $h-\phi$ and
  $t-\omega$ methods using edge elements for 3d eddy current problems.
\newblock IEEE Trans. Magn. \textbf{29}(6), 2461--2463 (1993).
\newblock \doi{10.1109/20.280983}

\bibitem{Weeks_1973aa}
Weeks, W., Jimenez, A., Mahoney, G., Mehta, D., Qassemzadeh, H., Scott, T.:
  Algorithms for {ASTAP} -- a network-analysis program.
\newblock IEEE Trans. Circ. Theor. \textbf{20}(6), 628--634 (1973).
\newblock \doi{10.1109/TCT.1973.1083755}

\bibitem{Weiland_1977aa}
Weiland, T.: A discretization method for the solution of {Maxwell}'s equations
  for six-component fields.
\newblock Int. J. Electron. Commun. (AEU) \textbf{31}, 116--120 (1977)

\bibitem{Weiland_1985aa}
Weiland, T.: On the unique numerical solution of {Maxwell}ian eigenvalue
  problems in three dimensions.
\newblock Particle Accelerators \textbf{17}(227--242) (1985)

\bibitem{Weiland_1996aa}
Weiland, T.: Time domain electromagnetic field computation with finite
  difference methods.
\newblock Int. J. Numer. Model. Electron. Network. Dev. Field \textbf{9}(4),
  295--319 (1996).
\newblock \doi{10.1002/(SICI)1099-1204(199607)9:4<295::AID-JNM240>3.0.CO;2-8}

\bibitem{Yee_1966aa}
Yee, K.S.: Numerical solution of initial boundary value problems involving
  {Maxwell}'s equations in isotropic media.
\newblock IEEE Trans. Antenn. Propag. \textbf{14}(3), 302--307 (1966).
\newblock \doi{10.1109/TAP.1966.1138693}

\end{thebibliography}
\end{document}